\documentclass[11pt,twoside,letterpaper]{amsart}

\usepackage{microtype}
\usepackage[french,english]{babel}  
\usepackage[T1]{fontenc}
\usepackage{xspace}
\usepackage{lscape}  
\usepackage{times}
\usepackage{dcolumn,amsmath}
\usepackage{amsfonts}
\usepackage{flafter}
\usepackage{graphicx}
\usepackage{caption}
\usepackage{textcomp}
\usepackage{import}
\usepackage{verbatim}
\usepackage{setspace}
\usepackage{amssymb}
\usepackage{esint}
\usepackage{verbatim}
\usepackage[all]{xy}
\usepackage{amsthm}
\makeatletter
\numberwithin{equation}{section} 
\numberwithin{figure}{section} 
\theoremstyle{plain}
\newtheorem{thm}{Theorem}[section]
\theoremstyle{definition}
\newtheorem{defn}[thm]{Definition}
\theoremstyle{plain}
\newtheorem{cor}[thm]{Corollary}
\theoremstyle{definition}
\newtheorem{example}[thm]{Example}
\theoremstyle{remark}
\newtheorem{rem}[thm]{Remark}
\theoremstyle{plain}
\newtheorem{prop}[thm]{Proposition}
\theoremstyle{plain}
\newtheorem{lem}[thm]{Lemma}
\theoremstyle{plain}
\newtheorem*{thm*}{Theorem}
\newtheorem*{lem*}{Lemma}
\newtheorem*{prop*}{Proposition}
\newtheorem*{cor*}{Corollary}
\newtheorem*{rem*}{Remark}
\newtheorem*{defn*}{Definition}
\newcommand{\R}{\mathbb{R}}
\author{Antonio Rieser}
\title[Lagrangian blow-ups, blow-downs, and applications]{Lagrangian blow-ups, blow-downs, and applications to real packing}
\begin{document}
\selectlanguage{english}
\setstretch{1}
\setcounter{page}{1}
\begin{abstract}Given a symplectic manifold $(M,\omega)$ and a
  Lagrangian submanifold $L$, we construct versions of the symplectic
  blow-up and blow-down which are defined relative to
  $L$. We further show that if $M$ admits an anti-symplectic
  involution $\phi$, i.e. a diffeomorphism such that $\phi^2=Id$ and
  $\phi^*\omega=-\omega$, and we blow-up an appropriately symmetric
  embedding of symplectic balls, then there exists an antisymplectic
  involution on the blow-up $\tilde{M}$ as well. We then derive a
  homological condition for real Lagrangian surfaces
  $L=\text{Fix}(\phi)$ which determines when the topology of $L$
  changes after a blow down, and we use these constructions to
  study the relative packing numbers and packing stability for
  real symplectic four manifolds which are non-Seiberg-Witten simple.
\end{abstract}
\maketitle
\tableofcontents
\section{Introduction}
\label{ch:Introduction}
The blow-up and blow-down constructions are important techniques in
complex geometry, leading to methods for resolving singularities as
well as classification schemes based on birational equivalence.  In
the symplectic category, the notion of blowing up a point or
submanifold has also been defined and studied from various points of
view, as the in papers by Guillemin and Sternberg
\cite{Guillemin_Sternberg_1989}, Lerman \cite{Lerman_1995}, and McDuff
and Polterovich \cite{McDuff_Polterovich_1994}.  When combined with
the theory of $J$-holomorphic curves, the blow-up and blow-down have
yielded a great deal of information on symplectic manifolds, notably
in packing problems \cite{Biran_1997,McDuff_Polterovich_1994}, in the
classification of rational and ruled symplectic 4-manifolds
\cite{McDuff_1990,Lalonde_McDuff_1994,Lalonde_McDuff_1996}, and in the
study of the topology of the space of symplectic embeddings of balls,
as, for example, in \cite{Lalonde_Pinsonnault_2004, Pinsonnault_2008,
  Anjos_Lalonde_Pinsonnault_2009}.  In this note, we study relative
and real versions of the symplectic blow-up and blow-down, in order to
apply them to questions regarding the topology of Lagrangian
submanifolds.  The relative blow-up takes the pair $(M,L)$ and a set
of relative ball embeddings $\psi:\coprod_{j=1}^k
(B_j^{2n}(1+2\epsilon),\lambda_j^2\omega_0,B_{\mathbb{R},j}(1+2\epsilon))\to (M,\omega,L)$ and obtains
another pair $(\tilde{M},\tilde{L})$, and a symplectic form $\tilde{\omega}$, in which the balls have been
replaced by copies of the tautological disk bundle over
$\mathbb{C}P^{n-1}$, and $\tilde{L}$ is Lagrangian in $(\tilde{M},\tilde{\omega})$. The
blow-down is the reverse procedure. The real blow-up and blow-down are
similar constructions which also respect a so-called real structure on the
manifolds.

As a first application, we study the packing problem in real
symplectic manifolds.  The relative and mixed packing problems were first
introduced by Barraud and Cornea in \cite{Barraud_Cornea_2007}, and upper
bounds for the relative embedding of one ball on the Clifford torus in
$\mathbb{C}P^{n}$ was given by Biran and Cornea in \cite{Biran_Cornea_2009} using Pearl
Homology.  Buhovsky \cite{Buhovsky_2010} further showed that the upper bound
given for the Clifford torus is sharp.
Schlenk, in \cite{Schlenk_2005}, directly constructed relative
packings of $k\leq 6$ balls in $(\mathbb{C}P^2,\mathbb{R}P^2)$ through
a detailed analysis of the moment map. A related construction for
packing $\mathbb{C}P^2$ for $k=7,8$ balls was done by Wieck in
\cite{Wieck_thesis}. It is not immediately clear if Wieck's techniques
can be made adapted to the relative setting, since the symplectic
tunnelling technique that he introduces does not produce relative
embeddings. In
Section \ref{sec:Applications}, we construct relative embeddings using
$J$-holomorphic techniques, following the general line of argument in
\cite{McDuff_Polterovich_1994} and \cite{Biran_1997}. Our results
extend those of McDuff and Polterovich\cite{McDuff_Polterovich_1994}
and Biran\cite{Biran_1997} to the real setting. Our packing method depends on the presence of a
real structure $\phi$ for which $L=Fix(\phi)$, and because of this, we
do not recover the lower bounds on the Clifford Torus considered by
Buhovsky\cite{Buhovsky_2010}. 

The results in this paper form a part of my PhD thesis, carried out at the Universit\'e de Montr\'eal
under the supervision of Octav Cornea and Fran\c{c}ois Lalonde.

\subsection{Setting and Notation} \label{subsec:Setting and Notation}

We now give several definitions and set notation for all that follows.

\begin{defn} Let $(M^{2n},\omega)$ be a symplectic manifold. We say
  that a submanifold $L$ is \emph{Lagrangian} if $\mbox{dim } L=n$ and
  $\omega|_{TL}=0$.
\end{defn}

\begin{defn}\label{defn:Notation}
  \begin{enumerate}
  \item We let $\mathcal{L}^n$ denote the tautological complex line
    bundle over $\mathbb{C}P^{n-1}$, and let $\mathcal{R}^n$ be the
    real tautological line bundle over $\mathbb{R}P^{n-1}$, i.e.
    $\mathcal{L}^n = \{ (z,l) \in \mathbb{C}^{n} \times \mathbb{C}P^{n-1} | z \in
    l \}$ and
    $\mathcal{R}^n = \{ (x,l) \in \mathbb{R}^{n} \times \mathbb{R}P^{n-1} | x \in
    l \}$. We will suppress the dimension $n$ when it is clear from the
    context.
  \item $\pi:\mathcal{L}\to\mathbb{C}^{n}$ and
    $\theta:\mathcal{L}\to\mathbb{C}P^{n-1}$ denote the canonical
    projections. \label{def-item-pi blow down map}
  \item $\mathcal{L}(r)$ and $\mathcal{R}(r)$ denote the canonical open
    disk bundles over $\mathbb{C}P^{n-1}$ and $\mathbb{R}P^{n-1},$
    respectively, of radius r. Abusing notation, we will use $\mathcal{L}(0)$ 
    and $\mathcal{R}(0)$ to refer to the zero section of these bundles.

  \item For each $\kappa,\lambda>0$, we define a closed two-form
    $\rho(\kappa,\lambda)$ on $\mathcal{L}(r)$ by
    \begin{equation*}
      \rho(\kappa,\lambda)=\kappa^{2}\pi^{*}\omega_{0}+\lambda^{2}\theta^{*}\sigma,
    \end{equation*}
where $\omega_{0}$ is the standard form on
$\mathbb{C}^{n}$, and $\sigma$ is the standard K\"ahler form on
$\mathbb{C}P^{n-1}$, normalized so that $\int_{\mathbb{C}P^{1}}\sigma=\pi$.

  \item Let $\tilde{c}:\mathcal{L}\to\mathcal{L}$ be the map
    $\tilde{c}(z,l)=(\bar{z},\bar{l})$, i.e. the restriction to
    $\mathcal{L}$ of the complex conjugation map on
    $\mathbb{C}^{n}\times\mathbb{C}P^{n-1}$.
  \end{enumerate}
\end{defn}

In addition, the manifolds we treat in our applications will have an
additional structure, as defined by
\begin{defn}
  Let $(M,\omega)$ be a symplectic manifold. A 
\emph{symplectic anti-involution}, 
or 
\emph{real structure}, 
is a diffeomorphism
  $\phi:M\to M$ such that $\phi^{2}=Id$ and
  $\phi^{*}\omega=-\omega$. We call a symplectic manifold equipped
  with a real structure a 
\emph{real symplectic manifold,} or
  simply a \emph{real manifold}, if the symplectic form is
  understood.
\end{defn}

\begin{rem}\label{rem:Fix phi Lagrangian}Note that $Fix(\phi)$ is Lagrangian.
\end{rem}
\begin{defn}
  \label{defn:real embedding}Let $(M,\omega,\phi)$ and
  $(M^{'},\omega^{'},\phi^{'})$ be real symplectic manifolds. We say
  that an embedding $\psi:(M^{'},\omega^{'},\phi^{'})\to(M,\omega,\phi)$ is a \emph{real symplectic embedding} 
if
  $\phi\circ\psi=\psi\circ\phi^{'}$ and
  $\psi^{*}\omega=\omega^{'}$. \end{defn}
\begin{lem}
  Let $(M,\omega_0)$ be a symplectic manifold, and let
  $(N,\omega_1,\phi)$ be a real symplectic manifold with symplectic
  form $\omega_1$ and real structure $\phi$. Suppose that there exists
  a symplectic embedding $\psi:(M,\omega_0)\to (N,\omega_1)$ such that
  $\mbox{Im} (\phi\circ\psi)=\mbox{Im} (\psi)$.  Then there exists an
  anti-symplectic involution $c$ on $M$ such that $\phi\circ
  \psi=\psi\circ c.$
\end{lem}
\begin{proof}
  Define $c:=\psi^{-1}\circ \phi\circ\psi$.  Then $\phi\circ
  \psi=\psi\circ c$ and $c^{*}\omega_0=\psi^{*} \phi^*
  (\psi^{-1})^{*}\psi^{*}\omega_1=-\omega_0$, so $\phi$ is an
  anti-symplectic involution on $M$.
\end{proof}

With the notation in Definition \ref{defn:Notation}, we have
\begin{cor}
  $\tilde{c}^{*}\rho(\kappa,\lambda)=-\rho(\kappa,\lambda),$ and
  $\mathcal{R}=\mbox{Fix}(\tilde{c})$.\end{cor}
\begin{proof}
  Let $c:\mathbb{C}^{n}\to\mathbb{C}^{n}$ and
  $\bar{c}:\mathbb{C}P^{n-1}\to\mathbb{C}P^{n-1}$ denote complex
  conjugation on $\mathbb{C}^{n}$ and $\mathbb{C}P^{n-1}$,
  respectively. Then by the definition of $\tilde{c},$
  $\tilde{c}(z,l)=(c(x),\bar{c}(l))$.  Since
  $\mathbb{R}^{n}=\mbox{Fix}(c)$ and
  $\mathbb{R}P^{n-1}=\mbox{Fix}(\bar{c})$,
  $\mathcal{R}=\mbox{Fix}(\tilde{c})$.

  Now let $(v_{0},w_{0}),(v_{1},w_{1})\in T_{(z,l)}\mathcal{L}\subset
  T_{z}\mathbb{C}^n\oplus T_{l}\mathbb{C}P^{n-1}$. Then
  \begin{align*}
    \tilde{c}^{*}\rho(\kappa,\lambda)((v_{0},w_{0}),(v_{1},w_{1})) =
&\tilde{c}^{*}\pi^{*}\kappa^{2}\omega_{0}((v_{0},w_{0}),(v_{1},w_{1})) +\\
    & \tilde{c}^{*}\theta^{*}\lambda^{2}\sigma((v_{0},w_{0}),(v_{1},w_{1}))\\
    =& -\kappa^2\omega_{0}(v_{0},v_{1})-\lambda^{2}\sigma(w_{0},w_{1})\\
    =& -\rho(\kappa,\lambda)((v_{0},w_{0}),(v_{1},w_{1})),\end{align*}
  which completes the proof.
\end{proof}

In order to put a symplectic form on the blow-up of a manifold $M$, we
will need to consider the relative embeddings of symplectic manifolds,
defined below.

\begin{defn} \label{defn:Relative embedding} 
  Let $(M,\omega,L)$ and
  $(M^{'},\omega^{'},L^{'})$ be symplectic manifolds with Lagrangians
  $L$ and $L^{'}$, respectively.  We say that a map
  $\psi:(M^{'},\omega^{'},L^{'})\to(M,\omega,L)$ is a \emph{relative
    symplectic embedding} when $\psi$ is a symplectic embedding,
  $\psi^{*}\omega=\omega^{'}$, and $\psi^{-1}(L)=L^{'}$.
\end{defn}
We will be primarily concerned with the following example.
\begin{example}
  Let $(M^{2n},\omega,L)$ be a symplectic manifold with Lagrangian
  $L$. Let $(B(\lambda),\omega_{0})$ be the
  ball of radius $\lambda$ in $\mathbb{C}^{n}$ with the standard symplectic structure
  $\omega_{0}$, and let $B_{\mathbb{R}}(\lambda)$ denote the
  ball of radius $\lambda$ in $\mathbb{R}^{n}\subset\mathbb{C}^{n}$. Then a symplectic
  embedding
  $\psi:(B^{2n}(\lambda),\omega_{0})\hookrightarrow(M^{2n},\omega)$
  is a relative symplectic embedding iff
  $\psi^{-1}(L)=B_{\mathbb{R}}(\lambda)$. \end{example}
\begin{rem}
  Note that in Definition \ref{defn:Relative embedding}, we have
  $\psi^{-1}(L)=L^{'}$, and not $\psi(L^{'})\subseteq L$. This is
  an important distinction, as shown by the following example. Let $C$
  denote an embedding of $S^1$ into $\mathbb{C}^1$, and let
\begin{equation*}
\Lambda:=\{\lambda\in\mathbb{R}|\exists \text{ a relative
    embedding }
  \psi:(B^{2}(1),\lambda^{2}\omega_{0},B_{\mathbb{R}}(1))
  \hookrightarrow (\mathbb{C}^{1},\omega_{0},C) \}.
\end{equation*} and
  $\Lambda_{\sup}:=\sup\Lambda$. Then for any $\lambda\in\Lambda$,
  $\lambda^{2}\pi\leq 2A$, where $A$ is the area inside
  $C\subset \mathbb{C}^2$. Therefore $\Lambda_{sup}\leq
  \sqrt{\frac{2A}{\pi}}.$ If, however, we only require that
  $\psi( B_{\mathbb{R}}(1))\subseteq C$, then $\Lambda$ is not
  bounded above.\end{rem}
\begin{defn}
  Let $\psi:\coprod_{i=1}^k
  (B_i(r),\omega_0,B_{\mathbb{R},i}(r))\hookrightarrow (M,\omega,L)$
  be a symplectic embedding, and let $\psi_i:=\psi|_{B_i}$. If $p$ of
  the $\psi_i$'s are relative embeddings, and for the other $q=k-p$ of
  the $\psi_i$'s, we have $Im(\psi_i) \cap L=\emptyset$, then we
  call $\psi$ a \emph{$(p,q)$-mixed embedding}.
\end{defn}

\subsection{Anti-Symplectic Involutions and Almost Complex
  Structures}
\label{subsubsec:Anti-Symplectic involutions and compatible almost
  complex structures}
Our constructions will use auxiliary almost complex structures which
satisfy certain additional properties. In this section, we give the
necessary definitions, and prove the existence of the complex
structures that we need.

\begin{defn}
  \label{defn:Tame acs}
  Let $(M,\omega)$ be a symplectic manifold. Then an almost complex
  structure $J$ \emph{tames $\omega$} or is $\omega$-\emph{tame} if
  $\omega(\cdot,J\cdot)>0$.
\end{defn}

\begin{defn}
  \label{defn:Compatible acs}
  Let $(M,\omega)$ be a symplectic manifold. Then an almost complex
  structure $J$ is \emph{compatible with $\omega$} or is
  $\omega$-\emph{compatible} if $J$ tames $\omega$, and if, in
  addition, $\omega(J\cdot,J\cdot)=\omega(\cdot,\cdot)$.
\end{defn}

\begin{defn}
  \label{defn:Relatively integrable}
  Let $(M,\omega)$ be a symplectic manifold, let $L\subset M$ be a
  Lagrangian submanifold, and let $p$ be a point in $L\subset M$. We
  say that $J$ is \emph{relatively integrable at p} if there is a
  holomorphic chart $U\subset M$, $\alpha:U\to \mathbb{C}^n$ centered
  at $p$ such that $\alpha^{-1}(\mathbb{R}^n)=U\cap L$.
\end{defn}

\begin{defn}
  \label{defn:Symmetrically integrable acs}
  Let $(M,\omega,\phi)$ be a real symplectic manifold with real
  structure $\phi$. Let $L$ denote $Fix(\phi)$, and let $p$ be a point
  in $L$. We say that $J$ is \emph{symmetrically integrable at p} if
  there is a holomorphic chart $U\subset M$, $\alpha:U\to
  \mathbb{C}^n$ centered at $p$ such that
  $\alpha\circ\phi=c\circ\alpha$.
\end{defn}

We first prove the existence of almost complex structures $J$ on a
real symplectic manifold $(M,\omega,\phi)$ which tame $\omega$ and
satisfy $J\phi_*=-\phi_* J$. Our discussion follows the methods in
Cannas da Silva \cite{Cannas_da_Silva_2001} and McDuff and Salamon
\cite{McDuff_Salamon_Intro_1998}.

\begin{defn} Given a symplectic form $\omega$ and an
  $\omega$-compatible almost complex structure $J$, we denote by
  $g_J:V\times V\to \mathbb{R}$ the bilinear form defined by
  \begin{equation}
    g_J(v,w)=\omega(v,Jw).
  \end{equation}
\end{defn}
\begin{lem}
  \label{lem:Sym acs vector space}
  Let $(V,\omega,\Phi)$ be a real symplectic vector space, i.e. a
  vector space $V$ with a closed, non-degenerate, skew-symmetric
  bilinear form $\omega$ and linear map $\Phi$ such that $\Phi^{2}=I$
  and $\Phi^*\omega=-\omega$. Let $\mathcal{J}_{\Phi}(V,\omega)$ be
  the space of $\omega$-compatible almost complex structures on $V$
  with $\Phi J=-J\Phi$, and let $\mathcal{M}et_{\Phi}(V)$ denote the
  space of positive definite bilinear forms $g$ such that
  $\Phi^{*}g=g$. Then there exists a continuous map
  $r:\mathcal{M}et_{\Phi}(V) \to \mathcal{J}_{\Phi}(V,\omega) $ such
  that $r(g_J)=J$.
\end{lem}

The proof follows \cite{Cannas_da_Silva_2001}.
\begin{proof}
  Let $g\in \mathcal{M}et_{\Phi}(V)$ and define the automorphism
  $A:V\to V$ by $\omega(v,w)=g(Av,w)$. Then $\omega(v,w)=-\omega(w,v)$
  implies that $g(Av,w)=-g(v,Aw)$, and therefore that $A^*=-A$. Let
  $A=QJ$ be the polar decomposition of $A$.  Then $Q$ is the unique
  square root of $A^*A$ which is $g$-self-adjoint and
  $g$-positive-definite.  We claim that $J_g:=Q^{-1}A$ is a complex
  structure compatible with $\omega$. First, note that $A$ commutes
  with $Q$,
  and therefore $J_g^2=Q^{-1}AQ^{-1}A=-Id$, so $J_g$ is an almost
  complex structure. To see that it is orthogonal, we have
  {\allowdisplaybreaks
  \begin{align*}
    \omega(J_gv,J_gw) & = g(AQ^{-1}Av,Q^{-1}Aw)\\
    & = g(AQ^{-1}Av,Q^{-1}AQ^{-1}A^*Q^{-1}Aw) \\
    & = g(Q^{-1}AQ^{-1}Av,AQ^{-1}A^*Q^{-1}Aw) \\
    & = g(-AQ^{-1}AQ^{-1}Av,Q^{-1}A^*Q^{-1}Aw) \\
    & = g(-AQ^{-1}AQ^{-1}Av,-Q^{-1}AQ^{-1}Aw) \\
    & = g(IAv,Iw) \\
    & = g(Av,w) = \omega(v,w).
  \end{align*} } Here, $I$ denotes the identity, and the second and second to
last equalities follow because $A^* = -A$ and $Q^{-1}AQ^{-1}A = -I$. Also, 
  $\omega(v,J_gv)=g(Av,Q^{-1}Av)=g(v,A^*Q^{-1}Av)=g(v,Q^{-1}A^*Av)>0$,
  since both $Q$ and $A^*A$ are positive definite. Therefore $J_g$ is
  compatible with $\omega$. 

  Define $J_g:= r_{\Phi}(g) = Q^{-1}A$. For an $\omega$-compatible almost
  complex structure
  $J$, we define $g_J = \omega(\cdot,J\cdot)$, and we note that
  \begin{equation*}r_{\Phi}(g_J)=r_{\Phi}(\omega(\cdot,J\cdot))=J,
  \end{equation*}since, in this case, $J=A$ and $Q=Id$.

  To see that $\Phi J_g=-J_g \Phi$, we recall that $\Phi^*g = g$ by hypothesis,
and we note that
  $-g(Av,w)=\Phi^{*}\omega(v,w)$, and therefore
  \begin{equation*}
    -g(Av,w)=\omega(\Phi v,\Phi w)=g(A\Phi
    v,\Phi w)=g(\Phi A\Phi v,w),
  \end{equation*} and therefore $\Phi A
  \Phi=-A$. Now note that $\Phi A^*A \Phi = -\Phi A^2 \Phi = \Phi A \Phi A = -A^2=A^*A$. Therefore $\Phi Q \Phi=Q$ as well,
  and $J_g\Phi=Q^{-1}A\Phi=-Q^{-1}\Phi A=-\Phi Q^{-1}A=-\Phi J_g$, as
  desired.

To see that the map is continuous, first note that the map $r_{\Phi}$ defined
above is
the restriction to the set $\mathcal{M}et_{\Phi}(V)$ of the map
$r:\mathcal{M}et(V) \to \mathcal{J}(V,\omega)$ defined in McDuff and
Salamon
\cite{McDuff_Salamon_Intro_1998}, Proposition 2.50(ii). Since the restriction of
a
continuous map is continuous (see, for instance, Munkres \cite{Munkres_2000}),
Theorem 18.2(d)), the result follows.
\end{proof}

\begin{cor}
  \label{cor:Existence of sym acs}
  Let $(M,\omega,\phi)$ be a real symplectic manifold. Let
  $\mathcal{J}_{\phi}(V,\omega)$ denote the space of
  $\omega$-compatible almost complex structures on $V$ with $\phi_*
  J=-J\phi_*$, and let $\mathcal{M}et_{\phi}(M)$ denote the space of
  positive definite bilinear forms $g$ such that $\phi^{*}g=g$. Then
  there exists a continuous map $r:\mathcal{M}et_{\phi}(M) \to
  \mathcal{J}_{\phi}(V,\omega) $ such that $r(g_J)=J$.
\end{cor}
\begin{proof}
  Let $g$ be a $\phi$-invariant Riemannian metric on $M$. Since the
  polar decomposition is canonical, we may construct an almost complex
  structure $J$ by constructing $J_{x}$ as in Lemma
  \ref{lem:Sym acs vector space} for each $x\in M$. By Lemma \ref{lem:Sym acs
vector space}, $J$ is
  $\omega$-compatible, and $J_x\phi_*=-\phi_*J_x$ for each $x\in Fix(\phi)$.

Now let $x\in M\backslash Fix(\phi)$. The proof that $\phi_* J_x=-J_{\phi(x)}
\phi_*$ 
also follows the proof of Lemma \ref{lem:Sym acs vector space}. In particular,
we first have that
  $-\omega_{x}(v,w)=\phi^{*}\omega_{x}(v,w)$, and therefore
  \begin{align*}
    -g_x(A_xv,w) & = -\omega_x(v,w) = \omega_{\phi(x)}(\phi_* v,\phi_* w) \\
& = g_{\phi(x)}(A_{\phi(x)}\phi_*v, \phi_* w) \\
& = g_{x}(\phi_*
A_{\phi(x)}\phi_* v,w),
  \end{align*} for all $v,w\in T_xM$, and therefore $\phi_* A_{\phi(x)}
  \phi_*=-A_x$. Now note that 
\begin{equation*}\phi_* A_{\phi(x)}^*A_{\phi(x)} \phi_* = -\phi_* A_{\phi(x)}^2
\phi_* = \phi_* A_{\phi(x)}
\phi_* A_x = -A_x^2=A_x^*A_x.
\end{equation*} Therefore $\phi_* Q_{\phi(x)} \phi_*=Q_x$ as well,
  and
\begin{equation*}J_{\phi(x)}\phi_*=Q_{\phi(x)}^{-1}A_{\phi(x)}\phi_*=-Q_{\phi(x)}^{-1}\phi_* A_x=-\phi_* Q_x^{-1}A_x=-\Phi J_x,
\end{equation*} as desired.
\end{proof}

\begin{rem}
  Indeed, this corollary shows that, for a real symplectic
  manifold $(M,\omega,\phi)$, there exists an $\omega$-compatible (and
  therefore tame) almost complex structure $J$ with
  $\phi_*J=-J\phi_*$.
\end{rem}

\begin{rem}
  \label{rem:Locally integrable acs}
  Note that if
  $\psi:(B(1+2\epsilon),\lambda^2\omega_0,B_{\mathbb{R}}(1+2\epsilon))\to
  (M,\omega,L)$ is a relative or real symplectic embedding, then the
  above constructions imply that there exists an $\omega$-tame
  (compatible) almost complex structure $J$ which equals $\psi_* i
  \psi^{-1}_*$ on a neighborhood of $\psi(0)$, and therefore $J$ is
  symmectrically or relatively integrable at $\psi(0)$ if $\psi$ is a
  real or relative embedding, respectively. If, in addition, $M$ has a
  real structure $\phi$ and $\psi$ is a real symplectic embedding,
  then $J$ also may be taken to satisfy $\phi_*J\phi_*=-J$. Similarly,
  if
  $\tilde{\psi}:(\mathcal{L}(1+2\epsilon),\rho(1,\delta),\mathcal{R}(1+2\epsilon))\to
  (\tilde{M},\tilde{\omega},\tilde{L})$ is a real or relative
  embedding, then there exists an $\tilde{\omega}$-tame almost complex
  structure $\tilde{J}$ such that
  $\tilde{J}=\tilde{\psi}_*\tilde{i}\tilde{\psi}^{-1}_*$ in a
  neighborhood of $\mathcal{L}(0)$.
\end{rem}

\subsection{Main Results}

We now state our main theorems, using the notation in Section
\ref{subsec:Setting and Notation}. Theorem \ref{thm:Blow-up}
is proved in Section \ref{subsec:Blow-up-down}, the proof of \ref{thm:Blow-down}
is completed in Section \ref{sec:Topological criterion}.	

\begin{thm}[Blow-up]
\label{thm:Blow-up}
\begin{enumerate}
  \item Let $(M,\omega)$ be a symplectic manifold and let $L \subset M$ be a 
  Lagrangian submanfiold. Suppose that for some small $\epsilon > 0$ there is a
    $(p,q)$-mixed symplectic embedding
\begin{equation*}
\psi:\coprod_{j=1}^{k}(B_{j}(1+2\epsilon),\lambda_j^2\omega_{0},B_{\mathbb{R},j}(1+2\epsilon))\hookrightarrow(M,\omega,L),
\end{equation*}let $P\subset M$ be the set $P:=\{\psi_j(0)\}_{j=1}^k$, and
let $J$ be an $\omega$-tame (compatible) almost complex structure which is
locally symetrically integrable in a neighborhood of $P$.

Then there exists a manifold
$\tilde{M}$, a family of symplectic forms $\tilde{\omega}_t, t\in
[0,1]$ on $\tilde{M}$, a submanifold $\tilde{L}\subset \tilde{M}$ which is
Lagrangian for each $\tilde{\omega_t}$, and an onto map $\Pi:\tilde{M}\to M$
such that the following is satisfied:
\begin{enumerate}
\item $\Pi$ is a diffeomorphism on $\Pi^{-1}(M\backslash P)$,
\item $\Pi^{-1}(\psi_j(0))\cong \mathbb{C}P^{n-1}$,
\item $\Pi(\tilde{L})=L$,
\item $\tilde{\omega}_0$ tames (is compatible with) an almost complex
structure $\tilde{J}$ for which each $\Pi^{-1}(\psi_j(0))$ is an almost complex
manifold, and
\item $\tilde{\omega}_1$ is in the cohomology class
\begin{equation*}
[\tilde{\omega}_1]=[\Pi^*\omega]+\sum_{j=1}^k \lambda_j^{2}e_{j},
\end{equation*}
where the $e_j$ are the Poincar\'e duals of the exceptional classes $E_j=[\Pi^{-1}(\psi_j(0))]$.
\end{enumerate}
  \item If, in addition, $M$ admits an anti-symplectic involution
    $\phi$ which satisfies
\begin{enumerate} 
  \item $\mbox{Fix} (\phi)=L$, 
 \item  $Im(\phi\circ \psi)=Im(\psi)$,
 \item  $Im(\phi\circ \psi_j)\cap Im(\psi_j)= \emptyset$ if
    $Im(\psi_j)\cap L=\emptyset$, and 
 \item $\psi_j \circ c=\phi \circ
    \psi_j$ if $Im(\psi_j)\cap L \neq \emptyset$, 
\end{enumerate}
then $\tilde{M}$ admits an involution $\tilde{\phi}:\tilde{M} \to
\tilde{M}$ such that $\tilde{\phi}^*\tilde{\omega}_t = -\tilde{\omega}_t$
 and $\phi \circ \Pi=\Pi \circ \tilde{\phi}$, and
$\tilde{\phi}_*\tilde{J}\tilde{\phi}_* = \tilde{J}$
  \end{enumerate}

\end{thm}

\begin{thm}[Blow-down]
  \label{thm:Blow-down}
\begin{enumerate}
\item Let $(\tilde{M},\tilde{\omega})$ be a symplectic manifold with
    Lagrangian $\tilde{L}$. Suppose there is a $(p,q)$-mixed
    symplectic embedding
\begin{equation*}\tilde{\psi}:\coprod_{j=1}^{k}(\mathcal{L}_j(r_j),\rho_j(\delta_j,\lambda_j),\mathcal{R}_j(r_j))\hookrightarrow(\tilde{M},\tilde{\omega},\tilde{L})
\end{equation*}
 such that $\psi^{-1}(\tilde{L})=\coprod_{j=1}^{p}\mathcal{R}_j(r_j)$. Let $C_j\subset \tilde{M}$ denote $\tilde{\psi}_j(\mathcal{L}(0))$, and let $C=\cup_j C_j$.

   Then there exists a symplectic manifold $(M,\omega)$, a $(p,q)$-mixed symplectic embedding \begin{equation}
\psi:\coprod_{j=1}^k (B(1+2\epsilon),\lambda_j\omega_0,B_{\mathbb{R}}(1+2\epsilon)) \to (M,\omega,L),
\end{equation} 
a Lagrangian submanifold $L\subset M$, and an onto map $\Pi:\tilde{M}\to M$ such that the following is satisfied:
\begin{enumerate}
\item $\Pi$ is a diffeomorphism on $\tilde{M} \backslash C$,
\item $\Pi(C_j) = p_j \in M$, where $p_j$ is a point,
\item $\Pi(\tilde{L})=L$, and
\item $\omega$ satisfies
\begin{equation*}
[\tilde{\omega}]-[\Pi^*\omega]\in \mathcal{E},
\end{equation*}

where $\mathcal{E}$ is the linear vector space generated by $e_1,\dots,e_k$, the Poincar\'e duals of the exceptional classes $E_j=[\tilde{\psi}_j(0)]$.
\end{enumerate}
  \item Suppose, in addition, $\tilde{M}$ admits an anti-symplectic
    involution $\tilde{\phi}$ which satisfies
\begin{enumerate}
\item $\mbox{Fix}(\tilde{\phi})=\tilde{L}$,
\item $Im(\tilde{\psi})=Im(\tilde{\phi} \circ \tilde{\psi})$,
\item $Im(\tilde{\phi}\circ \tilde{\psi_i})\cap Im(\tilde{\psi_i})= \emptyset$ if $Im(\psi_i)\cap L=\emptyset$, and
\item $\tilde{\psi}_i \circ \tilde{c}=\tilde{\phi} \circ \tilde{\psi}_i$ if $Im(\tilde{\psi}_i)\cap \tilde{L}
    \neq \emptyset$. 
\end{enumerate}

Then $(M,\omega)$ admits an anti-symplectic
    involution $\phi$ such that $\phi \circ \Pi=\Pi \circ \tilde{\phi}$.
  \end{enumerate}

\end{thm}

The idea of the relative blow-up construction is the same as blowing
up in the purely symplectic case: we remove the interior of a
ball from both $M$ and $\overline{\mathbb{C}P}^{n}$ (the bar indicating
that the orientation is reversed), and we glue them along their
boundaries, ensuring that the symplectic form $\tilde{\omega}$ of the
blow up $\tilde{M}$ acts appropriately. The difference in the relative
case is that the real parts of the balls removed from $M$ and
$\overline{\mathbb{C}P}^{n}$ are constrained to intersect the Lagrangians $L$ and
$\mathbb{R}P^{n}$, respectively, and the gluing proceedes so that the boundary of
the ($n$-dimensional) ball removed from $L$ is then glued to the
boundary of the corresponding hole in $\mathbb{R}P^{n}$, resulting in
the new Lagrangian $L\#\mathbb{R}P^{n}\cong\tilde{L}\subset\tilde{M}$ in
the blow-up.  The blow-down is the reverse process. We make these
operations precise in Section \ref{subsec:Blow-up-down}.

In four-dimensional complex geometry and symplectic topology, it is
extremely useful to know that one can blow down a symplectic manifold
$M$ along an embedded J-holomorphic sphere $C$ when $[C]\cdot [C]=-1$. In
complex geometry this is the so-called Castelnuovo-Enriques criterion
(see, for example, \cite{Griffiths_Harris_1978}, p.476). Unfortunately, it is a
difficult problem in general to derive a condition to detect
when blowing-down $C$ can be arranged to change the topology of a (non-orientable)
Lagrangian $\tilde{L}$. However, for Lagrangian submanifolds which are the fixed
point set of an anti-symplectic involution $\phi$ on a symplectic
$4$-manifold $M$, we have the following result, which we prove in Section \ref{sec:Topological criterion}.

We first give the following definition.

\begin{defn}
\label{defn:Exceptional}
We call $E\in H_2(M^4;\mathbb{Z})$ an \emph{exceptional class} if $E\cdot
E=-1$. 
If $u:\Sigma \hookrightarrow M^4$ is an embedding of the surface $\Sigma$, and
$u_*[\Sigma]=E$, then we say that $u(\Sigma)$ is an \emph{exceptional curve}. 
\end{defn}

\begin{thm}
\label{thm:Real blow-down condition}
  Let $(M^{4},\omega,\phi)$ be a real symplectic manifold with
  $L:=\mbox{Fix}(\phi)$, and let $J$ be an almost complex structure on 
  $M$ which tames $\omega$ and which satisfies $\phi_*J\phi_*=-J$. Suppose
$E\in
  H_{2}(M;\mathbb{Z})$ satisfies $E\cdot E=-1$ and
  $\phi_{*}E=-E$, and that there exists an embedded
$J$-holomorphic curve $C$ which represents $E$. Then
there exists a real symplectic manifold
  $(\check{M},\check{\omega},\check{\phi})$ and an onto map
  $\Pi:M\to \check{M}$ that satisfies
\begin{enumerate}
\item $\Pi$ is a diffeomorphism on $M \backslash C$,
\item $\Pi(C) = p \in \check{M}$, where $p$ is a point,
\item $\Pi\circ \phi=\check{\phi}\circ \Pi$, and
\item $\check{\omega}$ satisfies
\begin{equation*}
[\omega]-[\Pi^*\check{\omega}]\in \mathcal{E},
\end{equation*}
where $\mathcal{E}$ is the linear vector space generated by $e$, the Poincar\'e dual of the exceptional class $E=[\Pi^{-1}(p)]$.
\end{enumerate}

\end{thm}

As an application of the above theorems, we have the following theorem on the
real packing numbers for $(\mathbb{C}P^2,\mathbb{R}P^2)$, defined below.

\begin{defn}
\label{defn:Relative packing number}
Let $(M,\omega)$ be a symplectic manifold with Lagrangian submanifold $L\subset M$. We call the number
\begin{equation*}
p_{L,k}:=\sup_{\psi}\frac{\text{Vol}\left(\coprod_{i=1}^k (B(\lambda),\omega_0,B_{\mathbb{R}}(\lambda))\right)}{\text{Vol}(M)}
\end{equation*}
the \emph{k-th relative packing number for} $(M,L)$, where the $\sup$ is taken over all relative symplectic embeddings
\begin{equation*}
\psi:\coprod_{i=1}^k (B(\lambda),\omega_0,B_{\mathbb{R}}(\lambda))\to (M,\omega,L).
\end{equation*}
If $M$ is a real manifold with real structure $\phi$, $Fix(\phi)=L$,
and the $\sup$ is taken over all real embeddings of $k$ balls, then
$p_{L,k}$ is called the \emph{$k$-th real packing number}. We will
also denote the $k$-th real packing number by $p_{\mathbb{R},k}$. 
If the supremum is taken over all symplectic embeddings of $k$ balls into $M$,
then we denote the number $p_k$ and we call it the \emph{$k$-th packing number
of} $M$. 
\end{defn}
\begin{thm}
  \label{thm:Relative packing rp2}
For the pair $(\mathbb{C}P^{2},\mathbb{R}P^{2})$ with the standard
  symplectic form and real structure, the relative packing numbers
  $p_{\mathbb{R}P^2,k}$ are equal to the absolute packing numbers
for
  $\mathbb{C}P^{2}$.

\end{thm}

\section{Constructing the Relative and Real Blow-up and Blow-down}
\label{subsec:Blow-up-down}

We now construct the blow-up and blow-down of a symplectic manifold
$(M,\omega)$ relative to a Lagrangian submanifold $L$ or a real structure $\phi.$
The general strategy is to perform a complex blow-up or blow-down
locally and then define a symplectic form for the resulting
manifold. In each case, we first discuss the local models for the
symplectic forms in these constructions, and we then construct the
global blow up and blow down given a mixed, relative or real
symplectic embedding
\begin{align*}
\psi&:\coprod_{j=1}^k(B_j(1+2\epsilon),\lambda_j^{2}\omega_{0},B_{\mathbb{R},j}
(1+2\epsilon))\hookrightarrow
(M,\omega,L), \text{ or}\\
\tilde{\psi}&:\coprod_{j=1}^k(\mathcal{L}_j
(1+2\epsilon), \rho (\delta,\lambda_j),\mathcal{R}_j ) \hookrightarrow
(\tilde{M},\tilde{\omega},\tilde{L})
\end{align*} and the local models.

\subsection{Blow-up}\label{subsubsec:Blow-up}

In this section, we prove Theorem \ref{thm:Blow-up}, which we
restate
here for the convenience of the reader.

\begin{thm*}[Theorem \ref{thm:Blow-up}]

\end{thm*}

The construction proceeds as follows. We first construct a family of
symplectic forms $\tilde{\tau}(\epsilon, \lambda)$ on $\mathcal{L}$ by
pulling back the standard form $\omega_0$ on $\mathbb{R}^{2n}$ by a
family of specially constructed maps from
$\mathcal{L}\to \mathbb{R}^{2n}$. We arrange, in
particular, that the submanifold $\mathcal{R} \subset \mathcal{L}$ is
a Lagrangian for the forms $\tilde{\tau}(\epsilon,\lambda)$. We then
consider a relative symplectic and holomorphic embedding
$\psi:(B(1+2\epsilon),\lambda^2\omega_0,B_{\mathbb{R}}(1+2\epsilon),i)\to
(M,\omega,L,J)$, and we construct the blow-up manifold
$(\tilde{M},\tilde{L})$ by removing the ball and gluing in
$(\mathcal{L}(1+2\epsilon),\mathcal{R}(1+2\epsilon))$ along the
boundary. Finally, we use the local forms $\tilde{\tau}(\epsilon,\lambda)$ created on $\mathcal{L}$ in
the first step to construct the global symplectic form
$\tilde{\omega}$ on the blow-up $\tilde{M}$. For a real manifold $M$,
we also construct a real structure on the blow up $\tilde{M}$. We then show
that, given a relative symplectic embedding, and in view of some appropriate (and non-restrictive) assumptions on the almost complex structures, we may find a holomorphic
embedding of a smaller ball which is compatible with $L$ (or a real structure $\phi$), and we use this to remove the assumption
of holomorphicity on the embeddings.

In the following proposition, we construct the forms
$\tilde{\tau}(\epsilon, \lambda)$. Note that points \ref{enu:Pull-back
  of standard form}, \ref{enu:Local model is locally rho}, and
\ref{enu:Bar-tau-is-compatible} were proved in Proposition 5.1.A of
McDuff and Polterovich \cite{McDuff_Polterovich_1994}.
\begin{prop}
  \label{prop:Blow-up local model}Using the notation in Section
  \ref{subsec:Setting and Notation}, for every $\epsilon,\lambda>0$ there exists a
  symplectic form $\tilde{\tau}(\epsilon,\lambda)$ on $\mathcal{L}$
  such that the following holds:
  \begin{enumerate}
  \item \label{enu:Pull-back of standard form}
    $\tilde{\tau}(\epsilon,\lambda)=\pi^{*}(\lambda^{2}\omega_{0})$ on
    $\mathcal{L}-\mathcal{L}(1+\epsilon)$
  \item $\tilde{\tau}(\epsilon,\lambda)=\rho(1,\lambda)$ on
    $\mathcal{L}(\delta)$ for some $\delta>0$\label{enu:Local model is locally rho}
  \item $\tilde{\tau}(\epsilon,\lambda)$ is compatible with $\tilde{i}$, the
    canonical integrable complex structure on
    $\mathcal{L}$.\label{enu:Bar-tau-is-compatible}
  \item
    $\tilde{c}^{*}\tilde{\tau}(\epsilon,\lambda)=-\tilde{\tau}(\epsilon,\lambda)$,
    where $\tilde{c}$ denotes complex conjugation on
    $\mathcal{L}$.\label{enu:Symmetry in local blow up form}
  \item
    $\tilde{\tau}(\epsilon,\lambda)|_{\mathcal{R}}=0$\label{enu:R(r) is Lagrangian for bar(tau)}
  \end{enumerate}
\end{prop}

The proof of this proposition will be based on the following three
lemmas. Lemma \ref{lem:The calculation} was proved in
\cite{Guillemin_Sternberg_1989}, although it may also be proved by a
direct, if long, calculation. Lemma \ref{lem:Omega compatible i K\"ahler}
is a well-known result which we state so we may refer to it later,
and Lemma \ref{lem:Pull back of radial function is K\"ahler} was stated and the
proof sketched in
\cite{McDuff_Polterovich_1994}. The details of the proof are a
routine calculation. We begin with a definition.

\begin{defn}
\label{defn:Monotone radial}
  We say that $f:\mathbb{C}^n \to \mathbb{C}^n$ is a \emph{radial
    function} if $f(z)=\alpha(|z|)z$ for some real-valued function
  $\alpha:\mathbb{R}\to [0,\infty)$. We say that a radial function $f$
  is \emph{monotone} if $|z_0| \leq |z_1| \implies |f(z_0)|\leq
  |f(z_1)|$.
\end{defn}

\begin{lem}
  \label{lem:The calculation}
Let $h:\mathbb{R}^{2n}\to\mathbb{R}$ be
  the function $h(x)=\left(1+\frac{\lambda^{2}}{|x|^{2}}\right)^{1/2}$ and $\omega_0$ be the standard symplectic form on $\mathbb{R}^{2n}$. Let
  $H:\mathbb{R}^{2n}\backslash\{0\}\to\mathbb{R}^{2n}\backslash
  B(\lambda)$ be the mapping given by $H(x)=h(x)x$. Then
  $\pi^{*}H^{*}\omega=\rho(1,\lambda)$ on
  $\mathcal{L}\backslash\{(0,l)|l\in\mathbb{C}P^{n-1}\}$.

\end{lem}

\begin{lem}
  \label{lem:Omega compatible i K\"ahler}
Let $(M,\omega)$ be a
  symplectic manifold. Then $\omega$ is a K\"ahler form iff $\omega$ is
  compatible with an integrable almost complex structure $J$.


\end{lem}

\begin{lem}
  \label{lem:Pull back of radial function is K\"ahler}
Let $\omega$ be a
  K\"ahler form on $\mathbb{C}^{n},$ and suppose
  $f:\mathbb{C}^{n}\backslash\{0\}\to\mathbb{C}^{n}\backslash\{0\}$ is
  a monotone radial diffeomorphism. Then $f^{*}\omega$ is a K\"ahler
  form.

\end{lem}

We now prove the following proposition, which we follow with the proof
of Proposition \ref{prop:Blow-up local model}.

\begin{prop}
  \label{prop:Rho(k,l) is a symplectic form on L}For each
  $\kappa,\lambda>0$, $\rho(\kappa,\lambda)$ is a symplectic form on
  $\mathcal{L}$.
\end{prop}
\begin{proof}
  Let $\Omega=\omega_{0}^{n}$ denote the volume form on
  $\mathbb{R}^{2n}$, and let $H$ be defined as in the proof of Lemma
  \ref{lem:The calculation}.  Since $H\circ\pi$ is a diffeomorphism on
  $\mathcal{L}^{*}:=\mathcal{L}\backslash\{(0,z)|z\in\mathbb{C}P^{n-1}\}$,
  $\pi^{*}H^{*}\Omega$ is a volume form on $\mathcal{L}$, and
  therefore $\rho(1,\lambda)$ is non-degenerate for any
  $\lambda>0$. Since
  $\rho(\kappa,\lambda)=\kappa^{2}\rho(1,\lambda/\kappa)$, this
  implies that $\rho(\kappa,\lambda)$ is non-degenerate for
  $\kappa,\lambda>0$ as well. Since both $\omega_{0}$ and $\sigma$ are
  closed, $\rho(\kappa,\lambda)$ is closed as well on $\mathcal{L}^*$.

  Now let $(0,l)\in \mathcal{L}(0)$. Then $T_{(0,l)}\mathcal{L}\equiv
  T_l\mathbb{C}P^1\oplus T_0\mathbb{C}$. Taking $v\in
  T_l\mathbb{C}P^1$. Then
  $\rho(\kappa,\lambda)(v,iv)=\lambda^2\theta^*\sigma(v,iv)=\sigma(v,iv)>0$. Similarly,
  for $v\in T_0\mathbb{C}$,
  $\rho(\kappa,\lambda)(v,iv)=\pi^*\omega_0(v,iv)>0$, and therefore
  $\rho(\kappa,\lambda)$ is non-degenerate on $\mathcal{L}(0)$. Since
  $\rho(\kappa,\lambda)$ is closed as well, the form is symplectic as
  desired.\end{proof}

\begin{proof}[Proof of Proposition \ref{prop:Blow-up local model}]
  For each $\lambda>0$, let
  $h_{\lambda}:\mathbb{\mathbb{R}}^{2n}\backslash \{0\}\to\mathbb{R}$ be
  given by $h_{\lambda}(x)=\left(1+\frac{\lambda^{2}}{|x|^{2}}\right)^{1/2}$, and
  let $\delta,\epsilon, \epsilon_0 > 0$ satisfy
$(\delta+\epsilon_0)^{2}<\lambda^{2}\epsilon/2$. For
  $x\in B(\delta+\epsilon_0)$, we therefore have
  $|h_{\lambda}(x)x|^{2}=|x|^{2}+\lambda^{2} \leq (\delta + \epsilon_0)^2 +
\lambda^{2} < \lambda^{2}(\epsilon/2+1)$.
 Let $\beta(t):\R \to \R$ be a smooth non-increasing function which is $1$ for $t\leq \delta$ and $0$ for $t\geq \delta + \epsilon_0$, and let $\gamma(t):\R \to \R$ be a smooth non-increasing function which is $1$ for $t\leq (1+\epsilon)^{1/2}$ and $0$ for $t\geq 1+\epsilon$. Now define $F:\mathbb{R}^{2n}\backslash\{0\}\to\mathbb{R}^{2n}$ by
\begin{equation*}
  F(x)=\begin{cases}
    h_{\lambda}(x)x, & |x|<\delta\\
    \beta(|x|)h_{\lambda}(x)x + (1-\beta(|x|)\lambda
\left(1+\frac{\epsilon}{2}\right)^{\frac{1}{2}}\frac{x}{|x|}, & \delta \leq |x|
\leq \left(1 + \epsilon\right)^{\frac{1}{2}} \\
    \gamma(|x|)\lambda
\left(1+\frac{\epsilon}{2}\right)^{\frac{1}{2}}\frac{x}{|x|} +
(1-\gamma(|x|))\lambda x, &   \left(1 + \epsilon\right)^{\frac{1}{2}} < |x| <
1+\epsilon\\ 
    \lambda x, & 1+\epsilon\leq|x|\end{cases}
\end{equation*} 

\begin{lem}
\label{lem:F monotone}
The function $F$ defined above is a monotone radial diffeomorphism.
\end{lem}
\begin{proof}
We first note that $F$ is radial by definition. Furthermore, since the function is a diffeomorphism on each region, and since $\beta$ and $\gamma$ are smooth and all of their derivatives vanish on the boundary of each region, it follows that $F$ is continuous and all of the derivates of $F$ are well-defined for all $t$, and therefore $F$ is a diffeomorphism.
We now show that $F$ is monotone. For $|x| < \delta$, we have that $|F(x)|^2 = |x|^2 + \lambda^2$, so $F$ is monotone on this region. For $\delta \leq |x| \leq (1+\epsilon)^{1/2}$, we have 
\begin{align*}
|F(x)| &= \left|\beta(|x|)h_{\lambda}(x)x + (1-\beta(|x|)\lambda
\left(1+\frac{\epsilon}{2}\right)^{1/2}\frac{x}{|x|}\right|\\
&=  \left|\beta(|x|)h_{\lambda}(x) + (1-\beta(|x|)\lambda
\left(1+\frac{\epsilon}{2}\right)^{1/2}\frac{1}{|x|}\right||x|\\ 
&= \beta(|x|)(|x|^2+\lambda^2)^{1/2} + (1-\beta(|x|)\lambda
\left(1+\frac{\epsilon}{2}\right)^{1/2}.
\end{align*}
Here, the last equality follows since each term in the coefficient of $|x|$ is
always a positive real number. Setting $t:=|x|$ and $G(t):=|F(x)|$, we now
compute 
\begin{align}
\frac{dG}{dt} &= \beta(t)\frac{t}{(t^2+\lambda^2)^{1/2}} +
(t^2+\lambda^2)^{1/2}\beta^{'}(t) -\lambda
\left(1+\frac{\epsilon}{2}\right)^{1/2} \beta^{'}(t)\\
&=  \beta(t)\frac{t}{(t^2+\lambda^2)^{1/2}} + \left((t^2+\lambda^2)^{1/2}
-\lambda \left(1+\frac{\epsilon}{2} \right)^{1/2}\right) \beta^{'}(t).
\label{eqn:monotonicity}
\end{align}
The first term in the last equality is positive for all $t>0$. For $t\geq \delta+\epsilon_0$, $\beta^{'}(t)=0$, and for $\delta < t < \delta + \epsilon_0$, $\beta^{'}(t) \leq 0$. On this region, we also have that
\begin{align*}
t^2 + \lambda^2 < &(\delta + \epsilon_0)^2 + \lambda^2\ \\
 < & \lambda^2\epsilon /2 + \lambda^2\\
 = & \lambda^2\left(1 + \frac{\epsilon}{2}\right). 
\end{align*} 
Taking square roots of both sides of the inequality, we see that the 
coefficient to $\beta{̈́'}(t)$ in Equation \ref{eqn:monotonicity} is negative.
Therefore, $\frac{dG}{dt}$ is non-negative for $t \in
(\delta,(1+\epsilon)^{1/2})$, and it follows that $|F(x)|$ is non-decreasing
there.

Similarly, for $|x| \in ((1+\epsilon)^{1/2},1+\epsilon)$, we have
\begin{equation*}
|F(x)| = \gamma(|x|)\lambda \left(1+\frac{\epsilon}{2}\right)^{1/2} + (1-\gamma(|x|))\lambda |x|
\end{equation*}
Setting $t:=|x|$, and $G(t):= |F(x)|$, we compute 
\begin{equation*}
\frac{dG}{dt} = \left(\lambda \left(1+\frac{\epsilon}{2}\right)^{1/2}-\lambda t\right) \gamma^{'}(t) + (1-\gamma(|x|))\lambda.
\end{equation*} 
The last term is positive, and, since $\gamma^{'}(t) \leq 0$ and $t > (1+\frac{\epsilon}{2})^{1/2}$, the first term is positive as well. Therefore $G(t)$ is non-decreasing, and $|F(x)|$ is monotone on this interval. For $|x| \in [1+\epsilon,\infty)$, $|F(x)|$ is clearly monotone. This completes the proof.
\end{proof}

We now return to the proof of Proposition \ref{prop:Blow-up local model}.
Define $\tilde{\tau}(\epsilon,\lambda)$ by 
\begin{equation*}
\tilde{\tau}(\epsilon,\lambda) := \pi^{*}F^{*}\omega_{0} 
\end{equation*}
on $\mathcal{L}\backslash \mathcal{L}(0)$.
By Lemma \ref{lem:The
    calculation},
  $\pi^*F^*\omega_0=\rho(1,\lambda)=\tilde{\tau}(\epsilon,\lambda)$ on
  $\mathcal{L}(\delta)\backslash \mathcal{L}(0)$.  Since
  $\rho(1,\lambda)$ is a symplectic form on all of
  $\mathcal{L}(\delta)$, we may extend
  $\tilde{\tau}(\epsilon,\lambda)$ to all of $\mathcal{L}$ by
  assigning $\tilde{\tau}(\epsilon,\lambda):=\rho(1,\lambda)$ on
  $\mathcal{L}(0)=\pi^{-1}(0)$.  Now note that this form satisfies condition
  \ref{enu:Pull-back of standard form} and \ref{enu:Local model is locally
    rho} in the proposition by Lemma \ref{lem:The calculation} and the
definition of $F$.

  To see that $\tilde{\tau}(\epsilon,\lambda)$ is symplectic, we note
  that on $\mathcal{L}\backslash \mathcal{L}(0)$,
$\tilde{\tau}(\epsilon,\delta)$ is a
  pullback of the symplectic form $\omega_0$ by the diffeomorphism $F$, and on a
neighborhood of $\mathcal{L}(0)$, $\tau(\epsilon,\lambda)$ equals the symplectic
form $\rho(1,\lambda)$.

Items \ref{enu:Bar-tau-is-compatible}, \ref{enu:Symmetry in local blow up form},
and \ref{enu:R(r) is Lagrangian for bar(tau)} follow from a routine calculation.

\end{proof}

In the next proposition, we construct the global relative blow-up of a
manifold $M$ using a relative symplectic and holomorphic embedding of the ball
$(B(1+2\epsilon),\lambda^2\omega_0,B_{\mathbb{R}}(1+2\epsilon))$ with
the standard complex structure $i$. The use of holomorphic embeddings
here gives us extra control over the complex structure in the blow-up,
which we will be useful in our applications.

\begin{prop}
  \label{prop:Holomorphic blow-up} Let $(M,\omega)$ be a symplectic
  manifold with Lagrangian $L$, and let $J$ be an $\omega$-tame
  (compatible) almost complex structure.  Suppose that for $\lambda>0$ and some small
  $\epsilon > 0$, there is a relative
  symplectic and holomorphic embedding
  \begin{equation*}
\psi:\coprod_{j=1}^{k}(B_{j}(1+2\epsilon),\lambda_j^2\omega_{0},B_{\mathbb{R},j}(1+2\epsilon),i)\hookrightarrow(M,\omega,L,J).
\end{equation*} 

Then there exists a symplectic manifold
  $(\tilde{M},\tilde{\omega})$ with Lagrangian $\tilde{L}\subset \tilde{M}$,
  an $\tilde{\omega}$-tame (compatible) almost complex structure
  $\tilde{J}$, and an onto map
  $\Pi:\tilde{M}\to M$ such that

\begin{enumerate}
\item $\Pi$ is a diffeomorphism on $\Pi^{-1}(M\backslash \cup_{j=1}^k
\psi_j(0))$,
\item $\Pi_* \tilde{J} = J \Pi_*$
\item For all $j\in \{1,\dots,k\}, \Pi^{-1}(\psi_j(0))\cong \mathbb{C}P^{n-1}$,
\item $\Pi(\tilde{L})=L$, and
\item $\tilde{\omega}$ is in the cohomology class
\begin{equation*}
[\tilde{\omega}]=[\Pi^*\omega]+\sum_{j=1}^k \lambda_j^{2}e_{j},
\end{equation*}
where the $e_j$ are the Poincar\'e duals of the exceptional classes
\begin{equation*}
E_j=[\Pi^{-1}(\psi_j(0))].
\end{equation*}
\end{enumerate}
\end{prop}

\begin{rem}
  Note that the $E_i$ in the theorem above are the classes represented
  by the exceptional curves added in the blow-up.
\end{rem}

\begin{proof} 
  First, we consider the case when $k=1$. Consider the map
  $\pi:(\mathcal{L}(1+2\epsilon),\mathcal{R}(1+2\epsilon),\tilde{i})\to(B(1+2\epsilon),B_{\mathbb{R}}(1+2\epsilon),i)$ from Definition \ref{defn:Notation},
  where $\tilde{i}$ and $i$ are the standard complex structures on $\mathcal{L}$ and $\mathbb{C}^n$, respectively.
  Observing that $\pi$ gives a diffeomorphism between the boundaries
  $(\partial B(1+2\epsilon),\partial B_{\mathbb{R}}(1+2\epsilon))$ and
  $(\partial\mathcal{L}(1+2\epsilon),\partial\mathcal{R}(1+2\epsilon))$,
  we let $\pi_{\partial}$ denote the restriction of $\pi$ to
  $\partial\mathcal{L}(1+2\epsilon)$, and we define $\tilde{M}$ to be
  $\tilde{M}:=M\backslash \psi((B(1+2\epsilon),B_{\mathbb{R}}(1+2\epsilon))\cup_{\psi\circ\pi_{\partial}}(\mathcal{L}(1+2\epsilon),R(1+2\epsilon))$.
  This operation is summarized in the diagram below, with $\delta=1+2\epsilon$.\begin{equation}
    \xymatrix{(\mathcal{L}(\delta),\mathcal{R}(\delta))\ar[d]_{\pi} \ar@{^{(}->}^{\tilde{\psi}}[r] & (\tilde{M},\tilde{L})\ar[d]^{\Pi}\\
      (B(\delta),B_{\mathbb{R}}(\delta))\ar@{^{(}->}_{\psi}[r]
      & (M,L)}
    \label{eq:Blow-up/down diagram}\end{equation}
  where $\psi$ and $\tilde{\psi}$ are embeddings, and where the map
  $\Pi:(\tilde{M},\tilde{L})\to(M,L)$ is defined by \[
  \Pi(x)=\begin{cases}
    x, & x\notin\mbox{Im }\tilde{\psi}\\
    \psi\circ\pi\circ\tilde{\psi}^{-1}(x) & x\in\mbox{Im
    }\tilde{\psi}\end{cases}\] making the diagram commutative. Note
  that only $\psi$ is a symplectomorphism a
  priori.

  We now define a symplectic form on $\tilde{M}$. Recall that
  $\psi^{*}\omega=\lambda^{2}\omega_{0}$ by hypothesis.  We assign a
  symplectic form to $\tilde{M}$ by:\begin{equation}\label{eq:Def
      blow-up form} \tilde{\omega}=\begin{cases}
      \Pi^{*}\omega & \mbox{on } \tilde{M}\backslash\tilde{\psi}(\mathcal{L}(1+\epsilon))\\
      (\tilde{\psi}^{-1})^{*}\tilde{\tau}(\epsilon,\lambda) & \mbox{on
      } \tilde{\psi}(\mathcal{L}(1+2\epsilon))\end{cases}
  \end{equation}
  We check that $\tilde{\omega}$ is well-defined on $\mathcal{L}(1+2\epsilon)-\mathcal{L}(1+\epsilon)$. By Proposition \ref{prop:Blow-up local model} and the definition of
  $\tilde{\omega}$ and $\Pi$, on
  $\mathcal{L}(1+2\epsilon)-\mathcal{L}(1+\epsilon)$ we
  have\begin{align*}
    \Pi^{*}\omega = & (\tilde{\psi}^{-1})^{*}\pi^{*}\psi^{*}\omega\\
    = & \lambda^2(\tilde{\psi}^{-1})^{*}\pi^{*}\omega_0 =
    (\tilde{\psi}^{-1})^{*}\tilde{\tau}(\epsilon,\lambda),\end{align*}
  so $\tilde{\omega}$ is well defined.

  We define the almost complex structure $\tilde{J}$ on $\tilde{M}$ by 
  \begin{equation*}
  \tilde{J}=\begin{cases}\tilde{\psi}_*\tilde{i}\tilde{\psi}^{-1}_* & \text{ on } Im(\tilde{\psi})\\
            \Pi^{-1}_*J\Pi_*                         & \text{ on } \tilde{M}\backslash Im(\tilde{\psi})
\end{cases} 
  \end{equation*}
Note that since $\pi$ and
$\psi$ are holomorphic diffeomorphisms near the boundary of their respective
domains, $\Pi^{-1}_*J\Pi_*=\tilde{\psi}_*\tilde{i}\tilde{\psi}^{-1}_*$ on
$\tilde{\psi}(1+2\epsilon)\backslash \tilde{\psi}(1+\epsilon)$, and so
$\tilde{J}$ is well defined.
  To see that $\tilde{\omega}$ tames (is compatible with) $\tilde{J}$,
  we first note that $\Pi$ is holomorphic for $x\in
  \tilde{M}-\mathcal{L}(1+\epsilon)$, and we recall that
  $\tilde{\omega}=\Pi^{*}\omega$ on this region. Therefore, if
  $\omega$ tames $J$, then for $v,w\in T_{x}M$,
  $\tilde{\omega}(v,\tilde{J}v)=\lambda^{2}\omega(\Pi_{*}v,\Pi_{*}\tilde{J}v)=\lambda^{2}\omega(\Pi_{*}v,J\Pi_{*}v)>0$,
  so $\tilde{\omega}$ tames $\tilde{J}$ on this region. If, in
  addition, $\omega$ is compatible with $J$, we have,
  \begin{align*}
    \tilde{\omega}(\tilde{J}v,\tilde{J}w) = &
\Pi^*\omega(\tilde{J}v,\tilde{J}w)\\
    = & \omega(\Pi_*\tilde{J}v,\Pi_*\tilde{J}w)\\
    = & \omega(J\Pi_* v,J\Pi_* w)\\
    = & \omega(\Pi_* v,\Pi_* w) = \Pi^*\omega(v,w)
  \end{align*} as desired.
 
  For $x\in\mathcal{L}(1+\epsilon)$, we have that
  $\tilde{\omega}=(\tilde{\psi}^{-1})^*\tilde{\tau}$. Since
  $\tilde{\tau}$ is compatible with $\tilde{i}$, the canonical complex
  structure on $\mathcal{L}$, and $\tilde{\psi}$ is holomorphic, then
  $\tilde{\omega}$ is compatible with $\tilde{J}$ on this
  region. Therefore, if $\omega$ tames (is compatible with) $J$ on
  $M$, then $\tilde{\omega}$ tames (is compatible with) $\tilde{J}$ on
  all of $\tilde{M}$.

  Blowing up more than one point is done as above for each ball in the
  disjoint product
  $\psi:\coprod_{j=1}^{k}(B_{j}(r),\omega_{0},B_{\mathbb{R},j}(r))\hookrightarrow(M,\omega,L)$. That $\tilde{\omega}$ is in
  the desired cohomology class follows immediately from this
  construction.
\end{proof}

\begin{rem}
  When we want to emphasize the embedding $\psi$, we will refer to the
  symplectic blow up constructed as above as the blow-up of $M$
  \emph{relative to $\psi$}.
\end{rem}

In the following proposition we construct a real structure on the
blow-up $\tilde{M}$ given a real symplectic manifold $M$ and a
suitably symmetric embedding $\psi$ of a disjoint union of balls into
$M$.

\begin{prop}
  \label{prop:Real blow-up}
  Let $(M,\omega,\phi)$ be a real symplectic manifold, let $J$ be an $\omega$-tame (compatible)
  almost complex structure on $M$ which satisfies $\phi_* J
  \phi_*=-J$, and let
\begin{equation*}
\psi:\coprod_{j=1}^k
  (B_j(1+2\epsilon),\lambda_j^2\omega_0,i) \hookrightarrow (M,\omega,J)
\end{equation*}
  be a symplectic and holomorphic embedding. Suppose $\phi$ and $\psi$ satisfy
\begin{enumerate}
 \item $Im(\phi\circ \psi)=Im(\psi)$,
 \item $Im(\phi\circ \psi_j)\cap Im(\psi_j)= \emptyset$ if $Im(\psi_j)\cap L=\emptyset$, and 
 \item $\psi_j \circ c=\phi \circ
    \psi_j$ if $Im(\psi_j)\cap L \neq \emptyset$. 
\end{enumerate}

Then there exists a real symplectic manifold $(\tilde{M},\tilde{\omega},\tilde{\phi})$ and an onto map $\Pi:\tilde{M}\to M$ which satisfies 
\begin{enumerate}
\item $\Pi$ is a diffeomorphism on $\Pi^{-1}(M\backslash \cup_j \psi_j(0))$,
\item $\Pi^{-1}(\psi_j(0))\cong \mathbb{C}P^{n-1}$,
\item $\Pi \circ \tilde{\phi}=\phi \circ \Pi$, and
\item $\tilde{\omega}$ is in the cohomology class
\begin{equation*}
[\tilde{\omega}]=[\Pi^*\omega]-\sum_{j=1}^k \lambda_j^{2}e_{j},
\end{equation*}
where the $e_j$ are the Poincar\'e duals of the exceptional classes
\begin{equation*}
E_j = [\Pi^{-1}(\psi_j(0))] \in H_2(\tilde{M};\mathbb{Z}).
\end{equation*}
\end{enumerate}

Furthermore, the real structure $\tilde{\phi}$ and the almost complex structure $\tilde{J}$ in the blow-up $\tilde{M}$ satisfy $\tilde{\phi}_*
  \tilde{J}=-\tilde{J} \tilde{\phi_*}$, and for every $j$ with $\psi_j \circ c=\phi \circ \psi_j$, we have $\phi_* E_j =-E_j\in H_2(\tilde{M};\mathbb{Z})$.
\end{prop}
\begin{rem}
As we will see in the proof, in the case where there are balls which are embedded off of the Lagrangian, the blow-up is not constructed relative to $\psi$, but relative to another symplectic, holomorphic embedding with the same image. The ball embeddings whose image intersects the Lagrangian are left untouched, and those which take pairs of balls to $M\backslash L$ are changed to commute with $\phi$ and the standard real structure on $\mathbb{R}^{2n}$.
\end{rem}
In order to prove this proposition, we use the following lemmas. In the first lemma, we construct the blow-up given a real embedding $\psi$ on one ball such that $\psi\circ c=\phi \circ \psi$. In the second, we construct the simultaneous blow-up of an embedding $\psi$ of two balls $B_1$ and $B_2$ such that $\phi \circ \psi(B_1)=\psi(B_2)$.

\begin{lem}
  \label{lem:Real blow-up one ball}
  Let $(M,\omega,\phi)$ be a real symplectic manifold, let $J$ be an $\omega$-tame (compatible)
  almost complex structure on $M$ which satisfies $\phi_* J
  \phi_*=-J$. Suppose
\begin{equation*}
\psi:
  (B(1+2\epsilon),\lambda^2\omega_0,i) \hookrightarrow (M,\omega,J)
\end{equation*}
  is a symplectic and holomorphic embedding such that
  $\psi \circ c=\phi \circ \psi$. Then there exists a symplectic manifold $(\tilde{M},\tilde{\omega})$ that admits an anti-symplectic
  involution $\tilde{\phi}$ such that $\Pi$ and $\tilde{\omega}$ satisfy the conclusions of Proposition $\ref{prop:Holomorphic blow-up}$.

  Furthermore, the real
  structure $\tilde{\phi}$ in the blow-up $\tilde{M}$ satisfies $\tilde{\phi}_*
  \tilde{J}=-\tilde{J} \tilde{\phi_*}$, and 
\begin{equation*}
\tilde{\phi}_* [\Pi^{-1}(\psi(0))]=-[\Pi^{-1}(\psi(0))]\in
H_2(\tilde{M};\mathbb{Z}).
\end{equation*}
\end{lem}

\begin{proof}
  We first note that $\psi$ is a relative embedding, since
\begin{equation*}
\psi^{-1}(Fix(\phi))=Fix(c)=B_{\mathbb{R}}(1+2\epsilon).
\end{equation*}
  Now construct the blow-up $(\tilde{M},\tilde{\omega})$
  of $(M,\omega)$ relative to $\psi$ as in Proposition
  \ref{prop:Holomorphic blow-up}. Denote by $\tilde{c}$ the complex conjugation map on $\mathcal{L}$
  and recall that we have $\pi\circ\tilde{c}(z,l)=c\circ\pi(z,l)$,
  since $z\in l\Longleftrightarrow\overline{z}\in\overline{l}$ and $\overline{0}=0$. Given $\epsilon,\lambda > 0$, let $\tilde{\tau}(\epsilon,\lambda)$ be the symplectic form on
  $\mathcal{L}$ constructed in Proposition \ref{prop:Blow-up local
    model}, and recall that
    $\tilde{c}^{*}\tilde{\tau}(\epsilon,\lambda) = -\tilde{\tau}(\epsilon,\lambda)$.  We now define a map $\tilde{\phi}:\tilde{M}\to\tilde{M}$
  by
\begin{equation*} \tilde{\phi}(x)=\begin{cases}
    \Pi^{-1}\circ\phi\circ \Pi(x), & x\in\tilde{M}\backslash\tilde{\psi}(\mathcal{L}(1+\epsilon))\\
    \tilde{\psi}\circ\tilde{c}\circ \tilde{\psi}^{-1}(x), &
    x\in\tilde{\psi}(\mathcal{L}(1+2\epsilon)).\end{cases}
\end{equation*} 
By the commutativity of Figure
  \ref{eq:Blow-up/down diagram}, and the equivariance of $\psi$ we have, for $x\in \mathcal{L}(1+2\epsilon)\backslash \mathcal{L}(1+\epsilon),$ 
\begin{align*}
    \tilde{\psi}\circ\tilde{c}\circ\tilde{\psi}^{-1}(x) = &
\tilde{\psi}\circ\pi^{-1}\circ c\circ\pi\circ\tilde{\psi}^{-1}(x)\\
    = & \Pi^{-1}\circ\psi\circ c \circ\psi^{-1}\circ\Pi(x)\\
    = & \Pi^{-1}\circ\psi\circ\psi^{-1}\circ\phi\circ\Pi(x) =
\Pi^{-1}\circ\phi\circ\Pi(x).
\end{align*}
  Therefore $\tilde{\phi}$ is well-defined and a diffeomorphism. That
  $\tilde{\phi}$ is an anti-symplectic involution follows from the
  fact that $\Pi^{-1}\circ\phi\circ\Pi$ and
  $\tilde{\psi}\circ\tilde{c}\circ \tilde{\psi}^{-1}$ are
  anti-symplectic involutions on their respective domains.

  To see the last statement in the proposition, for $x\in \tilde{M}\backslash
\tilde{\psi}(\mathcal{L}(1+\epsilon))$, we compute
\begin{align*}
\tilde{\phi}_* \tilde{J} &= \Pi^{-1}_* \phi_* \Pi_*\tilde{J}\\
&= - \Pi^{-1}_*J\phi_*\Pi_*\\
&= - \tilde{J}\Pi^{-1}_*\phi_* \Pi_* = -\tilde{J}\tilde{\phi}.
\end{align*}
  For $x\in \tilde{\psi}(\mathcal{L}(1+2\epsilon))$, we have
  \begin{align*}
    \tilde{\phi}_*\tilde{J} &=
 \tilde{\psi}_*\tilde{c}_*\tilde{\psi}^{-1}_*\tilde{J}\\
    &= -\tilde{\psi}_*\tilde{i}\tilde{c}_*\tilde{\psi}^{-1}\\
    &= -\tilde{J}\tilde{\psi}_*\tilde{c}_*\tilde{\psi}^{-1} = 
-\tilde{J}\tilde{\phi},
  \end{align*}
  as desired.

Let $E=\tilde{\psi}(\mathcal{L}(0))$. To see that $\tilde{\phi}_*E=-E$, we note that $\tilde{c}(\mathcal{L}(0))=\mathcal{L}(0)$, and that $\tilde{c}$ reverses orientation. This completes the proof.

\end{proof}

\begin{lem}
\label{lem:Real blow-up two balls}
Let $(M,\omega,\phi)$ be a real symplectic manifold, let $J$ be an $\omega$-tame (compatible) almost complex structure. Suppose
\begin{equation*}
\gamma:\coprod_{i=1}^2 (B_i(1+2\epsilon),\lambda^2\omega_0,i)\to (M,\omega,J)
\end{equation*}
is a symplectic and holomorphic embedding such that $Im(\phi\circ \gamma_1)=Im (\gamma_2)$.
Then there exists a real symplectic manifold $(\tilde{M},\tilde{\omega})$
  with real structure $\tilde{\phi}$, an $\tilde{\omega}$-tame (compatible) almost complex structure, and an onto map $\Pi:\tilde{M}\to M$ which satisfies the conclusions of Proposition \ref{prop:Holomorphic blow-up}.

  Furthermore, the real
  structure $\tilde{\phi}$ and the almost complex structure
  $\tilde{J}$ in the blow-up $\tilde{M}$ satisfy $\tilde{\phi}_*
  \tilde{J}=-\tilde{J} \tilde{\phi_*}$. 
\end{lem}

\begin{proof}Define a map $\psi:\Pi_{i=1}^2 (B_i(1+2\epsilon),\lambda^2\omega_0,i)\to (M,\omega,J)$ by
\begin{equation*}
\psi(x)=\begin{cases} \gamma(x) & x\in B_1 \\
                     \phi \circ \gamma\circ c \circ \iota(x) & x\in B_2
\end{cases}
\end{equation*}
where $\iota:\Pi_{i=1}^2B_i\to \Pi_{i=1}^2B_i$ is the map given by
$\iota(x\in B_i)=x \in B_{i+1\mod 2}$. We note that, since $c$ and
$\phi$ are anti-holomorphic and $\gamma$ is holomorphic, $\psi$ is
holomorphic, and, similarly, since $c$ and $\phi$ and anti-symplectic,
and $\gamma$ is symplectic, $\psi$ is symplectic as well. Furthermore,
$\gamma,c,\phi,$ and $\iota$ are all 1-1, and we conclude that $\psi$
is a symplectic, holomorphic embedding. Now observe that $c\circ \iota$ is
an antisymplectic involution on $\Pi_{i=1}^2B_i$, $Im(\psi)=Im(\gamma)$ by definition, and that $\psi\circ
c\circ \iota=\phi\circ \psi$, so that $\psi$ is a real embedding for
the real structures $c\circ \iota$ and $\phi$. We now construct the
blow up of $M$ relative to $\psi$ as in McDuff and Polterovich \cite{McDuff_Polterovich_1994} (which is as in the relative blow-up without the Lagrangian). 

On $\coprod_{i=1}^2 \mathcal{L}_i$, we put the anti-symplectic involution
$\tilde{c}\circ \tilde{\iota}$, where $\tilde{c}$ is complex
conjugation on $\mathcal{L}$, and, as above,
$\tilde{\iota}:\coprod_{i=1}^2 \mathcal{L}_i\to \coprod_{i=1}^2\mathcal{L}_i$
is given by $\tilde{\iota}((z,l)\in \mathcal{L}_i)=(z,l)\in
\mathcal{L}_{i+1 \mod 2}$. Recall that
$\pi\circ\tilde{c}(z,l)=c\circ\pi(z,l)$, since $z\in
l\Longleftrightarrow\overline{z}\in\overline{l}$ and $\overline{0}=0$, and note that, by definition of
$\iota$ and $\tilde{\iota}$, we also have $\pi\circ\tilde{c}\circ \tilde{\iota}(z,l)=c\circ \iota \circ\pi(z,l)$.

  Given $\epsilon,\lambda > 0$, we define $\nu(\epsilon,\lambda)$ to
  be the symplectic form on
  $\coprod_{i=1}^{2}\mathcal{L}_{i}(1+2\epsilon)$, such that the
  restriction on each $\mathcal{L}_i$ is given by
  $\nu(\epsilon,\lambda)|_{\mathcal{L}_i}:=\tilde{\tau}(\epsilon,\lambda)$,
  where $\tilde{\tau}(\epsilon,\lambda)$ is the symplectic form on
  $\mathcal{L}$ constructed in Proposition \ref{prop:Blow-up local
    model}. 

 Now define a map $\tilde{\phi}:\tilde{M}\to\tilde{M}$
  by
\begin{equation*}
\tilde{\phi}(x)=\begin{cases}
    \Pi^{-1}\circ\phi\circ \Pi(x), & x\in\tilde{M}\backslash\tilde{\psi}\left(\coprod_{i=1}^2 \mathcal{L}(1+\epsilon)\right)\\
    \tilde{\psi}\circ\tilde{c}\circ \tilde{\iota}\circ \tilde{\psi}^{-1}(x), &
    x\in \tilde{\psi}\left(\coprod_{i=1}^2 \mathcal{L}(1+2\epsilon)\right),\end{cases}
\end{equation*}

where $\tilde{\psi}$ is the embedding of $\coprod_{i=1}^2
\mathcal{L}(1+2\epsilon)$ as in Figure \ref{eq:Blow-up/down
  diagram}. By the commutativity of Figure \ref{eq:Blow-up/down
  diagram}, we have, for $x\in \mathcal{L}(1+2\epsilon)\backslash
\mathcal{L}(1+\epsilon)$ \begin{align*}
  \tilde{\psi}\circ\tilde{c}\circ\tilde{\iota}\circ\tilde{\psi}^{-1}(x) = &
\tilde{\psi}\circ\pi^{-1}\circ c\circ\iota\circ \pi\circ\tilde{\psi}^{-1}(x)\\
  = & \mbox{\ensuremath{\Pi}}^{-1}\circ\psi\circ c\circ \iota
\circ\psi^{-1}\circ\Pi(x)\\
  = &
  \mbox{\ensuremath{\Pi}}^{-1}\circ\psi\circ\psi^{-1}\circ\phi\circ\Pi(x)
  =
  \mbox{\ensuremath{\Pi}}^{-1}\circ\phi\circ\Pi(x).\end{align*}
Therefore $\tilde{\phi}$ is well-defined and a diffeomorphism. That
$\tilde{\phi}$ is an anti-symplectic involution follows from the fact
that $\Pi^{-1}\circ\phi\circ\Pi$ and $\tilde{\psi}\circ\tilde{c}\circ\tilde{\iota}\circ
\tilde{\psi}^{-1}$ are anti-symplectic involutions on their respective
domains.

  To see the last statement in the proposition, for $x\in
  \tilde{M}\backslash \tilde{\psi}(\mathcal{L}(1+\epsilon))$, we
  compute
  \begin{align*}
    \tilde{\phi}_*{\tilde{J}} = & \Pi^{-1}_*\phi_*\Pi_*\tilde{J}\\
    = & \Pi^{-1}_*\phi_*J\Pi_*\\
    = & -\Pi^{-1}_*J\phi_*\Pi_*\\
    = & -\tilde{J}\Pi^{-1}_*\phi_*\Pi_* = -\tilde{J}\tilde{\phi}_*.
  \end{align*}
  For $x\in \tilde{\psi}(\mathcal{L}(1+2\epsilon))$, we have
  \begin{align*}
    \tilde{\phi}_*\tilde{J} = &
\tilde{\psi}_*\tilde{c}_*\tilde{\iota}_*\tilde{\psi}^{-1}_*\tilde{J}\\
    = & \tilde{\psi}_*\tilde{c}_*\tilde{\iota}_*\tilde{i}\tilde{\psi}^{-1}_*\\
    = & -\tilde{\psi}_*\tilde{i}\tilde{c}_*\tilde{\iota}_*\tilde{\psi}^{-1}_*\\
    = & -\tilde{J}\tilde{\psi}_*\tilde{c}_*\tilde{\iota}_*\tilde{\psi}^{-1}_* =
-\tilde{J}\tilde{\phi}_*,
  \end{align*}
  as desired.
\end{proof}

\begin{proof}[Proof of Proposition \ref{prop:Real blow-up}]
 For each $\gamma_i$ with $Im(\gamma_i)\cap L\neq \emptyset$ we construct the blow up using Lemma \ref{lem:Real blow-up one
    ball}. For each $\gamma_i$ such that
  $Im(\gamma_i)\cap Fix(\phi)=\emptyset$, we first recall that, by hypothesis, $Im(\gamma_i)\cap Im(\phi\circ \gamma_i)=\emptyset$. Since $Im(\phi\circ \gamma)=Im(\gamma)$, then there is a $\gamma_{i^{'}}$ with
  $Im(\phi\circ \gamma_i)=Im(\gamma_{i^{'}})$. We blow-up the pair $\gamma_i$,$\gamma_{i{'}}$ using Lemma
  \ref{lem:Real blow-up two balls}. The result follows.
\end{proof}

We now remove the hypothesis that our ball embeddings are
holomorphic. To do so, we isotope our form to a cohomologous one that admits a
small holomophic embedding around the
center of our embedding, which we may do under appropriate
assumptions on an almost complex structure that tames the symplectic
form. We then create a family of symplectic forms $\tilde{\omega}_t$ on the
blow-up such that the original one tames (or is compatible with) the
almost complex structure $\tilde{J}$ on the blow-up, and the last one
is in the cohomology class corresponding to the ball embedding. This
is the same strategy as that used in McDuff and Polterovich
\cite{McDuff_Polterovich_1994}, and the following proposition and its
proof are variants of Proposition 2.1.C in
\cite{McDuff_Polterovich_1994}, which we modify to keep track of the
Lagrangians $L$ and $\tilde{L}$ throughout the process. 
\begin{prop}
  \label{prop:Family of forms}
\begin{enumerate}
\item Let $\psi:
  (B(1+2\epsilon),\lambda^2\omega_0,B_{\mathbb{R}}(1+2\epsilon))\to
  (M,\omega,L)$ be a relative symplectic embedding. Suppose that $J$
  is an almost complex structure on $M$ which tames (is compatible
  with) $\omega$ and which is relatively integrable at $\psi(0)$. 
  
Then there
  exists a manifold $\tilde{M}$ with a submanifold $\tilde{L}$, a family of symplectic forms $\tilde{\omega}_t$, $t\in [0,1]$ on $\tilde{M}$, an almost complex structure $\tilde{J}$ on
  $\tilde{M}$, and an onto map $\Pi:\tilde{M}\to M$ such that $\tilde{\omega}_0$
  tames (is compatible with) $\tilde{J}$, $\tilde{L}$ is a Lagrangian
  for all the $\tilde{\omega}_t$, $\Pi(\tilde{L})=L$, and $\tilde{\omega}_1$
  satisfies
  \begin{equation*} [\tilde{\omega}_1]=[\Pi^*\omega] -
    \lambda^2 e,
  \end{equation*}
  where $e$ is the Poincare dual of the class $[\Pi^{-1}(\psi'(0))]\in
H_2(M;\mathbb{Z})$.
\item

  Suppose, furthermore, $M$ is a real symplectic manifold with real
  structure $\phi$, $\text{Fix} (\phi) =L$, $J$ satisfies $\phi_* J
  \phi_*=-J$, and $\psi \circ c=\phi
  \circ \psi$.  Then there exists a family of real structures
  $\tilde{\phi}_t$ on $\tilde{M}$ such that
  $\tilde{\phi}_t^{*}\tilde{\omega}_t=-\tilde{\omega}_t$,
  $(\tilde{\phi}_t)_*\tilde{J}(\tilde{\phi}_t)_*=-\tilde{J}$.
\end{enumerate}
\end{prop}

The proof depends on the following proposition, which is an adaptation
of Proposition 5.5.A in McDuff and Polterovich
\cite{McDuff_Polterovich_1994}, and which we prove in Section
\ref{subsec:Locally holomorphic maps}.

\begin{prop}
\label{prop:Normalization}
\begin{enumerate}
\item
  Let $(M,\omega)$ be a symplectic manifold and let $L \subset M$ be a Lagrangian submanifold. Let
\begin{equation*}\psi:(B(1+2\epsilon),\lambda^{2}\omega_{0},B_{\mathbb{R}}(1+2\epsilon))\to(M,\omega,L)
\end{equation*}
  be a relative symplectic embedding, and let $J$ be an
  almost complex structure on $M$ which tames
  $\omega$ and is relatively integrable at $\psi(0)\in L$. 

Then, for every compact subset $K \subset M \backslash \psi(0)$ there exists a
symplectic form $\omega'$ on $M$ isotopic to $\omega$ such that $\omega =
\omega^{'}$ on $K$ and $\omega^{'}$ is $J$-standard in a neighborhood
$\mathcal{N}$ of $\psi(0)$, i.e. $\omega{'}$ is Kahler on $\mathcal{N}$, and the
associated metric is flat in a neighborhood of $\psi(0)$.


\item \label{enum:Point 2}
\label{prop:Normalization2} 
In addition to the above, suppose that $M$ is a real
symplectic manifold with 
real structure $\phi$, $Fix(\phi)=L$, $J$ satisfies $\phi_*J\phi_*=-J$, $J$ is
symmetrically integrable around $\psi(0)$, and $\phi \circ \psi = \psi \circ c$.

Then we can construct the symplectic form $\omega^{'}$ on
$M$ to satisfy the conclusions above, and so that $\phi$ is a real structure for
$\omega^{'}$ and $\omega$ and $\omega'$ are isotopic through real
symplectic forms.

\end{enumerate}


\end{prop}

\begin{proof}[Proof of Proposition \ref{prop:Family of forms}]
  By Proposition \ref{prop:Normalization}, there is a symplectic form $\omega'$
on $M$ which is isotopic to $\omega$ an $J$-standard. By Proposition
\ref{prop:Real-Moser-Stability}, there is a diffeormorphism $F:M\to M$ such
that $F^*\omega = \omega'$. Replace $\omega$ by $\omega'$ and $\psi$ by $F
\circ \psi$. Abusing notation, we will refer to the new form and new
embedding by $\omega$ and $\psi$, respectively. Since $\omega$ is now $J$
standard, it follows that, for some 
  $\delta > 0$, there exists a relative holomorphic symplectomorphism
$\eta:B(\delta) \to M$, $\eta(0) = \psi(0)$. Now define the function
$S_{t}:B(1+2\epsilon)\to B(1+2\epsilon)$
by:
  \begin{equation*}
S_{t}(x)=\beta(t)x+(1-\beta(t))\left[\lambda(1+2\epsilon)\delta^{-1}
\alpha(|x|)+(1-\alpha(|x|))\right]x,
  \end{equation*}
  where $\beta(t)$ is a bump function with $\beta(t)=1$ for $t\leq 0$
  and $\beta(t)=0$ for $t\geq 1$, and $\alpha(t)$ is a bump function
  with $\alpha(t)=1$ for $t\leq \delta$ and $\alpha(t)=0$ for $t\geq
  1 + 2\epsilon - \epsilon'$ for some small $\epsilon' >0$. We wish to show that
$S_{t}$
  has the following properties:
  \begin{enumerate}
  \item $S_{0}=Id$
  \item $S_{t}$ is equal to the identity near $\partial
    B(1+2\epsilon)$
  \item $S_{t}^*\omega_0=\mu(t)\omega_0$, where
    $\mu(t):\mathbb{R}\to \mathbb{R}$ and
    $\mu(1)=\lambda^2(1+2\epsilon)^2\delta^{-2}$ on $B(\delta)$ for
    some $\delta > 0$
  \item $c \circ S_t = S_t \circ c$, where $c$ denotes complex conjugation
\label{item:complex conj}
  \item $B_{\mathbb{R}}(\lambda+\epsilon)$ is a Lagrangian for
    $S_{t}^*\omega_0$\label{item:Lag}
  \end{enumerate}
  The first four items above follow directly from the definitions of
$S_{t},\alpha$ and $\beta$, and item \ref{item:Lag} follows
immediately from item \ref{item:complex conj}. 

  Now let $F_t:M\to M$ be the extension of 
  $\psi \circ S_{t}\circ \psi^{-1}:Im(\psi)\subset M \to M$ by the identity
map, and
  set $\omega_t=F_t^*\omega$. Define
  \begin{equation*}    
\nu_{t}(z):=\eta\left(\frac{\delta}{1+2\epsilon}z\right):\left(B(1+2\epsilon),
     \frac{\delta^2}{(1+2\epsilon)^2}\mu(t)\omega_0\right) \to (M,\omega_t).
  \end{equation*}
  Since $\psi$ is relative holomorphic embedding, $\nu$ is
  also a holomorphic embedding, and since
  $\nu_t^*\omega_t=\frac{\delta^2}{(1+2\epsilon)^2}\mu(t)\omega_0$, $\nu_t$
  embeds symplectically into $(M,\omega_t)$ for every $t$.  Now take the forms
$\tilde{\omega}_t$
obtained by blowing up the family
  $\omega_t$ by the embeddings $\nu_t$.  We claim that
  $\tilde{\omega}_t$ verifies the conclusion of the theorem. By definition, the
$\nu_{t}$ are a symplectic and holomorphic maps into $M$, and by
hypothersis, $\omega_0$ is compatible with $J$, so by Proposition
  \ref{prop:Holomorphic blow-up}, $\tilde{\omega}_0$ is compatible
  with $\tilde{J}$. Since $F_1$ is isotopic to the identity, we see that
$[\omega_1]=[\omega]$, 
  from which it follows that $[\Pi^*\omega_1]=[\Pi^*\omega]$.
$\tilde{\omega}_1$ is therefore in the desired cohomology class, and the
  first part of the theorem is proved.

  If $M$ has a real structure $\phi$, and $\psi$ satisfies the
  hypotheses in the latter half of the theorem, then since $c\circ S_t = S_t
\circ c$, it follows that $\phi \circ F_t = F_t \circ \phi$, and therefore
$\phi^*\omega_t = -\omega_t$. Blowing up $(M,\omega_t)$, we create a family of
involutions $\tilde{\phi}_t:\tilde{M}\to \tilde{M}$ such that
$\tilde{\phi}^*\tilde{\omega}_t=-\tilde{\omega}_t$ and
$(\tilde{\phi})_*\tilde{J}(\tilde{\phi})_*=-\tilde{J}$, finishing the proof of
the proposition.
\end{proof}

We now prove Theorem \ref{thm:Blow-up}.

\begin{proof}[Proof of Theorem \ref{thm:Blow-up}]
  By Remark \ref{rem:Locally integrable acs}, there exists
  an almost complex structure on $M$ which is relatively integrable in a
  neighborhood of the points $\psi_j(0)$. Then by Proposition
  \ref{prop:Family of forms}, there exists a manifold $\tilde{M}$ with
submanifold $\tilde{L}$ and a family of symplectic
  forms $\tilde{\omega}_t$ on $\tilde{M}$ such that $\tilde{L}$ is a Lagrangian
for all $\tilde{\omega}_t$, and which satisfies
  $[\tilde{\omega}_1]=[\Pi^{*}\omega] - \sum_{k=1}^q \lambda_k^2e_k$,
  where the $e_k$ are the Poincar\'e duals of the exceptional spheres
  $C_k$ added in the blow-up.

  If, in addition, $M$ has a real structure $\phi$ and
  $Im(\psi)=Im(\phi\circ\psi)$, then, by Remark \ref{rem:Locally integrable acs}, $J$ may be chosen so that it is symmetrically integrable around the points $\psi_j(0)$ and $\phi_*J\phi_*=-J$. Therefore, by Proposition \ref{prop:Family of forms}, there exists a family of maps $\tilde{\phi}_t$ on the blow-up such
  that $\tilde{\phi}_t^*\tilde{\omega}_t=-\tilde{\omega}_t$, and this proves the theorem.
\end{proof}

\subsection{Blow-down}\label{subsubsec:Blow-down}

We now construct the blow-down of a symplectic manifold
$(\tilde{M},\tilde{\omega},\tilde{L})$. In particular, we will prove
Theorem \ref{thm:Blow-down}, stated again below.

\begin{thm*}[Theorem \ref{thm:Blow-down}]

\end{thm*}
In parallel to the blow-up construction, we begin by constructing a
family of forms on $\mathbb{C}^n$ from the forms
$\rho(\delta,\lambda)$, which we will then use to construct the global
form in the blow-down. The following proposition is adapted from
Proposition 5.1.B in \cite{McDuff_Polterovich_1994}.

\begin{prop}
  \label{prop:Blow-down local model}For every
  $\epsilon,\delta,\lambda>0$, there exists a K\"ahler form
  $\tau=\tau(\epsilon,\delta,\lambda)$ on $\mathbb{C}^{n}$ such that
  the following holds:
  \begin{enumerate}
  \item $\pi^{*}(\tau)=\rho(\delta,\lambda)$ on
    $\mathcal{L}-\mathcal{L}(1+\epsilon)$
  \item $\tau=\lambda^{2}\omega_{0}$ on $B(1)\subset\mathbb{C}^{n}$
  \item $\tau$ is compatible with $i$.
  \item $c^{*}\tau=-\tau$, where $c$ denotes complex conjugation on
    $\mathbb{C}^{n}$.
  \item $\mathbb{R}^{n}$ is a Lagrangian for $\tau$.\end{enumerate}

\end{prop}

\begin{proof}
  Note first that $\rho(\delta,\lambda)=\delta^{2}\rho(1,\nu)$ for
  $\nu=\lambda/\delta$.  Define
  \begin{equation*}
h_{\lambda}(z):=\left(1+\left(\frac{\lambda}{|z|}\right)^{2}\right)^{1/2},   
  \end{equation*}
and let $\epsilon_0>0$ be such that $2\epsilon_0(\nu^2 - 1) + \epsilon_0^2
(\nu^2 -1) < 1$. Let $\beta(t)$ be a smooth non-increasing function which is $1$
for $t\leq1$ and $0$ for
  $t\geq 1+\epsilon_0$, and let $\gamma(t)$ be a smooth non-increasing function which is $1$ for $t\leq 1+\epsilon_0$ and $0$ for $t\geq 1+\epsilon$. Then we define the map
  $G:\mathbb{C}^{n}\to\mathbb{C}^{n}$ by 
  \begin{equation*}
G(z)=\begin{cases}
    \nu z & \mbox{for }|z|\leq1\\
    \beta(|z|)\nu z+(1-\beta(|z|))\nu(1+\epsilon_0) & \mbox{for }1<|z|<1+\epsilon_0\\
    \gamma(|z|)\nu(1+\epsilon_0) + (1-\beta(|z|)h_{\nu}(z)z & \mbox{for } 1 +
\epsilon_0 \leq |z|\leq 1+\epsilon\\
    h_{\nu}(z)z & \mbox{for } 1+\epsilon < |z| 
    \end{cases}
  \end{equation*} and we
  define the form $\tau=\delta^{2}G^{*}\omega_0$. We claim that $\tau$
  satisfies the properties in the proposition. First,
to see the $G$ is a diffeomorphism, note that $\beta$ and
$\gamma$ are smooth. Since $G$ is a diffeomorphism on each region, it now
follows that $G$ is a diffeomorphism. 

To see that $G$ is monotone, we let $t:= |z|$ and $H(t):=|G(z)|$. On the first
region, it is clear from the definition of $G$ that $G$ is monotone. One
the second region, we have
\begin{equation*}
\frac{dH}{dt} = \beta'(t)\nu t + \beta(t)\nu - \beta'(t)(1 + \epsilon_0)\nu.
\end{equation*}
The second term is always non-negative, and, since $t\leq 1 + \epsilon_0$ and
$\beta'(t) \leq 0$, it follows that $\beta'(t)(t-1-\epsilon_0) \geq 0$, and
therefore $G$ is monotone on this region. 

For the third region, we have
\begin{equation*}
 \frac{dH}{dt} = \gamma'(t)(1+\epsilon_0)\nu + (t^2 +\nu^2)t(1 - \beta(t)) -
\gamma'(t)(t^2+\nu^2)^{1/2}
\end{equation*}
The second term is always non-negative, and since $2\epsilon_0(\nu^2 - 1) +
\epsilon_0^2 (\nu^2-1) < 1$ by hypothesis, it follows that
\begin{align*}
 \nu^2 +2\nu^2\epsilon_0 - 2\epsilon_0 - \epsilon_0^2 + \nu^2\epsilon_0^2 \leq&
1 + \nu^2 \implies\\
\nu^2(1+\epsilon_0)^2 \leq& \nu^2 + (1+\epsilon_0)^2.
\end{align*}
Recall that $\gamma'(t) \leq 0$, so taking the square root of both sides, we see
that 
\begin{equation*}
\gamma'(t)(\nu^2(1+\epsilon_0)^2 - \nu^2 + (1+\epsilon_0)^2) \geq 0, 
\end{equation*}
and therefore $G$ is monotone on this region. Since $G$ is monotone on the last
region by definition, $G$ is therefore
monotone everywhere.

The first property in the conclusion of the lemma
  follows from Lemma \ref{lem:The calculation}, the second from the
  definitions of $\tau$ and $G$ for $|z|\leq1$, and the third follows from Lemmas
  \ref{lem:Omega compatible i K\"ahler} and \ref{lem:Pull back of radial
    function is K\"ahler} and the fact that $G$ is a monotone radial
  function. To see the fourth point, note that $G(z)=\alpha(|z|)z$ for
  some real-valued function $\alpha:\mathbb{R}\to \mathbb{R}$. This
  implies that $c\circ G=G\circ c$, where $c$ is complex conjugation
  on $\mathbb{C}^n$, and therefore $c^*\delta^2 G^*\omega_0= \delta^2
  G^*c^*\omega_0=-\delta^2 G^* \omega_0$, as desired. This, in turn,
  proves the fifth point as well, and completes the proof.\end{proof}

In parallel to the blow-up construction, we split the blow-down into
two parts, the relative blow-down, in which we consider only a
Lagrangian, and we do not consider a real structure, and the real
blow-down. We now construct the relative blow-down.

\begin{prop}
  \label{prop:Mixed blow-down}
  Let $(\tilde{M},\tilde{\omega})$ be a symplectic manifold with
  Lagrangian $\tilde{L}$, and let $\tilde{J}$ be an
  $\tilde{\omega}$-tame (compatible) almost complex structure. Suppose
  there is a $(p,q)$-mixed holomorphic and symplectic embedding
  
\begin{equation*}\psi:\coprod_{j=1}^{k}(\mathcal{L}_j(r_j),\rho_j(\delta_j,\lambda_j),\mathcal{R}_j(r_j),i)\hookrightarrow(\tilde{M},\tilde{\omega},\tilde{L},\tilde{J})
\end{equation*}
  such that
  \begin{equation*}
    \psi^{-1}(L)=\coprod_{j=1}^{p}\mathcal{R}_j(r_j).
  \end{equation*} 
Then the conclusions of the first part of Theorem \ref{thm:Blow-down} are satisfied.
\end{prop}
\begin{proof}
  We consider the case when $(p,q)=(1,0)$. Let
  \begin{equation*} 
\tilde{\psi}:(\mathcal{L}(1+2\epsilon_0),\rho(\delta,\lambda),\mathcal{R}
(1+2\epsilon_0))
  \to (\tilde{M},\tilde{\omega},\tilde{L})
  \end{equation*}
  be a relative symplectic
  embedding such that
  $\tilde{\psi}^{*}\tilde{\omega}=\rho(\delta,\lambda)$.  We then perform a
  local complex blow down in $\mathcal{L}(1+2\epsilon)$, and we define
  the manifold $M$ by
 \begin{equation*}M:=\tilde{M}\backslash\tilde{\psi}(\mathcal{L}(1+2\epsilon))
  \cup_{\tilde{\psi} \circ \pi^{-1}|{\partial
      \mathcal{L}(1+2\epsilon)}}B(1+2\epsilon)\end{equation*} 
after which, as in the
  blow-up, we arrive at the commutative diagram \begin{equation}
    \xymatrix{(\mathcal{L}(1+2\epsilon),\mathcal{R}(1+2\epsilon))
      \ar[d]_{\pi} \ar@{^{(}->}[r]^(.67){\tilde{\psi}} & 
      (\tilde{M},\tilde{L})\ar[d]^{\Pi}\\
      (B(1+2\epsilon),B_{\mathbb{R}}(1+2\epsilon)) \ar@{^{(}->}[r]_(.67){\psi} &
      (M,L)}
  \end{equation} where $\Pi$ is defined by \[ \Pi(x)=\begin{cases}
    x & x\in\tilde{M}\backslash\tilde{\psi}(\mathcal{L}(1+2\epsilon))\\
    \psi\circ\pi\circ(\tilde{\psi}^{-1}) &
    x\in\tilde{\psi}(\mathcal{L}(1+2\epsilon)).\end{cases}\] We now
  define the following form on $M$:\[ \omega=\begin{cases}
    (\Pi^{-1})^{*}\tilde{\omega} & \mbox{on }M\backslash\psi(B(1+\epsilon))\\
    (\psi^{-1})^{*}\tau(\epsilon,\delta,\lambda) & \mbox{on
    }\psi(B(1+2\epsilon)).\end{cases}\] We check that the definition of $\omega$ agrees on $\psi(B(1+2\epsilon))\backslash \psi(B(1+\epsilon))$. On this region, we
  have \begin{align*}
    \omega = & (\psi^{-1})^{*}\tau(\epsilon,\delta,\lambda)\\
    = & (\psi^{-1})^{*}(\pi^{-1})^{*}\rho(1,\lambda)\\
    = & (\psi^{-1})^{*}\pi^{*}\tilde{\psi}^{*}\tilde{\omega} =
(\Pi^{-1})^{*}\tilde{\omega},\end{align*} so $\omega$ is
  well defined. Furthermore, we claim that $\omega$ is a symplectic
  form. Too see this, note that $\Pi$ is a diffeomorphism on
  $\Pi^{-1}(M\backslash\psi(B(1+\epsilon)))$, so $\omega^n$ is a volume form on $M\backslash \psi(B(1+\epsilon))$, and $\omega$ is therefore non-degenerate there. It is closed by
  definition. For $\psi(B(1+2\epsilon))$, we first note that by Proposition \ref{prop:Blow-down local model}, $\tau$
  is K\"ahler, and therefore symplectic on $\mathbb{R}^{2n}$. Since $\psi^{-1}$ is a
  diffeomorphism on $B(1+2\epsilon)$, $\omega$ is non-degenerate here
  as well, and closed by definition. 

 We define the almost complex structure $J$ on $M$ by 
  \begin{equation*}
  J=\begin{cases}\psi_*i\psi^{-1}_* & \text{ on } Im(\psi)\\
                        \Pi_* \tilde{J}\Pi^{-1}_*                         & \text{ on } M\backslash Im(\psi)
\end{cases} 
\end{equation*}

  Note that since $\pi$ and $\psi$ are holomorphic diffeomorphisms near the boundary of their respective domains, $\psi_*i\psi^{-1}_*= \Pi_* \tilde{J}\Pi^{-1}_*$ on $\psi(1+2\epsilon)\backslash \psi(1+\epsilon)$, and so $J$ is well defined.
  To see that $\omega$ tames (is compatible with) $J$,
  we first note that $\Pi$ is holomorphic and a diffeomorphism for $x\in
  \tilde{M}-\mathcal{L}(1+\epsilon)$, and we recall that
  $\omega=(\Pi^{-1})^{*}\tilde{\omega}$ on $M\backslash B(1+\epsilon)$. Therefore, if
  $\tilde{\omega}$ tames $J$, then for $v,w\in T_{\Pi(x)}M$,
  $\omega(v,Jv)=\tilde{\omega}(\Pi^{-1}_{*}v,\Pi^{-1}_{*}\tilde{J}v)=\tilde{\omega}(\Pi^{-1}_{*}v,J\Pi^{-1}_{*}v)>0$,
  so $\omega$ tames $J$ on $M\backslash \psi(B(1+\epsilon)$. If, in
  addition, $\tilde{\omega}$ is compatible with $\tilde{J}$, then on $M\backslash B(1+\epsilon)$, we have
  \begin{align*}
    \omega(Jv,Jw) = & (\Pi^{-1})^*\tilde{\omega}(Jv,Jw)\\
    = & \tilde{\omega}(\Pi^{-1}_*Jv,\Pi^{-1}_*Jw)\\
    = & \tilde{\omega}(\tilde{J}\Pi^{-1}_* v,\tilde{J}\Pi^{-1}_* w) =
\tilde{\omega}(\Pi^{-1}_* v,\Pi^{-1}_* w) = (\Pi^{-1})^*\tilde{\omega}(v,w)
  \end{align*} as desired.
 
  For $x\in\mathcal{L}(1+\epsilon)$, we have that
  $\omega=(\psi^{-1})^*\tau$. Since
  $\tau$ is compatible with $i$, the canonical complex
  structure on $B(1+2\epsilon)$, and $\psi$ is holomorphic (tautologically, by the definition of $J$), then
  $\omega$ is compatible with $J$ on this
  region. Therefore, if $\tilde{\omega}$ tames (is compatible with) $\tilde{J}$ on
  $\tilde{M}$, then $\omega$ tames (is compatible with) $J$ on
  $M$.

The condition on the cohomology class of $\omega$ follows immediately from the construction. This completes the proof of the proposition.
\end{proof}

We now construct the real blow-down for a real symplectic manifold
$\tilde{M}$.

\begin{prop}
  \label{prop:Real blow-down}
 Let $(\tilde{M},\tilde{\omega}, \tilde{\phi})$ be a real symplectic manifold and let $\tilde{L}=Fix(\tilde{\phi})$. Let $\tilde{J}$ be an $\tilde{\omega}$-tame (compatible) almost complex structure on $\tilde{M}$. Suppose that
    \begin{equation*}\tilde{\psi}:\coprod_{j=1}^{k}(\mathcal{L}_j(r_j),\rho_j(\delta_j,\lambda_j),\mathcal{R}_j(r_j),i)\hookrightarrow(\tilde{M},\tilde{\omega},\tilde{L},\tilde{J})
\end{equation*} is a symplectic and holomorphic embedding
    such that 
\begin{enumerate}
\item $\psi^{-1}(\tilde{L})=\coprod_{j=1}^{k}\mathcal{R}_j(r_j)$,
\item $Im(\tilde{\psi})=Im(\tilde{\phi} \circ \tilde{\psi})$, 
\item $Im(\tilde{\phi}\circ \tilde{\psi_i})\cap Im(\tilde{\psi_i})= \emptyset$ if $Im(\psi_i)\cap L=\emptyset$, and
\item    $\tilde{\psi}_i \circ \tilde{c}=\tilde{\phi} \circ \tilde{\psi}_i$ if $Im(\tilde{\psi}_i)\cap \tilde{L}
    \neq \emptyset$. 
\end{enumerate}

Then the conclusions of the second part of Theorem \ref{thm:Blow-down} are satisfied.
\end{prop}

As in the blow-up, we prove this in two parts. The first is the following.

\begin{lem}
\label{lem:Real blow-down one ball}
    Let $(\tilde{M},\tilde{\omega}, \tilde{\phi})$ be a real symplectic manifold and let $\tilde{L}=Fix(\tilde{\phi})$. Let $\tilde{J}$ be an $\tilde{\omega}$-tame (compatible) almost complex structure on $\tilde{M}$, and suppose that
    $\tilde{\psi}:(\mathcal{L}(r),\rho(\delta,\lambda),\mathcal{R}(r))\hookrightarrow(\tilde{M},\tilde{\omega},\tilde{L})$ is a symplectic embedding
    such that $\tilde{\psi}\circ c=\tilde{\phi}\circ \tilde{\psi}$.
    Then the blow-down $(M,\omega,L)$ admits an anti-symplectic
    involution $\phi$ and an almost complex structure $J$ such that $Fix(\phi)=L$ and $\phi_* J \phi_*=-J$.
\end{lem}

\begin{proof}
  Construct the blow-down $(M,\omega)$ as in Proposition
  \ref{prop:Mixed blow-down}. 
  Now define a map $\phi$ by
  \begin{equation*}
    \phi(x)=\begin{cases}
      \Pi\circ\tilde{\phi}\circ\Pi^{-1} & x\in M\backslash\psi(B(1+\epsilon))\\
      \psi\circ c \circ \psi^{-1}(x) &
      x\in\psi(B(1+2\epsilon))\end{cases},
  \end{equation*}
  Note that, for
  $x\in\psi(B(1+2\epsilon)-B(1+\epsilon)),$ \begin{align*}
    \psi\circ c\circ\psi^{-1}(x) = &
\psi\circ\pi\circ\tilde{c}\circ\pi^{-1}\circ\psi^{-1}(x)\\
    = &
\Pi\circ\tilde{\psi}\circ\tilde{c}\circ\tilde{\psi}^{-1}\circ\Pi^{-1}(x)\\
    = & \Pi\circ\tilde{\phi}\circ\Pi^{-1}(x),\end{align*} so the map
  $\phi$ is well-defined and a diffeomorphism. Furthermore,
  $\phi^{2}=Id$ by definition. To see that $\phi^{*}\omega=-\omega,$
  we have, for $x\in
  M\backslash\psi(B(1+2\epsilon))$,\begin{align*}
    \phi^{*}\omega_{x} & = & \phi^{*}(\Pi^{-1})^{*}\tilde{\omega}_{x}\\
    = & (\Pi^{-1})^{*}\tilde{\phi}^{*}\tilde{\omega}_{x}\\
    = & -(\Pi^{-1})\tilde{\omega}_{x}
      =  -\omega_{x},\end{align*} and for
  $x\in\psi(B(1+2\epsilon))$, we have\begin{align*}
    \phi^{*}\omega_{x} = & (\psi^{-1})^{*}c^{*}
\psi^{*}(\psi^{-1})^{*}\tau(\epsilon,\delta,\lambda)\\
    = & (\psi^{-1})^{*}c^{*}\tau(\epsilon,\delta,\lambda)\\
    = & -(\psi^{-1})^{*}\tau(\epsilon,\delta,\lambda)
      =  -\omega_{x},\end{align*} 

We now check that $\phi_* J \phi_*=-J$.  For $x\in M\backslash \psi(B(1+\epsilon))$, we compute
  \begin{align*}
    \phi_*J = & \Pi_*\tilde{\phi}_*\Pi^{-1}_*J\\
    = & \Pi_*\tilde{\phi}_*\tilde{J}\Pi^{-1}_*\\
    = & -\Pi_*\tilde{J}\tilde{\phi}_*\Pi^{-1}_*\\
    = & -J\Pi^{-1}_*\phi_* \Pi_*
      = -J \phi_*.
  \end{align*}
  For $x\in \psi(B(1+2\epsilon))$, we have
  \begin{align*}
    \phi_* J = & \psi_*c_*\psi^{-1}_*J\\
    = & \psi_*c_*i\psi^{-1}_*\\
    = & -\psi_*ic_*\psi^{-1}_*\\
    = & -J\psi_*c_*\psi^{-1}_*
      = -J\phi_*,
  \end{align*}
  which completes the proof.
\end{proof}
\begin{lem}
\label{lem:Real blow-down two balls}
    Let $(\tilde{M},\tilde{\omega}, \tilde{\phi})$ be a real symplectic manifold and let $\tilde{L}=Fix(\tilde{\phi})$. Suppose that
    $\tilde{\gamma}:\coprod_{j=1}^2(\mathcal{L}_j(r_j),\rho_j(\delta_j,\lambda_j),\mathcal{R}(r))\hookrightarrow(\tilde{M},\tilde{\omega},\tilde{L})$ is a symplectic embedding
    such that $\psi^{-1}(\tilde{L})=\emptyset$ and $Im(\tilde{\phi} \circ \tilde{\gamma}_1)=Im(\tilde{\gamma}_2)$.
    Then the blow-down $(M,\omega)$ admits an anti-symplectic
    involution $\phi$.
\end{lem}

\begin{proof}
   Since $Im(\tilde{\phi} \circ \tilde{\gamma}_1)=Im(\tilde{\gamma}_2)$, we can replace $\gamma$ with an embedding
\begin{equation*}\tilde{\psi}:\coprod_{j=1}^{2} ( \mathcal{L}_j ( r_j ), \rho_j(\delta_j,\lambda_j),\mathcal{R}(r))\hookrightarrow(\tilde{M},\tilde{\omega},\tilde{L})\end{equation*}
 
defined by
\begin{equation*}
\tilde{\psi}= \begin{cases} \tilde{\gamma}_1(x) & x\in \mathcal{L}_1 \\
                            \tilde{\phi} \circ \tilde{\gamma}_1 \circ \tilde{c}\circ \tilde{\iota}(x) & x\in \mathcal{L}_2,
\end{cases} 
\end{equation*}
where $\tilde{\iota}:\coprod_{j=1}^2 \mathcal{L}_j\to \coprod_{j=1}^2 \mathcal{L}_j$ is given by $\tilde{\iota}(x\in \mathcal{L}_{j})=x\in \mathcal{L}_{j+1 \mod 2}$. Note that $\tilde{c}\circ \tilde{\iota}$ is a real structure on $\coprod_{j=1}^2\mathcal{L}_j$ which makes $\tilde{\psi}$ a real map. The proof now follows exactly the proof of Lemma \ref{lem:Real blow-down one ball}, with $\tilde{c}\circ \tilde{\iota}$ in place of $\tilde{c}$.
\end{proof}

\begin{proof}[Proof of Proposition \ref{prop:Real blow-down}]
  For each $\tilde{\psi}_j$ with $Im(\tilde{\psi}_j)\cap L\neq \emptyset$, we construct the blow-down as in \ref{lem:Real blow-down one ball}. The rest of the maps come in pairs by assumption, and for each pair, we construct the blow-down as in \ref{lem:Real blow-down two balls}. The Proposition follows.
\end{proof}

Theorem \ref{thm:Blow-down} now follows easily from the above
propositions. We finish the proof here.

\begin{proof}[Proof of Theorem \ref{thm:Blow-down}]
  First, by Remark \ref{rem:Locally integrable acs}, there is an $\epsilon^{'}>0$, $\epsilon{'} < \epsilon$, and an $\tilde{\omega}$-tame almost complex structure $\tilde{J}$ such that $\tilde{J}$ is integrable on $\psi_i(\mathcal{L}(1+2\epsilon{'}))$ and which makes $\psi_i|_{(\mathcal{L}_{i}(1+2\epsilon{'})}$ holomorphic. Define $\mathcal{N}:=\coprod_{i=1}^k \mathcal{L}_i(1+2\epsilon^{'})$. If $M$ is not a real manifold, then we use
  Proposition \ref{prop:Mixed blow-down} to blow down $\tilde{M}$ using the map $\psi|_{\mathcal{N}}$. For a real manifold
  $\tilde{M}$ and a real embedding $\tilde{\psi}$, the theorem then
  follows from Proposition \ref{prop:Real blow-down}, again using the restriction $\psi|_{\mathcal{N}}$. This completes the proof.
\end{proof}

\begin{rem}
  We should note that the forms obtained in the local models,
  i.e. Propositions \ref{prop:Blow-up local model} and
  \ref{prop:Blow-down local model} are not the same as the forms
  constructed, respectively, from blowing up $\mathbb{C}^{n}$ at $0$
  and blowing down $\mathcal{L}$ along the exceptional divisor using
  Theorems \ref{thm:Blow-up} and \ref{thm:Blow-down}.  Constructing
  the genuine blow-up and blow-down forms, even of $\mathbb{C}^{n}$
  and $\mathcal{L}$, still requires an auxiliary symplectic embedding
  of either $B(r)$ or $\mathcal{L}(r)$, and these are absent from the
  form constructions of $\tau$ and $\tilde{\tau}$ in Propositions
  \ref{prop:Blow-up local model} and \ref{prop:Blow-down local
    model}. Because of this, we still use the constructions of Theorems
  \ref{thm:Blow-up} and \ref{thm:Blow-down}, even in these
  cases. \end{rem}

\subsection{Invariant Symplectic Neighborhoods and the Moser
  Stability Theorem in Real Symplectic Manifolds}
\label{subsubsec:Moser}

In this section we present a version of the Symplectic Neighborhood
Theorem adapted to leave invariant the fixed-point set of a real
symplectic manifold $(M,\omega,\phi)$.  We will use this below to
establish real packing results in
$(\mathbb{C}P^{2},\mathbb{R}P^{2})$ and other real symplectic four-manifolds. We
closely follow the
presentation of the analogous theorems for symplectic manifolds with
no real structure in McDuff and Salamon
\cite{McDuff_Salamon_Intro_1998}.

We begin with a definition.
\begin{defn}
\label{defn:Equivariant vector field}
Let $M$ be a smooth manifold and let $G$ be a compact Lie group which acts smoothly on $M$. We say that a vector field $X$ on $M$ is \emph{equivariant} with respect to $G$ (or \emph{$G$-equivariant}) if $\forall x\in M,g\in G$, we have $X(gx)=g_*X(x)$. 
\end{defn}

We now give the following standard result in equivariant dynamics, which we
quote from Ortega and Ratiu \cite{Ortega_Ratiu_2004} (Proposition 3.3.2(i))

\begin{prop}
\label{prop:Equiv dynamics}
Let $M$ be a smooth manifold, $A$ a subgroup of the group of diffeomorphisms of $M$. Let $U$ be an $A$-invariant open subset of $M$, and $X$ an $A$-equivariant vector field defined on $U$. Then, the domain of definition $Dom(F_t)\subset U$ of the flow $F_t$ of $X$ is $A$-invariant and $F_t$ is itself $A$-equivariant.
\end{prop}

\begin{lem}
  \label{lem:Real Moser's actual trick}
  Let $(M,\omega,\phi)$ be a real symplectic manifold with
  $Fix(\phi)=L$, and suppose $\omega_t, t\in [0,1]$ is a smooth family
  of symplectic forms with $\omega_0=\omega$ and
  $\phi^*\omega_t=-\omega_t$. Suppose, furthermore, that there exists
  a family of one-forms $\sigma_t$ with
  $\frac{d}{dt}\omega_t=d\sigma_t$ and
  $\phi^{*}\sigma_t=-\sigma_t$. Then there exists a family of
  diffeomorphisms $\alpha_t:M\to M$ such that

  \begin{align}
   \alpha_{t}^{*}\omega_{t} &=  \omega_{0}, \notag \\
   \alpha_{t}(L) &\subseteq  L\label{eq:Constraints},\\
   \alpha_{t}\circ \phi &=  \phi\circ \alpha_{t}. \notag
  \end{align}
\end{lem}
\begin{proof}
  We first note that, since the $\omega_t$ are non-degenerate, there exists a unique vector field $X_t$ which satisfies
\begin{equation}
 \sigma_t+\iota(X_{t})\omega_{t}=0.\label{eq:Moser's
      trick}
\end{equation} Given such a vector field $X_t$, let $\alpha_t$ be the solutions of
  \begin{align}
    \frac{d}{dt}\alpha_{t} &= X_{t}\circ\alpha_{t}, \label{eq:Diffeq}\\
    \alpha_{0} &= Id. \notag
  \end{align}
We now note that, because $\omega_{t}$ is closed, $d\omega_{t}=0$, and
  $\frac{d}{dt}\omega_{t}=d\sigma_t$, Equation \ref{eq:Moser's trick} implies that
\begin{equation*}
  0=\alpha_{t}^{*}\left(\frac{d}{dt}\omega_{t}+\iota(X_{t})d\omega_{t}+d\iota(X_{t})\omega_{t}\right)=\frac{d}{dt}\alpha_{t}^{*}\omega_{t}.
\end{equation*}
If $X_t$ is $\phi$-equivariant, then the flow $\alpha_t$ will be $\phi$-equivariant as well. To see that $X_t$ is $\phi$-equivariant, we first remark that
{\allowdisplaybreaks 
\begin{align*}
\phi^*(\sigma_t+\iota(X_t)\omega_t) = & 0, \\
                                   = & \phi^*\sigma_t+\phi^*\iota(X_t)\omega_t
\\
                                   = & -\sigma_t+\phi^*\iota(X_t)\omega_t,
\end{align*}
which implies that $\phi^*\iota(X_t)\omega_t = \sigma_t = -\iota(X_t)\omega_t$. Therefore, for all $v\in T_qM$,
\begin{equation*}
\omega_t(\phi(q);X_t(\phi(q)),\phi_*v)=-\omega_t(q;X_t(q),v).
\end{equation*}
However, $-\omega_t(q;X_t(q),v)=\omega_t(\phi(q);\phi_*X_t(q),\phi_*v)$, so
\begin{equation*}
\omega_t(\phi(q);X_t(\phi(q)),\phi_*v)=\omega_t(\phi(q);\phi_*X_t(q),\phi_*v).
\end{equation*} }
Since this is true for all $v\in T_qM$, $\phi_*$ is an isomorphism, and $\omega_t$ is non-degenerate, this implies that $\phi_*X_t(q)=X_t(\phi(q))$, and therefore the vector field $X_t$ is $\phi$-equivariant.

  Furthermore, for $v\in T_qL$, $v\neq 0$, we have that
  $\sigma_{t}(q;v)=-\sigma_t(q;\phi_*v)=0,$ so $\omega(q;X_{t},v)=0$, which implies that
  $X_{t}\in T_{q}L\subset T_{q}M$.  Since this is true for all $t\in
  [0,1]$, the diffeomorphisms $\alpha_{t}$ determined by equation
  \ref{eq:Diffeq} satisfy the constraints in equation
  \ref{eq:Constraints} as required.
\end{proof}

\begin{lem}
  \label{lem:Moser's trick}
  Let $M$ be a $2n$-dimensional smooth manifold, and let $\phi:M\to M$ be a diffeomorphism with $\phi^2=Id$. Let $L=Fix(\phi)$, and
  suppose $Q\subset M$ is a $\phi$-invariant submanifold. Suppose that $\omega_{0},\omega_{1}\in\Omega^{2}(M)$ are closed
  two forms with $\phi^{*}\omega_{i}=-\omega_{i}$ and such that, at
  every point $q\in Q$, $\omega_{0}|_{T_{q}M}=\omega_{1}|_{T_{q}M}$
  and the $\omega_{i}$ are non-degenerate on $T_qM$. Then there exist
  neighborhoods $\mathcal{N}_{0},\mathcal{N}_{1}$ of $Q$ and a
  diffeomorphism $\alpha:\mathcal{N}_{0}\to\mathcal{N}_{1}$ which
  satisfies 
\begin{enumerate}
\item $\alpha|_{Q}=Id$,
\item $\alpha^{*}\omega_{1}=\omega_{0}$,
\item $\alpha(\mathcal{N}_{0}\cap L)\subset L$, 
\item $\alpha\circ \phi=\phi\circ \alpha.$
\end{enumerate}
\end{lem}
\begin{proof}
  We may assume that $Q\cap L\neq\emptyset$, since, if this was not
  the case, we could just take the $\mathcal{N}_{i}$ small enough so
  that $\mathcal{N}_{i}\cap L=\emptyset$ and invoke the ordinary
  symplectic neighborhood theorem.

  Let $\mathcal{N}_{0}$ be a $\phi$-invariant tubular neighborhood of $L$. We
first show that there exists a $1$-form
  $\sigma\in\Omega^1(\mathcal{N}_0)$ such that
  \begin{align*}
    \sigma|_{T_{Q}M} &=0=\sigma|_{TL}, \\ 
    \phi^{*}\sigma &=-\sigma,\\  
    d\sigma &=\omega_{1}-\omega_{0}.
  \end{align*} To
  prove this, we endow $M$ with a $\phi$-invariant Riemannian metric, and
  consider the restriction of the exponential map to the normal bundle
  $TQ^{\perp}$. Since $Q$ is $\phi$-invariant, $TQ$ is $\phi_*$ invariant inside $TM$,
  and, therefore, since $\phi_*$ is an isomorphism from $T_xM$ to
  $T_{\phi(x)}M$, $TQ^{\perp}$ is $\phi_*$-invariant as well. Now, for a
  real number $\epsilon>0$, consider the neighborhood of the
  zero section of $TQ^{\perp}$
  \begin{equation*} V_{\epsilon}=\{(q,v)\in TM|q\in Q,v\in
    T_{q}Q^{\perp},|v|<\epsilon\}.
  \end{equation*}
  Define the set $U_{\epsilon}:=(V_{\epsilon}\cup
  \phi(V_{\epsilon}))$.  Then $U_{\epsilon}$ is $\phi$-invariant, and
  for $\epsilon$ sufficiently small, the restriction of the
  exponential map to $U_{\epsilon}$ is a diffeomorphism from
  $U_{\epsilon}$ to a neighborhood $\mathcal{N}_{1}$ of $Q$.  By a
  standard result in equivariant differential topology (Lemma
  \ref{lem:Equivariance of exp}, to be proven in Section
  \ref{subsub:Equivariant Differential Topology}), exp is
  equivariant as well. Now define
  $\psi_{t}:U_{\epsilon}\to\mathcal{N}_{1}$, $0<t<1$, by
  $\psi_{t}(\exp(q,v))=\exp(q,tv)$. For $t>0$, $\psi_{t}$ is a
  diffeomorphism onto its image. At $t=0$, $\mbox{Im}(\psi)\subseteq
  Q$, at $t=1$, $\psi_{1}=Id$, and $\psi_{t}|_{Q}=Id$ for all $t \in
  [0,1]$.  Since $\exp$ is equivariant, we also have $\psi_{t}\circ
  \phi(\exp(q,v))=\psi_{t}(\exp(c(q),\phi_{*}v))=\exp(\phi(q),t\phi_{*}v)=\phi\circ\exp(q,tv)=\phi\circ\psi_{t}$,
  so $\phi$ and $\psi_{t}$ commute.

  Let $\tau=\omega_{1}-\omega_{0}$. Then $\psi_{0}^{*}\tau=0$ and
  $\psi_{1}^{*}\tau=\tau$, and since $\psi_{t}$ is an equivariant
  diffeomorphism, we may define a $\phi$-equivariant vector field for $t>0$
  by $X_{t}=(\frac{\partial}{\partial t}\psi_{t})\circ \psi_t^{-1}$.
  Note that $X_{t}$ becomes singular at $t=0$. Nonetheless, we have\[
  \frac{d}{dt}\psi_{t}^{*}\tau=\psi_{t}^{*}\mathcal{L}_{X_{t}}\tau=d(\psi_{t}^{*}\iota(X_{t})\tau).\]
Let $\sigma_{t}=\psi_{t}^{*}\iota(X_{t})\tau$. Therefore, $\frac{d}{dt}\psi_{t}^{*}\tau=d\sigma_{t}$, and, by the definition of $X_t$, $\sigma_{t}$ is equal to \[
  \sigma_{t}(q;v)=\tau(\psi_t(q);\frac{d}{dt}\psi_{t}(q),d\psi_{t}(q)v).\]
Since $\sigma_t$ vanishes on $Q$ for all $t$, we may define $\sigma_{0}=0$, making $\sigma_t$ a smooth family for $t\in [0,1]$.
  In addition, we have that \[
  \tau=\psi_{1}^{*}\tau-\psi_{0}^{*}\tau=\int_{0}^{1}\frac{d}{dt}\psi_{t}^{*}\tau\,
  dt=d\sigma,\] where $\sigma=\int_{0}^{1}\sigma_{t}dt$.  It also
  follows from the equivariance of $\psi_{t}$ that  $(q,v)\in TL$, $\sigma_{t}=0$ for all $t\in [0,1]$. To see this, note that for
  $(q,v)\in TL$, $d\psi_{t}(q)v\in T_{q}L$, and since $\psi_{t}(q)\in
  L$ for all $t$, then $\frac{d}{dt}\psi_{t}(q)\in T_{\psi_{t}(q)}L$
  as well, making $\sigma_t(q;v)$ vanish by definition of $\tau$,
  because $L$ is Lagrangian for $\omega_0$ and $\omega_1$. To see that
  $\phi^{*}\sigma_t=-\sigma_t$, we compute
{\allowdisplaybreaks  \begin{align*}   
\phi^{*}\sigma_t(v) = &
\phi^{*}\tau(\psi_{t}(q);\frac{d}{dt}\psi_{t}(q),d\psi_{t} (q)\cdot)(v) \\
    = & \omega_1(\psi_{t}\circ \phi(q);\frac{d}{dt}\psi_{t}(\phi(q)),
d\psi_{t}\circ d\phi(q)v)\\
    & - \omega_0(\psi_{t}\circ \phi(q);\frac{d}{dt}\psi_{t}(\phi(q)),
d\psi_{t}\circ d\phi(q)v) \\
    = &
\omega_1(\phi\circ\psi_{t}(q);\frac{d}{dt}\phi\circ\psi_{t}(q),d\phi\circ
d\psi_{t}(q) v)\\
    & -\omega_0(\phi\circ\psi_{t}(q);\frac{d}{dt}\phi\circ\psi_{t}(q),d\phi\circ
d\psi_{t}(q) v) \\
    = & (\omega_1(\phi\circ\psi_{t}(q);d\phi\frac{d}{dt}\psi_{t}(q),d\phi\circ
d\psi_{t}(q) v)\\
    & -\omega_0(\phi\circ\psi_{t}(q);d\phi\frac{d}{dt}\psi_{t}(q),d\phi\circ
d\psi_{t}(q)\cdot v)) \\
    = & -\tau(\psi_{t}(q);\frac{d}{dt}\psi_{t}(q),d\psi_{t}(q)v) \\
    = & \vphantom{\frac{d}{dt}}-\sigma_t(v).
  \end{align*} }
  Therefore, $\phi^*\sigma=\int_0^1\phi^*\sigma_t\, dt=-\omega$. We
  have now created the desired $1$-form.

  Now consider the family of two-forms on $\mathcal{N}_0$ given by
  $\omega_{t}=\omega_{0}+t(\omega_{1}-\omega_{0})=\omega_{0}+td
  \sigma$, $t \in [0,1]$, and note that
  $\phi^{*}\omega_{t}=-\omega_{t}$ and
  $\frac{d}{dt}\omega_{t}=d\sigma$. The result now follows from Lemma
  \ref{lem:Real Moser's actual trick}. \end{proof}

\begin{thm}
  \label{thm:Real symplectic neighborhood}For $j=0,1$ let
  $(M_{j},\omega_{j},c_{j})$ be real symplectic manifolds with compact
  $c_{j}$-invariant symplectic submanifolds $Q_{j}$. Suppose that
  there is an equivariant symplectic isomorphism $\Phi:\nu_{Q_{0}}\to\nu_{Q_{1}}$
  of the symplectic normal bundles to $Q_{0}$ and $Q_{1}$ such that
  the restriction of $\Phi$ to the zero section is the
  symplectomorphism $\psi:(Q_{0},\omega_{0})\to(Q_{1},\omega_{1})$.
  Then there exist $c_{j}$-invariant neighborhoods $\mathcal{N}_{j}$
  of the $Q_{j}$ such that $\psi$ extends to an equivariant
  symplectomorphism
  $\psi^{'}:(\mathcal{N}_{0},\omega_{0},c_{0})\to (\mathcal{N}_{1},\omega_{1},c_{1})$,
  and $d\psi^{'}$ induces $\Phi$ on $\nu_{Q_{0}}$.\end{thm}
\begin{proof}
  We first show that $\psi$ extends to an equivariant diffeomorphism
  \begin{equation*}
  \psi_1:\mathcal{N}(Q_{0})\to\mathcal{N}(Q_{1})
  \end{equation*}
  that induces the
  map $\Phi$ on $\nu_{Q_{0}}$.  By Lemma \ref{lem:Equivariance of
    exp}, we may take the maps $\exp_i$ on $TM_i$ to be
  equivariant with respect to $c_i$. Define the map
  $\psi_1=\exp_{1}\circ\Phi\circ\exp_{0}^{-1}$, and consider the
  forms $\omega_{0}$ and $\omega_{1}^{'}=(\psi_1)^{*}\omega_{1}$ on
  $\mathcal{N}(Q_{0})$.  Note that, by construction, they are
  non-degenerate and they correspond on $T_{Q_{0}}M_{0}$. By Lemma
  \ref{lem:Moser's trick}, there is an equivariant diffeomorphism
  $\overline{\psi}$ of $\mathcal{N}(Q_{0})$ such that
  $\overline{\psi}^{*}\omega_{1}^{'}=\omega_{0}$. The composition
  $\psi^{'}=\psi_1\circ\overline{\psi}$ is the desired map.
\end{proof}
\begin{prop}
  \label{prop:Real symplectic chart}
  Let $(M,\omega,\phi)$ be a real symplectic manifold with real locus
  $L:=Fix(\phi)$. Let $x\in L$. Then there exists a symplectic
  equivariant map from a neighborhood $\mathcal{U}$ of $0$ in
  $(\mathbb{R}^{2n},\omega_0,c)$ to a neighborhood $\mathcal{V}$ of $x\in M$.
\end{prop}
In order to prove this proposition, we will need the following lemma.

\begin{lem}
\label{lem:Uniqueness of linear anti-symplectic involutions for omega0}
Let $\Phi:\mathbb{R}^{2n}\to \mathbb{R}^{2n}$ be a linear map such that
$\Phi^2=Id$ and $\Phi^*\omega_0(v,w)=-\omega_0(v,w)$ for all $v,w\in
\mathbb{R}^{2n}$. Then there exists a linear symplectic isomorphism
$\Psi:\mathbb{R}^{2n}\to \mathbb{R}^{2n}$ such that $\Psi \Phi =c_* \Psi$, where
$c$ is the standard anti-symplectic involution on $\mathbb{R}^{2n}$.
\end{lem}
\begin{proof}
We first consider the case $n=1$. (We do this to demonstrate the construction.
The proof does not proceed by induction.) Let $v\in Fix(\Psi)=Fix(c_*)$ such
that $\omega_0(v,iv)=1$, where $i$ is the standard complex structure on
$\mathbb{R}^2$. Then $\mathbb{R}^{2}=Span\{v,iv\}$. Let $w$ be an eigenvector of
$\Psi$ with eigenvalue $-1$. Let $\beta:=\omega_0(v,w)$. Now note that
$\{v,iv\}$ and $\{v,w\}$ are bases for $\mathbb{R}^2$. We define the map
$\Psi:\mathbb{R}^{2}\to \mathbb{R}^{2}$ to be the matrix sending $v \mapsto v$
and $w \mapsto (0,\omega_0(v,w))$, where the coordinates are the standard
$(x,y)=(v,iv)$ coordinates on $\mathbb{R}^2$. Then, for two vectors
$av+bw,cv+dw$, we have
\begin{align*}
\omega_0(av+bw,cv+dw) = & \omega_0(av,dw)+\omega_0(bw,cv) \\
                     = & (ad - bc)\beta.
\end{align*}
On the other hand,
\begin{align*}
\omega_0(\Psi(av+bw),\Psi(cv+dw)) = & \omega_0(av+\beta \cdot biv,cv+\beta\cdot
div) \\ 
 = & (ad - bc)\beta.
\end{align*}
Since the constants $a,b,c,d\in \mathbb{R}$ were arbitrary, we see that $\Psi$
is a linear symplectomorphism.
Without loss of generality, consider a linear anti-symplectic involution
$\Phi:\mathbb{R}^{2n}\to
\mathbb{R}^{2n}$ with
$Fix(\Phi)=\mathbb{R}^n$. Let $e_i,i\in \{1,\dots,2n\}$ denote the standard
basis in $\mathbb{R}^{2n}$, and consider the standard coordinates
$(x_1,\dots,x_n,y_1,\dots,y_n)$ in $\mathbb{R}^{2n}$. Take a basis
$(v_1,\dots,v_n)$ of the $-1$ eigenspace of $\Phi$, and define the map
$\Psi:\mathbb{R}^{2n}\to \mathbb{R}^{2n}$ to be the unique linear map sending
$e_i\mapsto e_i$, and $v_i\mapsto (0, \dots,0,
\omega_0(e_1,v_i),\dots,\omega_0(e_n,v_i))$, where there are $n$ leading zeros
in the coordinate (i.e. the $-1$ eigenspace of $\Phi$ is sent to the $-1$
eigensapce of $c_*$). 

We now show that $\Psi$ is a symplectomorphism. First note that for $i\in \{1,\dots,n\}$ we have $\omega_0(e_i,e_j)=0=\Phi^*\omega_0(e_i,e_j)$. Furthermore, we see that 
\begin{equation*}
-\omega_0(v_i,v_j)=\Phi^*\omega_0(v_i,v_j)=\omega_0(\Phi v_i,\Phi v_j)=\omega_0(v_i,v_j),
\end{equation*}
which implies that $\omega_0(v_i,v_j)=0=\Phi^*\omega(v_i,v_j)$. Now note that
\begin{equation*}
\Phi^*\omega_0(e_i,v_j)=\omega_0(e_i,v_j)\omega_0(e_i,e_i)=\omega_0(e_i,v_j),
\end{equation*}
as desired. Since $\Psi \Phi = \Phi c_*$, the proof of the lemma is complete.
\end{proof}

\begin{proof}[Proof of \ref{prop:Real symplectic chart}]
We first consider a $\phi$-invariant chart $(U,\alpha)$, $\alpha:U \subset M \to \mathbb{R}^{2n}$ centered at the point $p\in L$ which sends $L\to \mathbb{R}^{n}\subset \mathbb{C}^n$. We now consider the real structure $\Phi:=\alpha\circ\phi\circ \alpha^{-1}$ on $Im(\alpha)$. By Lemma \ref{lem:Uniqueness of linear anti-symplectic involutions for omega0}, there is a linear symplectic isomorphism $\Psi:\mathbb{R}^{2n}\to \mathbb{R}^{2n}$ such that $\Phi_* \Psi=c_* \Psi$ at the point $0$. Now apply Theorem \ref{thm:Real symplectic neighborhood} to the point $0\in \mathbb{R}^{2n}$.
\end{proof}

We now prove a real version of the Moser stability theorem.
\begin{prop}
  \label{prop:Real-Moser-Stability} Let $M$ be a closed manifold, and
  suppose that $\omega_{t}$ is a family of cohomologous symplectic
  forms on $M$ with anti-symplectic involution $\phi,$ i.e. such that
  $\phi^{*}\omega_{t}=-\omega_{t}$. Then there is a family of
  diffeomorphisms $\psi_{t}$ such that
  $\phi\circ\psi_{t}=\psi_{t}\circ\phi$, $\psi_{0}=id$, and
  $\psi_{t}^{*}\omega=\omega_{t}$.\end{prop}
\begin{proof}
  We must show that there is a smooth family of one forms $\sigma_{t}$
  such that
  \begin{equation}
    \label{eq:Dsigma = omega}
    d\sigma_{t}=\frac{d}{dt}\omega_{t}
  \end{equation}
  and $\phi^{*}\sigma_{t}=-\sigma_{t}$.

  The proof of Moser stability theorem (Theorem 3.17 in
  \cite{McDuff_Salamon_Intro_1998}) shows that there exists a smooth
  family of one forms $\tau_{t}$ satisfying (\ref{eq:Dsigma = omega}).
  Let $\sigma_{t}=\frac{1}{2}(\tau_{t}-\phi^{*}\tau_{t})$. Then
  $d\sigma_{t}=\frac{1}{2}(\frac{d}{dt}\omega_{t}-\phi^{*}\frac{d}{dt}\omega_{t})=\frac{1}{2}(\frac{d}{dt}\omega_{t}-\frac{d}{dt}\phi^{*}\omega_{t})=\frac{d}{dt}\omega_{t}$.
  Applying Lemma \ref{lem:Real Moser's actual trick}, we arrive at the
  desired result.
\end{proof}

\subsection{Locally holomorphic maps}
\label{subsec:Locally holomorphic maps}
In this section we prove Proposition \ref{prop:Normalization}, which shows that, given a relative or real symplectic
embedding 
\begin{equation*}
\psi:(B(1+2\epsilon),\lambda^2\omega_0,B_{\mathbb{R}}(1 + 2\epsilon))\to
(M,\omega,L)
\end{equation*}and an almost complex structure on $M$ which satisfies
some additional conditions, we may find a form $\omega^{'}$ on $M$
isotopic to $\omega$, and a relative symplectic embedding
$\psi^{'}:(B(\delta),\lambda^2\omega_0,B_{\mathbb{R}}(\delta))\to
(M,\omega^{'},L)$ whose image is contained in the image of $\psi$ but which
is holomorphic near the origin. We state the main proposition of this
section here.

\begin{prop*}[Proposition \ref{prop:Normalization}]
  
\item

\end{prop*}

In the proof we will use the following lemma, which is a modification
of Proposition 5.5.B in McDuff and Polterovich
\cite{McDuff_Polterovich_1994}.

\begin{lem}
  \label{lem:Local i-standard symplectic form}
  Let $\omega$ be a symplectic form on $B(1)$ which tames the standard
  complex structure $i$ and satisfies $c^*\omega=-\omega$ for the
  standard real structure $c$. Then there exists a
  symplectic form on $B(1)$, say $\Omega$, with the
  following properties:
  \begin{enumerate}
  \item $\Omega$ coincides with $\omega$ near the boundary of the
    ball;
  \item $\Omega$ tames $i$;
  \item $\Omega$ is $i$-standard near $0$, i.e. it is K\"ahler, and
    the associated metric is flat.
  \item $c^{*}\Omega=-\Omega$, and, in particular,
    $B_{\mathbb{R}}(1)$ is a Lagrangian for $\Omega$.
  \end{enumerate}
\end{lem}
\begin{proof}

  We divide the proof into three steps.

  Step 1. We claim that for every $\kappa>1$ and every $1>\epsilon >
  0$, there exists a K\"ahler form, say $\tau_{\kappa}$ on $B(1)$ which
  is equal to $\kappa^2\omega_0$ in $B(\epsilon/2\kappa)$ and
  coincides with $\epsilon^2\omega_0$ near the boundary, where
  $\omega_0$ is the standard symplectic form on $B(1)$. Indeed, take
  a monotone map $h$ defined by $h(z)=(\kappa/\epsilon)z$ for $z\in
  B(\epsilon/2\kappa)$ and such that $h$ is equal to the identity map near the
  boundary. Then the form $\tau_{\kappa}=h^*(\epsilon^2\omega)$ is
  K\"ahler by Lemma \ref{lem:Pull back of radial function is K\"ahler}.

  Step 2. Let $\rho$ be a bump function on $\mathbb{R}^{2n}$ which is
  radial, equal to $1$ near the origin, and which vanishes for
  $|z|>1-\delta$, for some $\delta>0$. Let $\omega_0$ be the standard
  symplectic form on $\mathbb{R}^{2n}$. Choose $\epsilon>0$ so that
  $\omega-\epsilon^2\omega_0$ tames $i$, and set
  $\rho_{\alpha}(z)=\rho((2\alpha/\epsilon)z)$, with $1 < \alpha < \kappa$. Finally, denote by
  $\beta$ a primitive of $\omega$ so that $\omega=d\beta$. Now
  consider the form
  \begin{equation*}
    \omega^{'}(\kappa)=\omega + (\tau_{\kappa} - \epsilon^2\omega_0 - d(\rho_{\alpha}\beta)).
  \end{equation*}
  We claim that $\omega^{'}(\kappa)$ satisfies the first four
  properties provided $\kappa$ is sufficiently large.

  We note that $\omega^{'}(\kappa)$ coincides with $\omega$ near the
  boundary, and near the origin $\omega^{'}(\kappa)$ is
  equal to $(\kappa^2-\epsilon^2)\omega_0$, and is therefore
  $J$-standard there. Moreover, $\rho_{\alpha}=0$ outside
  $B((\epsilon/2\alpha)(1 - \delta))$, and therefore
  $\omega^{'}(\kappa)=\omega-(\epsilon^2\omega_0 +\tau_{\kappa})$
  there. By assumption on $\epsilon$, $\omega-\epsilon^2\omega_0$
  tames $i$ on this region, and since $\tau_k$ is K\"ahler, $\omega{'}(\kappa)$ tames $i$ as well.

  We now check that $\omega^{'}(\kappa)$ tames $i$ inside
  $B(\epsilon/2\alpha)$. On this region
  \begin{equation*}
    \omega^{'}(\kappa)=(\kappa^2-\epsilon^2)\omega_0 + (1-\rho_{\alpha})\omega - d\rho_{\alpha} \wedge \beta.
  \end{equation*}

Let $|\cdot|$ denote the Euclidean distance of a vector in $\mathbb{R}^2n$.
Since $\overline{B}(\epsilon/2\alpha)$ is compact, the sphere bundle
  \begin{equation*}
    S=\{(x,\xi)|\,x\in \overline{B},|\xi|=1\}\subset T\mathbb{R}^{2n}
  \end{equation*}
  is compact, and therefore the function $d\rho_{\alpha} \wedge
\beta(\xi,i\xi)$ has a maximum, say $M$ on $S$. Therefore, for any $\xi \in
  T_x(B(\epsilon/2\alpha))$,
  \begin{equation*}
    d\rho_{\alpha}\wedge \beta(\xi,i\xi)=|\xi|^2d\rho_{\alpha}\wedge \beta
\left(\frac{\xi}{|\xi|},i\frac{\xi}{|\xi|}\right)
  \end{equation*}
  and therefore the maximum of $d\rho \wedge \beta (\xi,i\xi)$ on
  $S_a=\{(x,\xi)|\,x\in \overline{B},|\xi|=a\}\subset
  T\mathbb{R}^{2n}$ is $|\xi|^2M$. Since $\omega$ tames $i$ and $1-\rho\geq 0$,
$(1-\rho_{\alpha})\omega(\xi,i\xi)>0$ for all $\xi \neq 0$, and we conclude that
  $\omega^{'}(\kappa)(\xi,i\xi)>(\kappa^2-\epsilon^2-M)|\xi|^2$. Since
  the quantity on the right is positive for sufficiently large
  $\kappa$, $\omega^{'}(\kappa)$ tames $i$ if we choose $\kappa$
  large enough.

  Step 3. We see from the above that the symplectic form
  $\omega^{'}(\kappa)$ satisfies the first three properties, but does
  not necessarily respect the real structure. By Lemma \ref{lem:Form
    symmetrization}, however, since $\omega^{'}(\kappa)$ tames $i$ and
  $c_* i c_* = -i$, the form
  $\Omega=\frac{1}{2}(\omega^{'}(\kappa)-c^{*}\omega^{'}(\kappa))$ is
  symplectic, and satisfies the last property. We check that it
  satisfies the first three properties as well.  $\Omega$ tames $i$ by
  Lemma \ref{lem:Form symmetrization}, and, since
  $\omega^{'}(\kappa)=(\kappa^2-\epsilon^2)\omega_0$ near the origin,
  $\Omega(\kappa)=\omega^{'}(\kappa)$ near the origin, and is
  therefore $i$-standard on the same region as
  $\omega^{'}(\kappa)$. Furthermore, since $\omega^{'}(\kappa)$
  coincides with $\omega$ near the boundary of the ball and
  $c^*\omega=-\omega$, then $\Omega=\omega^{'}(\kappa)=\omega$ near
  the boundary of the ball as well. Thus $\Omega(\kappa)$ satisfies
  the conclusion of the lemma for $\kappa$ sufficiently large, and
  this completes the proof.
\end{proof}

\begin{proof}[Proof of Proposition \ref{prop:Normalization}]
 
  We first assume the hypotheses in Item \ref{prop:Normalization2} of the
proposition, i.e. that $M$ is a real symplectic manifold with real structure
$\phi$, $Fix(\phi) = L$, $\psi\circ c=\phi \circ \psi$, and $J$ is a tame almost
complex structure that is symmetrically integrable and satisfies $\phi_* J
\phi_* = -J$. We split the proof into two steps.

  Step 1. Let $(V,\gamma)$, $\gamma:V\subset M\to
  \mathbb{C}^n$ be a
  symmetric, holomorphic chart centered at $\psi(0)$, which exists
  because $J$ is symmetrically integrable around $\psi(0)$. Let
  $W\subset \gamma(V)$ be a small ball centered at $0$ inside
  $\gamma(V)$, and let $c$ denote
complex conjugation on $W$. By Lemma
  \ref{lem:Local i-standard symplectic form},
  there exists a symplectic form $\overline{\omega}$ on $W$
  which tames $i$, satisfies $c^{*}\overline{\omega} = -\overline{\omega}$, is
$i$-standard near $0$, and coincides with
  $(\gamma^{-1})^*\omega$ near the boundary of $W$. Let the form $\omega^{'}$ on
$M$ be given by
  \begin{equation*}
   \omega_{x}^{'} = \begin{cases} \omega_{x} & \text{for } x \in M \backslash
\gamma^{-1}(W) \\
(\gamma^{-1})^{*}\overline{\omega} & \text{for} x \in \gamma^{-1}(W)
\end{cases}
  \end{equation*}

Note that since $\overline{\omega}$ tames $i$ on $W$ and $\gamma$ is symmectric
and holomorphic, $\omega^{'}$ tames $J$ on $M$, $\omega^{'}$ is
$J$-standard near $\psi(0)$, and $\phi^*\omega^{'} = -\omega^{'}$. Therefore,
for each $s\in [0,1]$, $\omega_s = \omega' + s(\omega - \omega')$ is a
symplectic form, and, furthermore, $\phi^*\omega_s = -\omega_s$ for all $s \in
[0,1]$. Since the closed form $\omega - \omega'$ is non-zero only on a
contractible set, each of the $\omega_s$ are in the same cohomology class in
$H^2(M;\R)$ and $\omega$ and $\omega'$ are therefore isotopic through real
symplectic forms. 

  For the first part of the theorem, we note that, since $J$ is
  relatively integrable, there is a chart $(V,\gamma)$ around
  $\psi(0)$ such that $\gamma(L)\subset \mathbb{R}^n$ and where $c\circ
\gamma(V) = V$. If we use this
  chart in the place of $(V,\gamma)$ above and follow the same
  reasoning as above, the result follows.
\end{proof} 

\section{Topological Criterion for the Real Blow-down}\label{sec:Topological criterion}

In this section we prove Theorem \ref{thm:Real blow-down condition},
which gives a sufficient condition for blowing down a
real Lagrangian submanifold. 

\subsection{Equivariant Differential Topology}\label{subsub:Equivariant Differential Topology}

In this subsection, we collect the statements of several classical results from
equivariant differential topology which we will need for the proof of
Theorem \ref{thm:Real blow-down condition}. The proofs are found in Bredon
\cite{Bredon_1972} and Kawakubo \cite{Kawakubo_1991}

\begin{defn}
  Let $M$ be a $C^{\infty}$ manifold, and $G$ be a compact Lie
  group. If $\Phi:G\times M\to M$ is a smooth action of $G$, then we
  call $\Phi$ a \emph{$G$-action} on $M$, and if $M$ admits such a
  $G$-action, we call $M$ a \emph{$G$-manifold}.
\end{defn}

\begin{lem}
  \label{lem:Existence of invariant metric}Let $G$ be a compact Lie group,
  and let $M$ be a finite-dimensional $G$-manifold. Then there exists
  a $G$-invariant Riemannian metric $g$ on $M$.\end{lem}

\begin{lem}
  \label{lem:Fixclosed}Let $G$ be a compact Lie group, and let $M$ be
  a topological $G$-space. Then the fixed point set of $G$, $M^{G}$,
  is a closet set.\end{lem}

\begin{thm}
  \label{thm:Existence of invariant tubular neighborhood}
  Let $M$ be a $G$-manifold with $G$ finite. If $A$ is a
  closed $G$-invariant submanifold of $M$, then $A$ has an open $G$-invariant tubular
  neighborhood in $M.$ \end{thm} 

\begin{prop}
  \label{prop:Fixman}Let $G$ be a compact Lie group, and let $M$ be a
  $G$-manifold. Then the fixed point set of $G$, $M^{G}$, is a smooth closed
  submanifold of $G$. \end{prop}

\begin{lem}
  \label{lem:Equivariance of exp}Let $G$ be a compact Lie group, let $M$ be
  a finite dimensional $G$-manifold and let $g$ be a $G$-invariant Riemannian
  metric. Then the associated $exp$ map is $G$-equivariant. \end{lem}

\subsection{Proof of the Blow-down Criterion\label{sub:Proof-blow
    down crit}}

In this section, we prove Theorem \ref{thm:Real blow-down condition}, which we restate for convenience. 

\begin{thm*}[Theorem \ref{thm:Real blow-down condition}]

\end{thm*}

We begin by recalling a version of the
adjunction inequality, as given by Theorem 1.3 in McDuff
\cite{McDuff_1991_JDiffG}.
\begin{thm}
  \label{thm:Adjunction}Let $(M,J)$ be an almost complex 4-manifold and
  $A\in H_{2}(M;\mathbb{Z})$ be a homology class that is represented
  by a somewhere injective (closed) $J$-holomorphic curve $u:\Sigma\to M$.
  Then 
  \begin{equation*}
g\leq1+\frac{1}{2}(A\cdot A-c_{1}(A)),
\end{equation*} with equality iff
  $u$ is an embedding, where
$g$ is the genus of $\Sigma$.\end{thm}

We recall Definition \ref{defn:Exceptional} from Section \ref{ch:Introduction}.
\begin{defn*}
We call $E\in H_2(M^4;\mathbb{Z})$ an \emph{exceptional class} if $E\cdot
E=-1$. 
If $u:\Sigma \hookrightarrow M^4$ is an embedding of the surface $\Sigma$, and
$u_*[\Sigma]=E$, then we say that $u(\Sigma)$ is an \emph{exceptional curve}. 
\end{defn*}

It follows from the adjunction formula that
\begin{cor}
\label{lem:Chern class of exceptional sphere}
Let $M$ be a $4$-manifold, and let $u:\Sigma \hookrightarrow M$ be an 
exceptional $2$-sphere in $M$ such that $u_*[\Sigma] = E \in
H_2(M;\mathbb{Z})$. Then $c_1(u_*[\Sigma])=1$.
\end{cor}
\begin{proof}
Since $u$ is an embedding by definition, $0 = 1 + \frac{1}{2}(E\cdot E -
c_1(E))$, and therefore $c_1(E) = 1$.
\end{proof}

\begin{rem}\label{rem:Inherited involution}Suppose $(M,\omega,\phi)$ is a real
symplectic manifold with an almost complex structure $J$ which tames $\omega$
and satisfies $\phi_*J\phi_*=-J$. Let $u:\Sigma\to M$ be a closed
$J$-holomorphic curve, and suppose it is an embedding whose image is invariant
under $\phi$. Then $\Sigma$ inherits the symplectic form $u^*\omega$ and the
anti-symplectic involution $u^{-1}\circ\phi\circ u$.
\end{rem}

\begin{prop}
  \label{prop:FixS}Let $(S^{2},\omega)$ be endowed with an anti-symplectic
  involution $\phi$. If $\mbox{Fix}(\phi)\neq\emptyset$, then the
  fixed point set of $\phi$ is a circle.\end{prop}
\begin{proof}
  Let $G=\mathbb{Z}_2$ with smooth actions on $M$ given by the functions\
$\{Id,\phi\}$.  From Proposition \ref{prop:Fixman} we see that
$\mbox{Fix}(G)=\mbox{Fix}(\phi)$ is a closed submanifold of $S^{2}$. Denote
this submanifold by $K$. Now suppose $p\in K$, and let $v,w\in T_{p}K$. Then
$\omega(v,w)=\phi^{*}\omega(v,w)=-\omega(v,w)=0$, and so the fixed
point set is an isotropic
submanifold of $S^2$. By Remark \ref{rem:Fix phi Lagrangian}, $L$ is Lagrangian,
and therefore one-dimensional. 
  $\mbox{Fix}(\phi)$ is therefore equal to a closed Lagrangian submanifold of
  $S^{2}$ and is therefore diffeomorphic to a union of
  non-intersecting circles. This union is compact, and therefore
  finite, since $\mbox{Fix}(\phi)$ is topologically closed and $S^{2}$ is
  compact. 

  Suppose there is more than one circle in
  $\text{Fix}(\phi)$, say $\alpha_{1},...,\alpha_{k}$. Now choose two
  circles, which we denote $\gamma_{1}$ and $\gamma_{2}$.  $S^{2}$
  therefore decomposes as $S^{2}=D_{1}\cup C\cup D_{2}$, where the
  $D_{i}$ are the non-intersecting disks bounded by the $\gamma_{i}$,
  and $C$ is the closed cylinder between the discs. Now consider
  $\phi(D_{1})$.  Since $\phi$ is a diffeomorphism, it must send
  $D_{1}$ onto a disc bounded by $\gamma_{1}$, i.e. either $D_{1}$ or
  $C\cup D_{2}$.  

  Now suppose $\phi(D_{1})=C\cup D_{2}$. Then there is a
  point $x\in D_{1}$ such that
  $\phi(x)\in\gamma_{2}\Rightarrow\phi^{2}(x)\in\gamma_{2}\nsubseteq
  D_{1}$, which contradicts the assumption that $\phi^{2}=\mbox{Id}$. Therefore,
  $\phi(D_{1})=D_{1}$.
  Note that for any $x\in \gamma_1$, one of the eigenvalues of $d\phi(x)$ is
$-1$.  Therefore, there are points in a collar neighborhood of $\gamma_1$ in
$D_1$ which are sent by $\phi$ to a collar neighborhood of $\gamma_1$ in $D_2
\cup C$. However, this contradicts that $\phi(D_1)=D_1$, and concludes the
proof.
\end{proof}
\begin{cor}
  \label{cor:Intersection is a circle}
  Let $(M^{4},\omega,\phi)$ be a real symplectic manifold, and let
  $L:=\text{Fix}(\phi)$.  Let $J$ be an almost complex structure on $M$ such that
  $\phi_{*}J\phi_{*}=-J$, and let $E\in H_{2}(M;\mathbb{Z})$ be an
  exceptional class with $\phi_{*}E=-E$.  Suppose $u:S^2 \to M$ is an embedded
rational $J$-holomorphic curve that represents $E$. Then
  $u(\Sigma)\cap L$ is diffeomorphic to a circle. \end{cor}
\begin{proof}
Note first that
  $\phi\circ u\circ c$ is another $J$-holomorphic embedding
  that represents $E$, and its image is equal to
  $\mbox{Im}(\phi\circ u)$. Suppose now that $Im(u)\neq Im(\phi\circ u)$. Let $c$ denote complex conjugation
  on $\Sigma=S^{2}$. Because both maps $u$ and $\phi\circ u\circ c$ are $J$-holomorphic, their
  intersections are at most countable, and since $[\mbox{Im}(\phi\circ
  u\circ c)]=[\mbox{Im}(u)]=E\in H_{2}(M;\mathbb{Z})$, positivity of
  intersections in dimension $4$ (e.g. Theorem E.1.4 in McDuff and Salamon \cite{McDuff_Salamon_2004}) implies that $0\leq|\{\mbox{Im
  }u\}\cap\{\mbox{Im }\phi\circ u\circ c\}|\leq E\cdot E=-1,$ which
  is a contradiction. Therefore, $\mbox{Im}(u)=\mbox{Im}(\phi\circ
  u)$.  By Remark \ref{rem:Inherited involution}, $u(\Sigma)$ inherits a real
  structure from $M$, and it follows from Proposition \ref{prop:FixS} that the fixed
  point set of $\phi$ restricted to $u(\Sigma)$ is a circle. Since
  $\text{Fix}(\phi)=L\subset M$, it follows that $u(\Sigma)\cap L$ is
  diffeomorphic to a circle.\end{proof}
\begin{lem}
  There is a natural isomorphism between the oriented Lagrangian
  subspaces of $\mathbb{R}^{2n}$ and the quotient space
  $U(n)/SO(n)$.\end{lem}
\begin{proof}
  We recall from McDuff and Salamon \cite{McDuff_Salamon_Intro_1998} that the unitary matrix
  $U=X+iY$ given by a unitary Lagrangian frame is determined by the
  Lagrangian subspace $\Lambda$ up to right multiplication by a matrix
  in $O(n)$. Similarly, given an orientation $o(\Lambda)$ of
  $\Lambda$, we see that $U$ is determined by $(\Lambda,o(\Lambda))$
  up to right multiplication by a matrix in $SO(n)$.\end{proof}

\begin{lem}
  \label{lem:Odd maslov non-trivial bundle}Let $u:(D,\partial
  D)\to(M,L)$ be a $J$-holomorphic disk with boundary on a Lagrangian
  $L$. Suppose the Maslov index of $u$, $\mu(u)$, satisfies $\mu(u)\mbox{ mod }2=1$. Then
  $TL|_{\partial D}$ is a non-trivial bundle.\end{lem}
\begin{proof}
  Consider the commutative diagram\[
  \xymatrix{\pi_{1}(SO(n))\ar[r]_{i}\ar[d]_{0} & \pi_{1}(O(n))\ar[d]_{0}\\
    \pi_{1}(U(n))\ar[r]_{\cong}\ar[d] & \pi_{1}(U(n))\cong \mathbb{Z}\ar[d]\\
    \pi_{1}(U(n)/SO(n))\ar[r]_{\alpha}\ar[d] & \pi_{1}(U(n)/O(n))\cong \mathbb{Z}\ar[d]_{\beta}\\
    0\ar[r]\ar[d] & \mathbb{Z}_2\ar[d]\\
    0\ar[r] & 0}
  \]
  Note that the vertical exact sequences in the diagram are taken from the respective homotopy
  long exact sequences. Note that it follows from the diagram that the
  map $\beta$ is onto, and therefore that the map $\alpha$ is
  multiplication by $2$. Identifying the Maslov class of a loop
  $\gamma$ of Lagrangians with $[\gamma]\in \pi_1(U(n)/O(n))$, we see
  that the Maslov class of any loop $\gamma$ of oriented Lagrangians
  is even. Now consider a trivialization $\Phi:u^*TM\to D\times
  \mathbb{C}^n$. If $TL|_{\partial D}$ is trivial, then the loop of
  Lagrangians $\Lambda\circ \Phi|_{\partial D}\to U(n)/O(n)$
  is a loop of oriented Lagrangians, and therefore $\mu(u)$ is
  even. This concludes the proof. \end{proof}
\begin{lem}
  \label{lem:Non-trivial bundle in Lagrangian}Let $(M,\omega,\phi)$ be a
  four-dimensional real symplectic manifold with real structure
  $\phi$. Denote the fixed point set of $\phi$ by $L$, and let $E\in
  H_{2}(M;\mathbb{Z})$ be a homology class such that $E\cdot
  E=-1$. Suppose $u:(\mathbb{C}P^{1},\sigma,i)\to(M,\omega,J)$ is a J-holomorphic embedding such that $u_{*}[\mathbb{C}P^{1}]=E,$ and such
  that the intersection $Im(u)\cap L\cong S^{1}$. Then the intersection of $TL$ with the normal bundle of
  $Im(u)$, i.e. $TL\cap\nu(Im(u))$, is nontrivial.\end{lem}
\begin{proof}
  We note that $c_{1} (u^*TM)=2-1=1,$ and that the Maslov number of
  $u=2c_{1}(E)=2$.  Let $u_1,u_2:D^2\to M$ denote the two disks which make up $u$. We claim that the Maslov index of each disc
  must be $1$. First, recall that
  $\mu(u_1)+\mu(u_2)=\mu(u)$ by the properties of the Maslov index. Second, the involution $\phi:M\to M$ induces a diffeomorphism from $Im(u_1)$ to $Im(u_2)$, and $\phi_*:TM\to TM$ is a vector bundle isomorphism from $u_1^*TM$ to $u_2^*TM$. Again, the properties of the Maslov index (see Theorem C.3.5 in McDuff and Salamon \cite{McDuff_Salamon_2004}) imply that $\mu(u_1)=\mu(u_2)$, and this implies that possibilities other than $(1,1)$
  for the Maslov indices of the two discs may not occur. It follows that
  that the bundle $T_{S^{1}}L=TS^{1}\oplus\nu_{L}(S^{1})$ is
  non-trivial by Lemma \ref{lem:Odd maslov non-trivial bundle}, where $\nu_{TL}(S^{1})$ denotes the part of
  the normal bundle of $S^{1}$ which lies in $TL$. Since $TS^{1}$ is
  trivial, then $\nu_{L}(S^{1})$ cannot be, and the lemma is
  proved.\end{proof}
\begin{lem}
\label{lem:Non-empty intersection}
  Let $M$ be a four-dimensional real symplectic manifold with real
  structure $\phi$. Denote the fixed point set of $\phi$ by $L$, and
  let $E\in H_{2}(M;\mathbb{Z})$ be a homology class such that $E\cdot
  E=-1$  and $\phi_{*}E=-E$. Suppose, furthermore, that there exists an embedding of the surface $\Sigma$, $i:\Sigma\to M$, with
  $i_{*}[\Sigma]=E$. Then $E\cdot L=1$ mod
  $2$.\end{lem}
\begin{proof}
  First, we perturb $i$ so that $i(\Sigma)\cap L$ and
  $i(\Sigma)\cap\phi\circ i(\Sigma)$ are generic. Let $p\in
  i(\Sigma)\cap\phi\circ i(\Sigma),$ $p\notin L.$ Then $\phi(p)\in
  i(\Sigma)\cap\phi\circ i(\Sigma)$, $\phi(p)\notin L,$ and, in particular,
  $p$ and $\phi(p)$ do not affect the value of either $E\cdot E$ mod $2$ or $E\cdot L$ mod $2$. Suppose now
  that $E\cdot L=0$ mod $2.$ Then there exist an even number of points
  in the intersection $i(\Sigma)\cap L,$ and, combined with the above,
  this implies that there are an even number of points in
  $i(\Sigma)\cap\phi\circ i(\Sigma).$ However,
  $i_{*}[\Sigma]\cdot\phi_{*}i_{*}[\Sigma]=1$ mod $2,$ which is a
  contradiction. Therefore $E\cdot L=1$ mod $2.$
\end{proof}

We recall a version of the Riemann Mapping Theorem from \cite{Pommerenke_1992} (see also \cite{Burckel_1979}).
\begin{thm}
\label{thm:Caratheodory}
Let $D$ denote the unit disk in $\mathbb{C}$, let $\Omega$ be a simply connected domain in $\mathbb{C}$, $(\Omega\neq \mathbb{C})$, and assume that the boundary $\partial \Omega$ is locally connected. Then there is a holomorphic isomorphism $f:D\to \Omega$ that extends to a continuous map from $\bar{D} \to \bar{\Omega}$. Moreover, if $\partial \Omega$ is a Jordan curve, then $f$ extends to a homeomorphism from $\bar{D}$ to $\bar{\Omega}$.
\end{thm}

We now prove Theorem \ref{thm:Real blow-down condition}.

\begin{proof}[Proof of Theorem \ref{thm:Real blow-down condition}]
  Let $u:\Sigma \to M$ be the embedded $J$-holomorphic curve whose image is $C$.
By hypothesis, $[C]\cdot [C]=-1$ so Lemma \ref{lem:Non-empty intersection}
implies that $C\cap L\neq \emptyset$.  By Corollary \ref{cor:Intersection is a
    circle}, $C$ intersects $L$ in a circle, whose preimage we denote
  $S$. Let $D_1$ and $D_2$ be the two open discs in $C$ with boundary
  $S$. Note that, for each $x\in D_1$, $\phi(x)\in D_2$. Now let $H_1$
  and $H_2$ denote the two hemispheres of $\mathbb{C}P^1$ with
  boundary $\mathbb{R}P^1$. By Theorem \ref{thm:Caratheodory}
  there exists a holomorphic map $\alpha:D_1\to H_1$ which extends to a homeomorphism from $\bar{D}_1$ to $\bar{H}_1$.  Now define a map
  $\tilde{\alpha}:C \to \mathbb{C}P^1$ by
\begin{equation}
\tilde{\alpha}(x)= 
\begin{cases}
  \alpha(x) & \text{if } x\in \bar{D}_1 \\
  c\circ \alpha(\phi(x)) & \text{if } x\in D_2, \\
\end{cases}
\end{equation}
where $c$ denotes complex conjugation on $\mathbb{C}P^1$. We claim that $\tilde{\alpha}$ is holomorphic on all of $\mathbb{C}P^1$. First, choose a holomorhpic chart $\gamma_1:W\subset C \to \mathbb{C}$ centered at a point $x\in \mathbb{R}P^1$ which sends $U\cap \mathbb{R}P^1$ to $\mathbb{R}$. Let $\gamma_2:V\subset \mathbb{C}P^1\to \mathbb{C}$ be a holomorphic chart centered at $\tilde{\alpha}(x)\in V$, and note that $\tilde{\alpha}$ is holomorphic iff $\gamma_2\circ\tilde{\alpha}\circ \gamma_1^{-1}$ is holomorphic for any pair of charts. To prove that this is the case, we appeal to Morera's theroem, which we recall below, as stated in Conway \cite{Conway_1978}, following the proof of the Schwartz Reflection Principle.

{ \allowdisplaybreaks
\begin{thm}[Morera's Theorem]
\label{thm:Morera's theorem}
Let $U$ be a region in $\mathbb{C}$ and let $f:U\to \mathbb{C}$ be a continuous function such that $\int_T f=0$ for every triangular path $T$ in $U$. Then $f$ is analytic in $U$.
\end{thm}
To apply this theorem, we need to show that for each triangular path $T\subset U$, $\int_T f=0$.
Denote $\gamma_1^{-1}(U)$ by $U$, let $U^+=U\cap \{z|Im(z)>0\}$, $U^0=\{z|Im(z)=0\}$, $U^-=\{z|Im(z)<0\}$, and $f:=\gamma_2\circ \tilde{\alpha}\circ \gamma^{-1}:U\to \mathbb{C}$. Choose a triangular path $T$ in $U$. We see that $\int_T f=0$ iff $\int_P f=0$ for any triangular or quadrilateral path $P$ in $U^+ \cup U^0$ and $U^-\cup U^0$. Furthermore, if $P \subset U^\pm$, then $\int_P f=0$, since $f$ is holomorphic on $U^\pm$ by definition. We therefore let $T$ be the triangle with vertices $[a,b,c]$, where the edge $[b,c]$ is contained in the real axis. The same argument applies for a quadrilateral path. Let $\Delta$ denote the union of the path $T$ and its interior. $f$ is continuous on $U^+\cup U^0$ by construction, and therefore it is uniformly continuous on $\Delta$. Therefore, for any $\epsilon>0$ there exists a $\delta>0$ such that $|z-z^{'}|<\delta \implies |f(z)-f(z^{'})|<\epsilon$. Now choose a small $\epsilon>0$, and a $\delta>0$ such that $0<\delta < \epsilon$ and $|z-z^{'}|<\delta \implies |f(z)-
f(z^{'})|<\epsilon$. Pick points $\alpha$ and $\beta$ on the line segments $[a,b]$ and $[a,c]$, respectively, so that $|c-\alpha|<\delta$ and $|b-\beta|<\delta$. Let $T^{'}$ and $Q$ be the paths $T^{'}=[\alpha,\beta,a,\alpha]$ and $Q=[\alpha,c,b,\beta,\alpha]$ as in Figure \ref{fig:Figure 1} below.
\renewcommand{\thefigure}{\thesection.\arabic{figure}} 
\captionsetup{labelsep=none}
\begin{figure}
\centering
\includegraphics{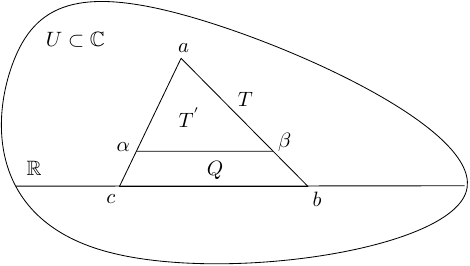}
\caption{\label{fig:Figure 1}}
\end{figure}
Then 
\begin{equation*}
\int_T f=\int_{T^{'}} f+\int_{Q}f.
\end{equation*}
However, since $T^{'}$ and its interior are contained in $U^+$, $f$ is holomorphic there, and therefore $\int_{T^{'}} f=0$. 

We now approximate $\int_{Q} f$. First, note that, for $t\in [0,1]$,
\begin{equation*}
  \left|[t\beta + (1-t)\alpha]-[tb + (1-t)c]\right|<\delta
\end{equation*}
and therefore
\begin{equation*}
\left|f(t\beta + (1-t)\alpha)-f(tb + (1-t)c)\right|<\epsilon.
\end{equation*}
Now let $M=\text{max }\{|f(z)| \text{ }|\, z\in \Delta\}$, and let $l=$ the length of the perimeter of $T$. Then
\begin{align*}
\left|\int_{[\alpha,c]} f\right| &\leq M\left|c-\alpha\right| \leq M\delta\\
\left|\int_{[\beta,b]} f\right| &\leq M\left|b-\beta\right| \leq M\delta,
\end{align*}
and
\begin{align*}
\left|\int_{[b,c]}f+\int_{[\beta,\alpha]}f\right| & = \left| (b-c)\int_0^1
f(tb + (1-t)c) dt \right.\\
& \left. + (\alpha-\beta)\int_0^1 f(t\beta+(1-t)\alpha) dt
\vphantom{\int_{[\beta,\alpha]}}\right|\\
& \leq |b-c|\left|\int_0^1 f(tb + (1-t)c)-f(t\beta+(1-t)\alpha)\right|\\
& \vphantom{\int_{[\beta,\alpha]}}+ |(b-c)-(\beta-\alpha)|\left|\int_0^1
f(t\beta + (1-t)\alpha) dt \right|\\
& \vphantom{\int_{[\beta,\alpha]}}\leq \epsilon|b-c| + M|(b-\beta)+(c-\alpha)|\\
& \vphantom{\int_{[\beta,\alpha]}}\leq \epsilon l + 2M\delta.
\end{align*}
Therefore,
\begin{equation*}
\left|\int_T f\right| \leq \epsilon l + 4M\delta.
\end{equation*} }
Since $\epsilon$ is arbitrary, and we may choose $\delta < \epsilon$, it follows that $\int_T f=0$, and therefore $f$ is holomorphic. From this we conclude that $\tilde{\alpha}$ is holomorphic as well.
 
We have now shown the existence of a holomorphic map $\tilde{\alpha}$ that verifies $\phi\circ \tilde{\alpha} \circ c = \tilde{\alpha}$ and $Im(\tilde{\alpha})=Im(u)$. Now let $S=C\cap L$.

We now remark that the cohomology class $[\tilde{\alpha}^{*}\omega]\in
H^2(\mathcal{L}(0))$ is determined by the integral
$\int_{\mathbb{C}P^{1}}\tilde{\alpha}^{*}\omega$, where here we understand
$\mathbb{C}P^{1}=\mathcal{L}(0)$. 
Therefore, for $\lambda^2:=\int_{\mathbb{C}P^{1}}\tilde{\alpha}^{*}\omega$, the
form $\rho(1,\lambda)$ is in the same cohomology class. By Proposition \ref{prop:Real-Moser-Stability}, there exists a diffeomorphism $\beta_0:\mathcal{L}(0)\to \mathcal{L}(0)$ such that $\beta_0^*\tilde{\alpha}^*\omega=\rho(1,\lambda)$. Let $\gamma_0:=\tilde{\alpha}\circ \beta_0$.

Now, by Lemma \ref{lem:Non-trivial bundle in Lagrangian} the normal bundle of $S$ in $TL$ is non-trivial. Consider the bundles $\gamma_0^*(\nu(C))$ and $\nu(\mathcal{L}(0))$, where $\nu(\cdot)$ denotes the normal bundle of the submanifold in question. Since the
Chern class of $C$ is $2$, the Maslov index of the two disks $D_1$ and
$D_2$ in $C$ with boundary on $L$ is $1$, and the restriction of
$\nu(C)$ to $L$ is non-trivial, then by Theorem C.3.7 in McDuff and
Salamon \cite{McDuff_Salamon_2004}, there is a (complex) isomorphism
$\Phi$ between the bundles $\gamma_0^*(\nu(D_1),TL\cap\nu(D_1))$ and
$\nu(\mathcal{L}(0)^{+},\mathcal{R}(0))$,
where $\mathcal{L}(0)^{\pm}$ denote the upper and lower hemispheres
of $\mathcal{L}(0)$, respectively.

Now note that
$\tilde{\phi}_*\Phi\tilde{c}_*$ gives an isomorphism of
$\tilde{\alpha}^*(\nu(D_2),TL\cap\nu(D_2))$ and
$\nu(\mathcal{L}(0)^{-},\mathcal{R}(0))$, and therefore the
map
\begin{equation*}
\Psi = \begin{cases} \Phi & (x,v)  \in \nu(\mathcal{L}(0)^{+}) \\
                     \tilde{\phi}_*\Phi \tilde{c}_* & (x,v) \in \nu(\mathcal{L}(0)^{-}) 
\end{cases}
\end{equation*}
is a complex equivariant isomorphism from $\nu(\mathcal{L}(0)) \to \alpha^*\nu(C)$.

Furthermore, since $\Psi$ is a complex bundle isomorphism, it is
symplectic as well. It therefore follows from Proposition
\ref{thm:Real symplectic neighborhood}, that for some $\delta > 0$, we
can find a $\mathbb{Z}_{2}$-equivariant map
$\beta_{\lambda}:\mathcal{L}(\delta)\to M$ such that
$\beta_{\lambda}^*\omega=\rho(1,\lambda)$ which restricts to
the symplectomorphism $\gamma_0:\mathcal{L}(0)\to C\subset M$. We may
now construct the blow-down by theorem \ref{thm:Blow-down} using the
equivariant symplectic map $\beta:\mathcal{L}(\delta)\to M$.
\end{proof}

\section{Applications to Real Packing}
\label{sec:Applications}
We now apply our results to problems of real packing in
symplectic four manifolds. 
 That is, given a real, symplectic $4$-manifold $(M,\omega,\phi)$, we wish to
know the quantity 
\begin{equation*}
p_{L,k}=\mbox{sup}_{\psi,r}\frac{\text{Vol
}\psi\left(\coprod_{i=1}^{k}B_{i}(r)\right)}{\text{Vol }M},
\end{equation*} 
where $\psi:\coprod_{i=1}^{k}B_{i}(r)\hookrightarrow M^{4}$
is a symplectic embedding such that the preimage
$\psi^{-1}(Fix(\phi))=\coprod_{i=1}^kB_{i,\mathbb{R}}(r)$. We will, in
particular, treat the cases $(\mathbb{C}P^{2}, \sigma,
\phi)$ and $(S^2 \times S^2,
\sigma_{S^2} \oplus \sigma_{S^2}, \phi')$ with the canonical
real structures $\phi$, $\phi'$, where $\mathbb{R}P^{n}=\mbox{Fix}(\phi)$ and
the direct sum of the equators
$S^1 \times S^1 = \mbox{Fix}(\phi')$. In particular, we will
prove

\begin{thm*}[Theorem \ref{thm:Relative packing rp2}]
For the pair $(\mathbb{C}P^{2},\mathbb{R}P^{2})$ with the standard
  symplectic form and real structure, the relative packing numbers
  $p_{\mathbb{R}P^2,k}$ are equal to the absolute packing numbers
for
  $\mathbb{C}P^{2}$.

\end{thm*}

\begin{thm}
\label{thm:Relative packing S1 x S1}
Let $\phi:S^2 \to S^2$ be the reflection on $S^2$ which sends the
upper hemisphere to the lower one and fixes the equator.
The relative packing numbers for $(S^2\times S^2, \sigma_{S^2} \oplus
\sigma_{S^2}, \phi \oplus \phi)$ are equal to the absolute packing numbers for
$(S^2\times S^2), \sigma_{S^2} \oplus
\sigma_{S^2}).$
\end{thm}

The packing numbers for $(\mathbb{C}P^2, \sigma)$ and $(S^2 \times S^2,
\sigma_{S^2} \oplus \sigma_{S^2})$ are given in Tables
\ref{table:Packing numbers RP2} and \ref{table:Packing numbers S2 x S2} below,
quoted from \cite{Biran_1997}.
\renewcommand{\thetable}{\thesection.\arabic{table}}
\begin{table}[h]
\renewcommand{\arraystretch}{1.75}
\centering
\begin{tabular}{| l | c c c c c c c c c |}
\hline
$k$ & 1 & 2 & 3 & 4 & 5 & 6 & 7 & 8 & $\geq$ 9 \\ \hline
$p_{\mathbb{R}P^2,k}$ & 1 & $\frac{1}{2}$ & $\frac{3}{4}$ & 1 & $\frac{4}{5}$ &
$\frac{24}{25}$ & $\frac{63}{64}$ & $\frac{288}{289}$ & 1 \\ \hline
\end{tabular}
\caption{$p_{\mathbb{R}P^2,k}(\mathbb{C}P^2,\sigma)=p_k(\mathbb{C}P^2, \sigma)$
\label{table:Packing numbers RP2}}
\end{table}
\begin{table}[h]
\renewcommand{\arraystretch}{1.75}
\centering
\begin{tabular}{| l | c c c c c c c c |}
\hline
$k$ & 1 & 2 & 3 & 4 & 5 & 6 & 7 & $\geq 8$ \\ \hline
$p_{S^1 \times S^1,k}$ & $\frac{1}{2}$ & $1$ & $\frac{2}{3}$ &
$\frac{8}{9}$ & $\frac{9}{10}$ &
$\frac{48}{49}$ & $\frac{224}{225}$ & $1$ \\ \hline
\end{tabular}
\caption{$p_{S^1 \times S^1,k}(S^2 \times S^2, \omega \oplus
\omega)=p_k(S^2 \times S^2, \omega \otimes \omega)$}
\label{table:Packing numbers S2 x S2}
\end{table}

Our basic strategy follows \cite{McDuff_Polterovich_1994} and
\cite{Biran_1997}. We create the blow up $\tilde{M}$ of $M$ using
symplectic and holomorphic embeddings of small balls, and we determine which
classes in $H^2(M;\mathbb{R})$ are respresented by symplectic forms, in this
case trying to increase the area of the exceptional divisors as much as
possible. 

The following proposition shows that, once we have altered the form on
$\tilde{M}$ to increase the area of the exceptional divisors, we are able to
show the existence of larger ball embeddings $M$. It is an
adaptation of Proposition 2.1.C in
McDuff and Polterovich \cite{McDuff_Polterovich_1994} to real
symplectic manifolds.

\begin{prop}
  \label{prop:Blow-down ha ha ha}
  Let $(M,\omega,\phi)$ be a real symplectic manifold, and let $J$ be
  an $\omega$-tame almost complex structure which is symmetrically
  integrable around a set of $k$ points $I=\{p_1,\dots,p_k\}\subset
  L$, where $L=Fix(\phi)$. Suppose that for some set of real numbers
  $\kappa_q> 0$, $q\in \{1,\dots,k\}$, there exists a real symplectic
  and holomorphic embedding
  \begin{equation*}
    \psi=\coprod_{q=1}^{k}\psi_{q}:\coprod(B(1+2\epsilon_{q}),B_{\mathbb{R}}(1+2\epsilon_{q}),\kappa_{q}^{2}\omega_{0},i,c)\to(M,L,\omega,J,\phi)
  \end{equation*} such that $\psi_{q}(0)=p_{q}$. Let $\Pi:\tilde{M}\to M$ denote the real symplectic blow-up of $(M,L)$ relative to $\psi$, and let $\tilde{J}$, $\tilde{\omega}$, and $\tilde{\phi}$ be the complex, symplectic, and real structures, respectively, on
  $\tilde{M}$ constructed from $J$, $\omega$, and $\phi$ by blowing-up $M$. Let $C_q,q\in \{1,...,k\}$ denote the exceptional curves $\Pi^{-1}(\psi_q(0))$ added in the blow-up, and let $e_{q}\in H^2(M;\mathbb{Z})$ denote the Poincar\'e duals of the homology classes $[C_{q}]\in H_2(M;\mathbb{Z})$.  

  Suppose, furthermore, that there exists a smooth family of
  symplectic forms $\tilde{\omega}_{t}$ on $\tilde{M}$ such that

  \begin{enumerate}

  \item $\tilde{\omega}_0=\tilde{\omega}$ is obtained by a real blow up
    relative to the embedding $\psi$.
  \item $\tilde{\omega}_{0}$ tames $\tilde{J}$,

  \item For all $q\in \{1,...,k\}$, $\tilde{\omega}_{t}|_{C_q}$, the
    restriction of $\tilde{\omega}_t$ to the exceptional divisors
    $\{C_q\}_{q=1}^{k}$ added in the blow-up, tames
    $\tilde{J}|_{C_q}$,

  \item $\phi^{*}\tilde{\omega}_{t}=-\tilde{\omega}_{t}$, so that
    $\tilde{L}=\Pi^{-1}(L)$ is Lagrangian for each of the forms
    $\tilde{\omega}_{t},$ and

  \item$[\tilde{\omega}_{t}]=[\Pi^{*}\omega]-\sum_{i=1}^{k}\lambda_{i}^{2}(t)e_{q}$
    for positive constants $\lambda_{q}(t),$ $0\leq t\leq1.$

  \end{enumerate}

  Then $(M,L,\omega,\phi)$ admits a real symplectic embedding of $k$
  disjoint standard symplectic balls of radii $\lambda_{q}(1)$,
  $q\in\{1,...,k\}$.\end{prop}
\begin{proof}
  Since $\tilde{M}$ is the real symplectic and holomorphic blow-up at $k$ real points
  of $(M,L,J,\phi,\omega)$, then, according to the construction in the
  proof of Proposition \ref{prop:Real blow-up}, there exists a real
  symplectic and holomorphic embedding
 
\begin{equation*}\tilde{\psi}=\coprod_{q=1}^{k}\tilde{\psi}_{q}:\coprod(\mathcal
{L}(1+2\epsilon_{q}),\mathcal{R}(1+2\epsilon_{q}),\rho(1,\kappa_{q}),i,\tilde{c}
)\to(\tilde{M},\tilde{L},\tilde{\omega}_0,J,\tilde{\phi})
  \end{equation*}
  
  We will show that for each $q$ there exists a family of equivariant
  diffeomorphisms $g_{t}:\tilde{M} \to \tilde{M},t\in[0,1]$ with the
  following properties:

  \begin{enumerate}

  \item $g_{0}=Id$

  \item There exists a $\delta \in \mathbb{R}$, $0 < \delta <
    1+2\epsilon$, such that, for all $t$,
    $\tilde{\psi}_{q}^*g_{t}^{*}\tilde{\omega}_{t}=\rho(1,\lambda_{q}(t))$
    on $\mathcal{L}(\delta)$

  \item $g_{t}\circ \tilde{\phi}=\tilde{\phi}\circ g_{t}$,
    $g_{t}(Im(\tilde{\psi})) = Im(\tilde{\psi})$, and
    $g_{t}(\tilde{\psi}_q(\mathcal{L}(0)))=\tilde{\psi}_q(\mathcal{L}(0))$.

  \end{enumerate}

  To see this, first note that the $\lambda_{i}(t)$ satisfy the
  equation
\begin{equation*}
\int_{\mathcal{L}(0)}\tilde{\psi}_q^*\tilde{\omega}_t=\lambda_{
i } (t)^2\int_ { \mathcal{L}(0)}\sigma=\lambda_{i}(t)^2,
\end{equation*}
  so $\tilde{\psi}_{q}^*\tilde{\omega}_t$ is in the same cohomology
  class on $\mathcal{L}(0)$ as $\rho(1,\lambda_{q}(t))$.  Then since
  both of these forms tame $\tilde{i}$ on $\mathcal{L}(0)$, the forms
  $s\rho(1,\lambda_{q}(t)) + (1-s)\tilde{\psi}_{q}^*\tilde{\omega}_t$
  are non-degenerate for all $s\in [0,1]$. Therefore, by Proposition
  \ref{prop:Real-Moser-Stability}, for each $t$, there exists an
  equivariant symplectomorphism
  $F_{q,t}:(\mathcal{L}(0),\rho(1,\lambda(t)))\to
  (\mathcal{L}(0),\tilde{\psi}^*\tilde{\omega}_t)$ such that
  $\tilde{c}\circ F_{q,t}=F_{q,t}\circ\tilde{c}$ and
  $F_{q,t}^{*}\tilde{\psi}^*\tilde{\omega}_t= \rho(1,\lambda_q(t))$ on
  $\mathcal{L}(0)$. Since $\tilde{\omega}_t$ and
  $\rho(1,\lambda_q(t))$ form smooth families of forms, the $F_{q,t}$
  must also be smooth with respect to $t$ as well.

  We extend the $F_{q,t}$ to an isomorphism of the normal bundle $\nu$
  of $\mathcal{L}(0)$ in $\mathcal{L}(1+2\epsilon)$ by defining
  $f_{q,t}:\nu \to \nu$ by $f_{q,t}(z,v)=(F_{q,t}(z), v)$. Since the
  restriction of both $\rho(1,\lambda(t))$ and $\rho(1,\kappa_q)=\tilde{\psi}_q^*\tilde{\omega}$ to the fiber $\nu_z$ is
  $\omega_0$, this isomorphism is both equivariant and
  symplectic. Then, by Theorem \ref{thm:Real symplectic neighborhood},
  $F_{q,t}$ extends to an equivariant symplectomorphism $G_{q,t}$ of a
  neighborhood $\mathcal{N}_{0,t}$ of $\mathcal{L}(0)$ in
  $(\mathcal{L}(1+2\epsilon),\rho(1,\lambda(t)))$ to a neighborhood
  $\mathcal{N}_{1,t}$ of $\mathcal{L}(0)$ in
  $(\mathcal{L}(1+2\epsilon),\tilde{\psi}^*\tilde{\omega}_t)$. Let
  $\delta_q \in \mathbb{R}$, $0<\delta_q<1+2\epsilon$ be such that
  $\mathcal{L}(\delta_q)\subset \mathcal{N}_{0,t}$ and for all $t\in
  [0,1]$. Note now that the $G_{q,t}|_{\mathcal{L}(\delta_q)}$ also
  form a smooth family of maps with respect to $t$. Extend $G_{q,t}$
  to a smooth family of equivariant differentiable maps from
  $\mathcal{L}(1+2\epsilon)\to \mathcal{L}(1+2\epsilon)$ which is the
  identity in a neighborhood of the boundary.

  Define $g_{q,t}=\tilde{\psi}_q \circ G_{q,t} \circ
  \tilde{\psi}^{-1}$, extend the $g_{q,t}$ to all of $\tilde{M}$ by
  the identity outside $\tilde{\psi} \left(\coprod_{q=1}^k
    \mathcal{L}(1+2\epsilon)\right)$, and denote the extension by
  $g_t$. Then
  $\tilde{\psi}^*g_{t}^*\tilde{\omega}_t=\rho(1,\lambda_q(t))$ on
  $\mathcal{L}(\delta_q)$, making $\tilde{\psi}$ a symplectomorphism
  with respect to the forms $g^*_{t}\tilde{\omega}$ for all $t$.

  Now let $\delta=min\{\delta_q\}_{q=1}^k$, and let $(M,\omega_{t})$
  be the blow-down of $(\tilde{M},g^*_t \tilde{\omega}_{t})$ using the
  symplectic and holomorphic embedding $\tilde{\psi}|_{\coprod_{q=1}^k
    \mathcal{L}_q(\delta)}$. Note that by Theorem \ref{thm:Blow-down},
  each form of the family $\omega_t$ is cohomologous to
  $\omega_{0}$. Also, $\omega_{0}$ tames $J$ and
  $[\omega_{0}]=[\omega]$, and therefore all the forms $\omega_t$ and
  $s\omega_0+(1-s)\omega$, $t,s\in [0,1]$, are symplectic and in the
  same cohomology class. Furthermore, note that $\frac{d}{dt}\omega_t$
  is supported on a finite union of balls, and is therefore
  exact. Therefore, by Proposition \ref{prop:Real-Moser-Stability} and
  Lemma \ref{lem:Moser's trick}, there exists a family of equivariant
  diffeomorphisms $H_r:M\to M$, $r\in [0,1]$, such that $H_0=Id$ and
  $H_1^{*}\omega=\omega_1$.  Since $(M,\omega_{1})$ admits a real
  symplectic embedding of
  $\coprod_{q=1}^{k}(B(1+2\epsilon),\lambda_q\omega_{st})$, where
  $\omega_{st}$ here is the standard symplectic form on
  $B(1+2\epsilon)$, this completes the proof.\end{proof}

The following corollary is an easy consequence.

\begin{cor}
  \label{cor:Packing}Let $(M,\omega,\phi)$ be a real symplectic
  manifold with almost complex structure $J$ which tames $\omega$ and
  is symmetrically integrable around the points
  $\{p_1,\dots,p_k\}$. Let $(\tilde{M},\tilde{\omega}_0,\tilde{\phi})$
  be a real manifold obtained by blowing up a real symplectic and
  holomorphic embedding $\psi$ of balls of radii $\kappa>0$, $\kappa$
  small, and let $\tilde{J}$ be the almost complex structure created
  in the blow-up.

  Now suppose that there exists a real symplectic form $\tilde{\omega}$ on
  $\tilde{M}$ such that $\tilde{\omega}$ tames the almost complex
  structure $\tilde{J}$ on $\tilde{M}$ and represents the cohomology class
  \[
  [\tilde{\omega}]=[\Pi^{*}\omega]-\sum_{i=1}^{k}\pi\lambda_{i}^{2}e_{i}.\]
  
  Then $(M,L,\omega)$ admits a real symplectic embedding of $k$
  disjoint standard symplectic balls of radii
  $\lambda_{1},...,\lambda_{k}$.
\end{cor}
\begin{proof}
  By Proposition \ref{prop:Real blow-up}, the blow-up
  $\tilde{\omega}_0$ relative to $\psi$ tames $\tilde{J}$, and
  therefore the forms $\omega_s:=s\tilde{\omega}_0+(1-s)\tilde{\omega}$
  tame $\tilde{J}$ as well, so the family of forms $\omega_s$
  satisfies the hypothesis of Proposition \ref{prop:Blow-down ha ha
    ha}. The conclusion follows.
\end{proof}

We now prove a lemma which allows us to symmetrize a symplectic form
given a real-structure and a tame, symmetric pseudo-holomorphic
structure $J$.

\begin{lem}
  \label{lem:Form symmetrization}Let $(M,\omega)$ be a symplectic
  manifold, and let $J$ be an almost complex structure tamed by
  $\omega$. Suppose there exists an anti-holomorphic involution $\phi$ (a map
  $\phi:M\to M$ such that $\phi^{2}=Id$ and
  $\phi_{*}J\phi_{*}=-J$). Then the $2$-form
  $\overline{\omega}=\frac{1}{2}(\omega-\phi^{*}\omega)$ has the
  properties \end{lem}
\begin{enumerate}
\item $\overline{\omega}$ is symplectic
\item $\phi^{*}\overline{\omega}=-\overline{\omega}$
\item $\overline{\omega}$ tames $J$
\end{enumerate}
\begin{proof}
  Since $\omega$ tames $J$, we have that
\begin{equation*}
\overline{\omega}=\frac{1}{2}(\omega(v,Jv)-\omega(\phi_{*}v,\phi_*Jv))=\frac{1}{2}(\omega(v,Jv)+\omega(\phi_* v,J\phi_* v)>0,
\end{equation*} and
  therefore $\overline{\omega}$ tames $J$. It follows that $\overline{\omega}$
  is non-degenerate. Furthermore,
  $d\overline{\omega}=\frac{1}{2}d(\omega-\phi^{*}\omega))=0$, so
  $\overline{\omega}$ is closed, and therefore symplectic. \end{proof}

\subsection{Stability of Real Packing}

We begin with the question of packing stability. Specifically, we show that for
a
real, 
rank-$1$ symplectic $4$-manifold of non-Seiberg-Witten simple type,
the real packing numbers $p_{\mathbb{R},k}$ stabilize for large $k$, extending a
theorem of Biran\cite{Biran_1997} to our setting. We begin by recalling several
definitions, and we then state our theorem.

\begin{defn}
  We say that a symplectic manifold $(M,\omega)$ is of
  \emph{Seiberg-Witten simple type} if the only non-zero Seiberg-Witten
  invariants are in dimension 0. Otherwise, we say that $(M,\omega)$ is of
  \emph{non-Seiberg-Witten simple type}. We denote by
  $\mathcal{C}$ the class of symplectic manifolds which are of
  non-Seiberg-Witten simple type, and we let
  $\mathcal{C}_{\mathbb{R}}\subset \mathcal{C}$ denote the real
  symplectic manifolds in class $\mathcal{C}$. (See Taubes
  \cite{Taubes_1995} for a definition and overview of the
  Seiberg-Witten invariants.)
\end{defn}

\begin{rem}
As noted in Biran\cite{Biran_1997}, $\mathcal{C}$ contains
\begin{enumerate}
\item Symplectic
manifolds with $b_{2}^{+}=1$ and $b_{1}=0$, and
\item Ruled symplectic manifolds and
their blow-ups.
\end{enumerate}
\end{rem}

\begin{defn}
We say that a differential form $\omega$ on $M$ is \emph{rank-$1$} if
$[\omega] = c[\omega']$, where $c\in \mathbb{R}$ and $[\omega'] \in
H^{*}(M,\mathbb{Q})$.
\end{defn}

\begin{defn}
Let $(M,\omega)$ be a closed symplectic $4$-manifold, and let $D_{\omega}$
denote the set
\begin{equation*}
D_{\omega}:=\{ B\in H_2(M;\mathbb{Z})|\omega(B)>0, c_1(B)\geq 2, B\cdot B \geq
0\}.
\end{equation*}
Define $d_{\omega}\in \mathbb{R}$ to be
\begin{equation*}
d_{\omega}:=\inf_{B\in D_{\omega}}\frac{\omega(B)}{c_1(B)} \in [0,\infty],
\end{equation*}
where we adopt the convention that $\inf \emptyset = \infty$.
\end{defn}

\begin{thm}
\label{thm:Symmetric Biran1}
Let $(M,\omega,\phi)$ be a real symplectic $4$-manifold in
the class
$\mathcal{C}$ where $\omega$ is rank-$1$ and $Fix(\phi)=L$. Define
$Vol(M)=\int_M \omega^2$, suppose that $0<d_{\omega}\leq \infty$, and let
  $\lambda_1,\dots,\lambda_n < \sqrt{d_{\omega}}$ be positive numbers
  which satisfy
\begin{equation*}
\sum_{q=1}^{n}\lambda^4_q < \text{Vol}(M,\omega).
\end{equation*} Then the manifold $(M,\pi\omega,L)$ admits a real symplectic
packing by $n$ balls of radii $\lambda_1,\dots,\lambda_n$. In
particular, if 
\begin{equation*}
n \geq \frac{Vol(M,\omega)}{d^2_{\omega}}
\end{equation*}
then there exists a full real packing of $(M,\pi\omega,L)$ by n equal
balls, and $p_{L,n}(M) = p_n(M)$, i.e. the relative and absolute packing
numbers for $n$ balls are equal.
\end{thm}

As a corollary, we have

\begin{cor}
The relative packing numbers $p_{\mathbb{R}P^2,k}$
for $(\mathbb{C}P^2,\mathbb{R}P^2,\sigma)$ are
equal to the absolute packing numbers $p_k$ for $(\mathbb{C}P^2,\sigma)$ for
all $k \geq 9$.
\end{cor}

\begin{proof}
  Note first that $\sigma$ is rank-$1$, and that
  $d_{\left(\frac{1}{\pi}\sigma\right)}=\frac{1}{3}$. Therefore, by Theorem
  \ref{thm:Symmetric Biran1}, there is a full real packing of
  $(\mathbb{C}P^2,\mathbb{R}P^2,\sigma)$ for $k\geq 9$.
\end{proof}

For the proof of Theorem \ref{thm:Symmetric Biran1}, we will appeal to the
following result of Biran.

\begin{thm}[Biran \cite{Biran_1997}, Theorem 4.1.A]
\label{thm:Biran form}
 Let $(M,\omega)$
  be a closed symplectic $4$-manifold in the class
  $\mathcal{C}$. Suppose that $0<d_{\omega}\leq \infty$ and let
  $\lambda_1,\dots,\lambda_n < \sqrt{d_{\omega}}$ be positive numbers
  which satisfy
\begin{equation*}
\sum_{q=1}^{N}\lambda^4_q < \text{Vol}(M,\omega).
\end{equation*}
Denote by $\Pi: (\tilde{M},\tilde{\omega})\to (M,\omega)$ a complex blow-up of $(M,\omega)$ at $n$
distinct points. Then the cohomology class
\begin{equation*}
[\Pi^*\omega]-\sum_{q=1}^{N} \lambda^2_q e_q\in H^{2}(\tilde{M};\mathbb{R})
\end{equation*}
admits a symplectic representative $\tilde{\omega}^{'}$.
\end{thm}

\begin{rem}
\label{rem:Inflation and J}
If, in addition to satisfying the hypothesis of Theorem \ref{thm:Biran
  form}, suppose also that $\tilde{\omega}$ is rank-$1$ and tames an almost complex
structure $\tilde{J}$ on $\tilde{M}$. Then, in the proof of Theorem
\ref{thm:Biran form} given by Biran\cite{Biran_1997}, we may obtain the
form $\tilde{\omega}^{'}$ by inflating along a single curve using the
version of symplectic inflation given in McDuff\cite{McDuff_2001},
Lemma 3.1. This allows us to make $\tilde{\omega}^{'}$ tame $\tilde{J}$ as
well.
\end{rem}

We now prove Theorem \ref{thm:Symmetric Biran1}.

\begin{proof}[Proof of Theorem \ref{thm:Symmetric Biran1}]
  Let $L:=Fix(\phi)$, and choose an almost complex structure $J$ which
  is symmetrically integrable around the points $\{p_1,\dots,p_k\} \in
  L$. By Theorem \ref{thm:Blow-up}, we construct the real blow up
  of $(M,\omega)$ with respect to small, real symplectic balls of
  radius $\epsilon>0$ centered at the points $\{p_i\}_{i=1}^k$, and we
  obtain the real symplectic manifold
  $(\tilde{M},\tilde{\omega}_{\epsilon},\tilde{\phi})$ and an almost
  complex structure $\tilde{J}$ which is tamed by
  $\tilde{\omega}_{\epsilon}$. By Remark \ref{rem:Inflation and J}, for every
  $\lambda >0$ such that $\lambda < \sqrt{d_{\omega}}$ and
  $\sum_{q=1}^{N}\lambda^4_q < \text{Vol}(M,\omega)$, there exists a
  symplectic form $\tilde{\omega}_{\lambda}$ which tames $\tilde{J}$
  and represents the cohomology class $[\Pi^*\omega]-\sum_{q=1}^{N}
  \lambda^2_q e_q$. Now note that by Lemma \ref{lem:Form symmetrization},
  the family of symplectic forms
  $\tilde{\omega}_{\mathbb{R},\lambda}=\frac{1}{2}(\tilde{\omega}_{\lambda} -
  \phi^{*}\tilde{\omega}_{\lambda})$ satisfies
  $\phi^{*}\tilde{\omega}_{\mathbb{R},\lambda} =
  -\tilde{\omega}_{\mathbb{R},\lambda}$, and by Lemma \ref{lem:Form symmetrization} these forms also tame
  $\tilde{J}$. Therefore, by Corollary \ref{cor:Packing},
  $(M,\omega,\phi)$ admits a real symplectic embedding of balls of
  radius $\lambda$, which proves the theorem.
\end{proof}

\subsection{Obstructions to real packing}
We will now show that the real packing numbers below the stable range for real
rank-$1$ symplectic manifolds in the class $\mathcal{C}$ are also identical to
the absolute packing numbers.

This follows from a refined version of Theorem \ref{thm:Symmetric
Biran1}, following Theorem 6.A in Biran \cite{Biran_1997}. We begin with some
definitions.

\begin{defn}
 Let $(M^{4},\omega,\phi)$ be a real symplectic four-manifold in the class
$\mathcal{C}$. Let $\Pi:\tilde{M}_k \to M$ be a real blow-up at $k$ points in
$L:=Fix(\phi)$, and let $\mathcal{E}_k\subseteq H_2(\tilde{M}_k;\mathbb{Z})$ be
the subset of homology classes representing exceptional spheres in
$\tilde{M}_k$. Let $\Pi_*: H_2(\tilde{M}_k;\mathbb{Z}) \to H_2(M;\mathbb{Z})$ be
the projection induced by $\Pi$, and let 
\begin{align*}
\mathcal{E}^{'} =& \Pi_* (\mathcal{E}_k)\backslash \lbrace 0 \rbrace
\subset H_2(M;\mathbb{Z}) \\
d^{'}_{k} =& \inf_{B \in \mathcal{E}^{'}_{k}} \frac{\omega(B)}{c_{1}(B)-1}.
\end{align*}
\end{defn}

\begin{thm}
\label{thm:Symmetric Biran 2}
Let $(M^4,\omega,\phi)$ be a real, rank-$1$ symplectic four-manifold
in the class $\mathcal{C}$. Then
\begin{align*}
\lambda_{\R,k}^{2} =& \min \left\lbrace \frac{1}{\pi}
d^{'}_{k},\frac{1}{\pi} \sqrt{\frac {2 \text{Vol}
(M,\omega)}{k}}\right\rbrace\\
p_{\R,k} =& \min \left\lbrace \frac{k d^{'2}_k}{2
\text{Vol}(M,\omega)},1 \right\rbrace 
\end{align*}
In particular, the real and absolute packing numbers are equal.
\end{thm}

The proof is an adaptaton to our setting of the proof of Theorem
6.A of Biran \cite{Biran_1997}.

We will need the following lemma. We state here the version quoted in Biran
\cite{Biran_1997}. Part \ref{lem:Class of exceptional holomorphic curves:item 1}
follows from Lemma 3.1 in McDuff\cite{McDuff_1990} (see also Proposition 2.3.A
in McDuff and Polterovich\cite{McDuff_Polterovich_1994}), and Part
\ref{lem:Class of exceptional holomorphic curves:item 2} follows from the same
lemma cited above and the positivity of intersections for $J$-holomorphic
curves 

\begin{lem}
\label{lem:Class of exceptional holomorphic curves}
 Let $(M^4,\omega)$ be a closed symplectic $4$-manifold. Denote by
$\mathcal{E}$ the set of all homology classes which can be represented by
$\omega$-symplectic exceptional spheres. Then
\begin{enumerate}
 \item \label{lem:Class of exceptional holomorphic curves:item 1} $\mathcal{E}$
depends only on the deformation class of $\omega$ \item \label{lem:Class of
exceptional holomorphic curves:item 2} If $E',E''$ are distinct classes in
$\mathcal{E}$, then $E'\cdot E''\geq 0$.
\end{enumerate}
\end{lem}

\begin{proof}[Proof of \ref{thm:Symmetric Biran 2}] We begin by remarking that
the upper bounds on the absolute packing numbers proven in Theorem 6.A of Biran
\cite{Biran_1997} are also upper bounds on the relative packing numbers.
Therefore
\begin{equation*}
p_{\R,k} \leq \min \left\lbrace \frac{k d^{'2}_k}{2
\text{Vol}(M,\omega)},1 \right\rbrace.
\end{equation*}
To show that the lower bounds are the same, let $L:= Fix(\phi)$, and use
\ref{thm:Blow-up} to construct the real blow-up of $M$ with respect to the
symplectic embedding of $k$ balls or radius $\epsilon$. Let
$\tilde{\omega}_{\epsilon}$ denote the resulting form on the blow-up. Choose
$\lambda > 0$ such that 
\begin{equation*}
\label{eq:lambda less than}
 \lambda^2 < \min \left\lbrace \frac{1}{\pi}
d^{'}_{k},\frac{1}{\pi} \sqrt{\frac {2 \text{Vol}
(M,\omega)}{k}}\right\rbrace
\end{equation*}
Consider the cohomology class
\begin{equation*}
 a = [\Pi^{*}\omega]-\pi\lambda^2 \sum_{j=1}^{k} e_{j}
\end{equation*}
and let $A$ be the Poincar\'e dual of $a$. Assume without loss of generality
that $a$ is a rank-$1$ cohomology class. It is clear that $A\cdot A > 0$, and,
by taking $\epsilon$ small enough, $\tilde{\omega}_{\epsilon} >0$ as well. Let
$\mathcal{E}_k$ be the set of homology classes in $H_2(\tilde{M};\mathbb{Z})$
that can be represented by $\tilde{\omega}_{\epsilon}$-symplectic exceptional
spheres. 

We claim that, for any $E\in \mathcal{N}$, $A\cdot E >0$. To see this, let
$E:=B - \sum_{j=1}^{k} m_{j} E_{j}$. Suppose first that $\omega(B) = 0$ and
$B\neq 0$. Then $E \neq E_j \forall j$, so by Lemma \ref{lem:Class of
exceptional holomorphic curves}, $E\cdot E_j \geq 0$, which implies that $m_j
\geq 0$ for all $j$. If every $m_j$ is $0$, then $E = B$ and
$\tilde{\omega}_{\epsilon}(E) = \omega(B) = 0$, which is a contradiction, so
there
exists at least one $j\in \{1,\dots,k\}$ for which $m_j \gneq 0$. This implies
that $\tilde{\omega}_{\epsilon}(E) = 0-\epsilon \sum_{j=1}^k m_j < 0$, a
contradiction. Therefore, if $\omega(B)=0$, then $B = 0$ as well. In this case,
it follows easily from $E\cdot E=-1$ that $E = E_j$ for some $1\leq j \leq k$,
and therefore $A\cdot E > 0$.

If $B\neq 0$, then $E\neq E_j \forall j$, and it follows from Lemma
\ref{lem:Class of exceptional holomorphic curves}, part
\ref{lem:Class of exceptional holomorphic curves:item 2} that $m_{j} \geq 0$ for
$1\leq j \leq k$. We
therefore have $1 = c_1(E) = c_1(B) - \sum_{j=1}^k m_j$,
so
\begin{align*}
 A \cdot E &= \omega(B) - \pi \lambda^2 \sum_{j=1}^k m_{j} \\
 &= \omega(B) - \pi \lambda^2(c_{1}(B) - 1) > 0,
\end{align*}
where the last inequality follows because $B\in \mathcal{E}^{'}$ by
definition and $\pi \lambda^{2} < d^{'}_{k}$ by
hypothesis. This proves the claim.

It now follows from Remark \ref{rem:Inflation and J} and Lemma \ref{lem:Form
symmetrization} that there exists a closed $2$-form $\rho$ representing the
class $a = PD(A)$, such that
$\tilde{\omega}_y = \frac{1}{y}\tilde{\omega}_{\epsilon} + \rho$ is
symplectic $\forall y>0$ and $\phi^*\tilde{\omega}_y = -\tilde{\omega}_y$. By
Corollary \ref{cor:Packing},
$(M,\omega)$ admits a symplectic packing by $k$ equal balls of radius
arbitrarily close to $\lambda$. Since this is true for every $\lambda$ that
satisfies Equation \ref{eq:lambda less than}, we have
\begin{equation*}
  \lambda_{\sup}^2 \geq \min \left\lbrace \frac{k d^{'2}_k}{2
\text{Vol}(M,\omega)},1 \right\rbrace,
\end{equation*}
and the proof is complete.
\end{proof}

Theorems \ref{thm:Relative packing rp2} and \ref{thm:Relative packing S1 x
S1} now follow immediately.

\begin{cor*}[Theorem \ref{thm:Relative packing rp2}]
For the pair $(\mathbb{C}P^{2},\mathbb{R}P^{2})$ with the standard
  symplectic form and real structure, the relative packing numbers
  $p_{\mathbb{R}P^2,k}$ are equal to the absolute packing numbers
for
  $\mathbb{C}P^{2}$.

\end{cor*}

\begin{cor*}[Theorem \ref{thm:Relative packing S1 x S1}]
Let $\phi:S^2 \to S^2$ be the reflection on $S^2$ which sends the
upper hemisphere to the lower one and fixes the equator.
The relative packing numbers for $(S^2\times S^2, \sigma_{S^2} \oplus
\sigma_{S^2}, \phi \oplus \phi)$ are equal to the absolute packing numbers for
$(S^2\times S^2), \sigma_{S^2} \oplus
\sigma_{S^2}).$
\end{cor*}
\subsection*{Acknowledgements}

I would like to thank Octav Cornea and Fran\c{c}ois Lalonde for their constant
encouragement and interesting discussions during this project, and the anonymous
reviewer for many helpful suggestions which greatly improved both the content
and presentation of this paper.

\bibliographystyle{amsplain}
\bibliography{/home/antonio/Mathematics/Bibtex/Equivariant,/home/antonio/Mathematics/Bibtex/Symplectic,/home/antonio/Mathematics/Bibtex/Algebraic_Geometry,/home/antonio/Mathematics/Bibtex/Topology}

\providecommand{\bysame}{\leavevmode\hbox to3em{\hrulefill}\thinspace}
\providecommand{\MR}{\relax\ifhmode\unskip\space\fi MR }
\providecommand{\MRhref}[2]{%
  \href{http://www.ams.org/mathscinet-getitem?mr=#1}{#2}
}
\providecommand{\href}[2]{#2}
\begin{thebibliography}{10}

\bibitem{Anjos_Lalonde_Pinsonnault_2009}
S{\'i}lvia Anjos, Fran\c{c}ois Lalonde, and Martin Pinsonnault, \emph{{The
  homotopy type of the space of symplectic balls in rational ruled
  4-manifolds}}, Geom. Topol. \textbf{13} (2009), no.~2, 1177--1227.
  \MR{MR2491660}

\bibitem{Barraud_Cornea_2007}
Jean-Fran\c{c}ois Barraud and Octav Cornea, \emph{{Lagrangian intersections and
  the {S}erre spectral sequence}}, Ann. of Math. (2) \textbf{166} (2007),
  no.~3, 657--722. \MR{MR2373371 (2008j:53149)}

\bibitem{Biran_1997}
Paul Biran, \emph{{Symplectic Packing in Dimension Four}}, Geom. Func. Anal.
  \textbf{7} (1997), no.~3, 420--437.

\bibitem{Biran_Cornea_2009}
Paul Biran and Octav Cornea, \emph{{Rigidity and Uniruling for Lagrangian
  Submanifolds}}, Geometry and Topology \textbf{13} (2009), 2881--2989.

\bibitem{Bredon_1972}
Glen Bredon, \emph{Introduction to compact transformation groups}, Academic
  Press, 1972.

\bibitem{Buhovsky_2010}
Lev Buhovsky, \emph{{A maximal relative symplectic packing construction}}, J.
  Symplectic Geom. \textbf{8} (2010), no.~1, 67--72. \MR{MR2609629}

\bibitem{Burckel_1979}
Robert~B. Burckel, \emph{An introduction to classical complex analysis. {V}ol.
  1}, Pure and Applied Mathematics, vol.~82, Academic Press Inc. [Harcourt
  Brace Jovanovich Publishers], New York, 1979. \MR{MR555733 (81d:30001)}

\bibitem{Cannas_da_Silva_2001}
Ana {Cannas da Silva}, \emph{{Lectures on symplectic geometry}}, {Lecture Notes
  in Mathematics}, vol. 1764, Springer-Verlag, Berlin, 2001. \MR{MR1853077
  (2002i:53105)}

\bibitem{Conway_1978}
John~B. Conway, \emph{Functions of one complex variable}, second ed., Graduate
  Texts in Mathematics, vol.~11, Springer-Verlag, New York, 1978. \MR{MR503901
  (80c:30003)}

\bibitem{Griffiths_Harris_1978}
Phillip Griffiths and Joseph Harris, \emph{Principles of algebraic geometry},
  Wiley Classics Library, John Wiley \& Sons Inc., New York, 1994, Reprint of
  the 1978 original. \MR{MR1288523 (95d:14001)}

\bibitem{Guillemin_Sternberg_1989}
V.~Guillemin and S.~Sternberg, \emph{{Birational equivalence in the symplectic
  category}}, Invent. Math. \textbf{97} (1989), no.~3, 485--522. \MR{MR1005004
  (90f:58060)}

\bibitem{Kawakubo_1991}
Katsuo Kawakubo, \emph{The theory of transformation groups}, Oxford University
  Press, 1991.

\bibitem{Lalonde_McDuff_1994}
Fran\c{c}ois Lalonde and Dusa McDuff, \emph{{{$J$}-curves and the
  classification of rational and ruled symplectic {$4$}-manifolds}}, {Contact
  and symplectic geometry ({C}ambridge, 1994)}, {Publ. Newton Inst.}, vol.~8,
  Cambridge Univ. Press, Cambridge, 1996, pp.~3--42. \MR{MR1432456 (98d:57045)}

\bibitem{Lalonde_McDuff_1996}
\bysame, \emph{{The classification of ruled symplectic {$4$}-manifolds}}, Math.
  Res. Lett. \textbf{3} (1996), no.~6, 769--778. \MR{MR1426534 (98b:57040)}

\bibitem{Lalonde_Pinsonnault_2004}
Francois Lalonde and Martin Pinsonnault, \emph{{The Topology of the Space of
  Symplectic Balls in Rational 4-Manifolds}}, Duke Math. J. \textbf{122}
  (2004), no.~2, 347--397.

\bibitem{Lerman_1995}
Eugene Lerman, \emph{{Symplectic cuts}}, Math. Res. Lett. \textbf{2} (1995),
  no.~3, 247--258. \MR{MR1338784 (96f:58062)}

\bibitem{McDuff_1990}
Dusa McDuff, \emph{{The Structure of Rational and Ruled Symplectic
  4-Manifolds}}, J. Am. Math. Soc. \textbf{3} (1990), no.~3, 679--712.

\bibitem{McDuff_1991_JDiffG}
\bysame, \emph{{The local behaviour of holomorphic curves in almost complex
  4-manifolds}}, Journal of Differential Geometry \textbf{34} (1991), 143--164.

\bibitem{McDuff_2001}
\bysame, \emph{{Symplectomorphism groups and almost complex structures}},
  {Essays on geometry and related topics, {V}ol. 1, 2}, {Monogr. Enseign.
  Math.}, vol.~38, Enseignement Math., Geneva, 2001, pp.~527--556. \MR{1929338
  (2003i:57042)}

\bibitem{McDuff_Polterovich_1994}
Dusa McDuff and Leonid Polterovich, \emph{{Symplectic Packings and Algebraic
  Geometry}}, Invent. Math. \textbf{115} (1994), no.~3, 405--434.

\bibitem{McDuff_Salamon_Intro_1998}
Dusa McDuff and Dietmar Salamon, \emph{{Introduction to symplectic topology}},
  second ed., {Oxford Mathematical Monographs}, The Clarendon Press Oxford
  University Press, New York, 1998. \MR{MR1698616 (2000g:53098)}

\bibitem{McDuff_Salamon_2004}
\bysame, \emph{{{$J$}-holomorphic curves and symplectic topology}}, {American
  Mathematical Society Colloquium Publications}, vol.~52, American Mathematical
  Society, Providence, RI, 2004. \MR{MR2045629 (2004m:53154)}

\bibitem{Munkres_2000}
James~R. Munkres, \emph{Topology}, Prentice-Hall, 2000.

\bibitem{Ortega_Ratiu_2004}
Juan-Pablo Ortega and Tudor~S. Ratiu, \emph{{Momentum Maps and {H}amiltonian
  Reduction}}, Birkhauser, 2004.

\bibitem{Pinsonnault_2008}
Martin Pinsonnault, \emph{{Symplectomorphism groups and embeddings of balls
  into rational ruled 4-manifolds}}, Compos. Math. \textbf{144} (2008), no.~3,
  787--810. \MR{MR2422351}

\bibitem{Pommerenke_1992}
Ch. Pommerenke, \emph{Boundary behaviour of conformal maps}, Grundlehren der
  Mathematischen Wissenschaften [Fundamental Principles of Mathematical
  Sciences], vol. 299, Springer-Verlag, Berlin, 1992. \MR{MR1217706
  (95b:30008)}

\bibitem{Schlenk_2005}
Felix Schlenk, \emph{{Packing Symplectic Manifolds by Hand}}, J. Symplectic
  Topology \textbf{3} (2005), no.~3, 313--340.

\bibitem{Taubes_1995}
Clifford~Henry Taubes, \emph{{The {S}eiberg-{W}itten and {G}romov invariants}},
  Math. Res. Lett. \textbf{2} (1995), no.~2, 221--238. \MR{1324704 (96a:57076)}

\bibitem{Wieck_thesis}
Ingo Wieck, \emph{{Explicit symplectic packings: symplectic tunnelling and new
  maximal constructions}}, Ph.D. thesis, Universit{\"a}t zu K{\"o}ln, 2008.

\end{thebibliography}

\end{document}